\documentclass[10pt]{amsart}

\usepackage{amsmath}
\usepackage{graphicx}

\usepackage{color}
\usepackage[dvipsnames]{xcolor}

\usepackage{amssymb}

\usepackage{euscript,color}

\definecolor{red}{rgb}{1,0,0}

\newtheorem{thm}{Theorem}[section]
\newtheorem{prop}{Proposition}
\newtheorem{lemma}[thm]{Lemma}

\theoremstyle{definition}
\newtheorem{defn}[thm]{Definition}

\theoremstyle{remark}
\newtheorem{remark}[thm]{Remark}

\numberwithin{equation}{section}

\def\C{\mathbb{C}}
\def\R{\mathbb{R}}

\def\Z{\mathbb{Z}}

\def\d{{\text{d}}}

\def\G{\mathfrak{G}}

\def\O{\mathrm{O}^{\uparrow}_+}

\def\HH{\mathcal{H}}

\def\N{\mathcal{N}}

\def\PP2{\mathcal{P}^2}
\def\G2{\mathrm{G}^+_2}
\def\G1{\mathrm{G}^+_1}
\def\NR1{\mathrm{N}^+_1}
\def\d{\mathrm{d}}
\begin{document}
\title[]{Closed $1/2$-Elasticae in the Hyperbolic Plane}

\author{Emilio Musso}
\address{(E. Musso) Dipartimento di Matematica, Politecnico di Torino,
Corso Duca degli Abruzzi 24, I-10129 Torino, Italy}
\email{emilio.musso@polito.it}

\author{\'Alvaro P\'ampano}
\address{(A. P\'ampano) Department of Mathematics and Statistics, Texas Techu University, Lubbock, TX, 79409, USA}
\email{alvaro.pampano@ttu.edu}

\thanks{Authors partially supported by PRIN 2014-2017 ``Variet\`a reali e complesse: geometria, to\-po\-lo\-gia e analisi ar\-mo\-ni\-ca'' and by the GNSAGA of INDAM. The second author has been partially supported by the AMS-Simons Travel Grants Program 2021-2022}


\date{\today}

\maketitle

\begin{abstract}
We study critical trajectories in the hyperbolic plane for the $1/2$-Bernoulli's bending energy with length constraint. Critical trajectories with periodic curvature are classified into three different types according to the causal character of their momentum. We prove that closed trajectories arise only when the momentum is a time-like vector. Indeed, for suitable values of the Lagrange multiplier encoding the conservation of the length during the variation, we show the existence of countably many closed trajectories with time-like momentum, which depend on a pair of relatively prime natural numbers.\\

\end{abstract}

\section{Introduction} 

The present paper is aimed to study the existence and the global properties of closed critical points of the functionals
\begin{equation}\label{1/2}
	\mathcal{B}_{\lambda}:\gamma\longmapsto\int_\gamma\left(\sqrt{|\kappa|}+\lambda\right),
\end{equation}
defined on convex curves $\gamma$ of the hyperbolic plane $\HH^2$. The constant $\lambda$ is a Lagrange multiplier encoding the conservation of the length during the variation and $\kappa$ is the geodesic curvature of the curve. This is the natural continuation of a previous work \cite{MP}, devoted to the analogous problem for curves in the plane $\R^2$ and in the sphere $\mathbb{S}^2$. In turn,  \cite{MP} extended the results obtained in  \cite{AGM,AGP,AGP1,Blaschke, LP} from the unconstrained ($\lambda=0$) to the constrained case ($\lambda\in\mathbb{R}$).   

The functional (\ref{1/2}) belongs to the family of Bernoulli's functionals 
$$\mathcal{B}_{p,\lambda}:\gamma \longmapsto\int_\gamma\left(|\kappa|^p+\lambda\right),$$
whose critical curves are known as \emph{p-elasticae}.  These functionals appeared for the first time in a letter that D. Bernoulli sent to L. Euler in 1738 (see \cite{T}, p. 172).  

The case $p=2$ has been extensively studied since then. In addition to the original motivations connected with the theory of elastic rods, $2$-elasticae have found several interesting applications in other areas, such as, in surface theory (Willmore surfaces \cite{LS2,Pin1,VDM}, constrained Willmore surfaces \cite{BPP,He,SCDPS}, and surfaces with spherical lines of curvature \cite{CPS}, among others) and in the Canham-Hefrich-Evans modeling \cite{ Can,Ev,Helf} of bilipid membranes \cite{CGR1,JMN,TY}. The existence and the geometric properties of closed $2$-elasticae have been treated in several works published since the 1980s (see for instance \cite{He,LS2,Y}).  

In a recent paper \cite{MY} Miura-Yoshizawa obtained a complete classification of $p$-elasticae in the plane $\mathbb{R}^2$, for every real number $p>1$. For $p\in(0,1)$, since the Lagrangian is merely continuous at the origin, the Bernoulli's functionals are defined for convex curves. Although these cases have been less studied, some particular cases are well known. The unconstrained (ie. $\lambda=0$) cases in $\R^2$ with $p=1/2$, and $p=1/3$, were considered by W. Blaschke (\cite{Blaschke}, Vol I, 1921, and Vol II, 1923, respectively) who showed that the critical curves are catenaries ($p=1/2$) or parabolas ($p=1/3$).  The case $p=1/3$ and $\lambda=0$ corresponds with the equi-affine length for convex curves. After the seminal paper \cite{COT}, equi-affine geometry of convex curves has been consistently used in recent studies on human curvilinear $2$-dimensional drawing movements and recognition for non-rigid planar shapes (see for instance \cite{FH, RK}, respectively, and the literature therein).

For natural values of $p>2$, unconstrained $p$-elasticae were previously studied in \cite{AGM}. They have been used to construct Willmore-Chen submanifolds in spaces with Riemannian and pseudo-Riemannian warped product metrics \cite{ABG,BFLM} and have been applied to analyze conformal tensions in string theories \cite{BFL}. In the case of smooth spherical curves, the only closed critical trajectories are geodesics. Surprisingly, smooth closed spherical unconstrained $p$-elasticae other than geodesics are $2$-elasticae or else $p\in(0,1)$ \cite{GPT}. In \cite{AGM,AGP1,GPT,MOP,MoP}, for $0<p<1$, the existence of infinitely many closed unconstrained $p$-elasticae in ${\mathbb S}^2$ was shown. When $p=(n-2)/(n+1)$, $n\in {\mathbb N}$, and $n>2$, unconstrained $p$-elasticae in Riemannian $2$-space forms arise in the theory of biconservative hypersurfaces as the generating curves of rotational ones \cite{MOP,MoP}. In particular, unconstrained $1/2$-elasticae have also been characterized as the generating curves of invariant minimal surfaces in Riemannian and Lorentzian $3$-space forms \cite{AGP}.

The case $p=1/2$ of the Bernoulli's functionals $\mathcal{B}_{p,\lambda}$ is also special for the following reason: after a suitable contact transformation, the invariant signatures \cite{COST,HO,KRV,MNJMIV} of the critical curves are the connected components of the smooth strata of singular elliptic curves. For this reason, $p=1/2$ can be considered to play the role of the classical case $p=2$ among the possible values of $p\in(0,1)$. As in the case $p=2$, their study can be made by resorting to elliptic functions and integrals. However, the  appearance of a singularity in the elliptic curves containing the signatures does not allow one to explicitly express the curvature in terms of elliptic functions as is the case when $p=2$ \cite{He,LS2}. 

In the plane $\mathbb{R}^2$ there are no closed $1/2$-elasticae \cite{LP,MP}, while all closed  $1/2$-elasticae in ${\mathbb S}^2$ were recently found in \cite{MP} (constrained case) and previously in \cite{AGM, AGP1} for the unconstrained case. Motivated by these results of Arroyo, Garay, Menc\'ia, Musso and P\'ampano about the existence of closed $1/2$-elasticae in ${\mathbb S}^2$, this paper aims to investigate constrained $1/2$-elasticae in the hyperbolic plane $\mathcal{H}^2$. Despite obvious formal analogies, the hyperbolic and spherical cases present substantial differences due to the different topologies of their isometry groups, $ \O(1,2)$ and  ${\rm SO}(3)$ respectively, and  to the different properties of the adjoint representations.  Another substantial difference is the following:
for every fixed $\lambda$,  the critical curves of $\mathcal{B}_{\lambda}$ are the flow lines of a vector field $\vec{X}_{\lambda}$ of the half-plane ${\mathbb H}=\{(x,y)\in \R^2\,/\, x>0\}$. In the spherical case $\vec{X}_{\lambda}$ possesses a unique stable equilibrium point (a center), while in the hyperbolic case $\vec{X}_{\lambda}$ may have, for suitable values of $\lambda$, a stable equilibrium point (a center) and an unstable equilibrium point (a saddle point). This causes difficulties in the numerical evaluations of those critical curves whose signatures are near the unstable critical point.

\section{$1/2$-Elasticae}

Let $\R^{1,2}$ be the Minkoswki 3-space  with the  inner product 
$$\vec{x}\cdot \vec{y}=(x^1,x^2,x^3)\cdot (y^1,y^2,y^3)=-x^1y^1+x^2y^2+x^3y^3\,,$$  
oriented by  the volume form $\d x^1\wedge \d x^2 \wedge \d x^3$,  time-oriented by the
causal cone $\N^+=\{\vec{x} \in \R^{1,2}\,/\, \vec{x}\cdot \vec{x}\le 0, \, x^1>0\}$, and equipped with the
vector cross-product $\times$ defined by $(\vec{x}\times \vec{y})\cdot \vec{w}={\rm det}(\vec{x},\vec{y},\vec{w})$ where ${\rm det}$ stands for the determinant.

The \emph{hyperbolic plane} is the space-like surface  
$$\HH ^2=\{\vec{x}\in \N^+\,/\,  \vec{x}\cdot \vec{x} = -1\}\,,$$ 
with the induced Riemannian metric 
$g_{\HH^2}= (d\vec{x}\cdot d \vec{x})|_{\HH^2}$ of constant sectional curvature $-1$.  We identify $\HH^2$ and the  unit disk ${\rm D}^2\subset \R^2$ with the  Poincar\'e metric by means of the isometry
\begin{equation*}\label{poincaremodel}
	\phi: \vec{x}=(x^1,x^2,x^3)\in \HH^2\longmapsto \frac{1}{1+x^1}(x^2,x^3)\in {\rm D}^2.
\end{equation*}

\begin{remark} \emph{Throughout this paper we will use the two models of the hyperbolic plane. In particular, for visualization purposes, the Poincar\'e disk will be used.}
\end{remark} 

Let $\gamma: {\rm I}\subseteq\R\to \HH^2$ be a smooth immersed curve where ${\rm I}\subseteq\mathbb{R}$ is its maximal interval of definition. The \emph{hyperbolic Frenet frame} along $\gamma$ is defined by 
$${\mathcal F}= (\gamma, \dot{\gamma},\gamma\times \dot{\gamma}):{\rm I}\subseteq\R \longmapsto \O(1,2)\,,$$
where $\O(1,2)$ is the restricted Lorentz group of $\mathbb{R}^{1,2}$ and the upper dot represents the derivative with respect to the hyperbolic line element $ds$ (ie. $d\gamma=\dot{\gamma}ds$). The hyperbolic Frenet frame satisfies the linear system
\begin{equation}\label{frenet}
	{\mathcal F}^{-1}\cdot d{\mathcal F} = \begin{pmatrix}
0& 1& 0\\
1& 0&-\kappa\\
0&\kappa&0
\end{pmatrix}ds\,,
\end{equation}
where $\kappa$ is the \emph{geodesic curvature} of $\gamma$. If $\kappa(t_o)=0$ for some value $t_o\in{\rm I}\subseteq\R$, we say that $\gamma(t_o)$ is a \emph{hyperbolic inflection point}.  Curves with no inflection points have a preferred orientation such that $\kappa>0$ everywhere. If $\dot{\kappa}(t_o)=0$ for some $t_o\in{\rm I}\subseteq\R$,  we say that $\gamma(t_o)$ is a \emph{hyperbolic vertex}. 

\begin{remark} \emph{We will distinguish between the curve $\gamma:{\rm I}\subseteq\R\to\HH^2$ and its image, denoted by $|[\gamma]|\subset\HH^2$, which will be referred to as the trajectory of $\gamma$.}
\end{remark}

Let $\gamma$ be a curve such that its curvature $\kappa$ is non-constant and periodic. The \emph{wavelength} of $\gamma$ is the least period $\omega>0$ of $\kappa$. The \emph{monodromy} of the curve is defined by ${\mathfrak m}_{\gamma} ={\mathcal F}|_{\omega}\cdot {\mathcal F}|_0^{-1}$, and it belongs to the stabilizer of the trajectory $|[\gamma]|$. The curve $\gamma$ is closed (ie. periodic) if and only if ${\mathfrak m}_{\gamma}$ has finite order ${\bf n}_{\gamma}\in {\mathbb N}$. We call ${\bf n}_{\gamma}$ the \emph{wave number}. 

\begin{remark} \emph{From the geometric point of view, the wave number is the order of the symmetry group of $\gamma$.} 
\end{remark}

If $\gamma$ is closed, then ${\rm I}=\mathbb{R}$ and $\omega_{\gamma}={\bf n}_{\gamma}\omega$ is the \emph{least period} of $\gamma$. A closed curve $\gamma$ induces an immersion $\boldsymbol\gamma : {\mathbb S}^1_{\omega_{\gamma}}=\R/\omega_{\gamma}\Z\to \HH^2$.  If  $\boldsymbol\gamma$ is injective,  we say that $\gamma$ is a \emph{simple closed curve}. 

For later use we recall here the hyperbolic analogue of the classical Four Vertex Theorem \cite{Kn}, proved by S. B. Jackson in \cite{Ja} (see \cite{Sg} for an alternative proof in the convex case):

\begin{thm}\label{fourvertices} A simple closed smooth curve in $\HH^2$ possesses at least four hyperbolic vertices.
\end{thm}

\begin{remark} \emph{The geometric fact underlying the hyperbolic Four Vertex Theorem is that the notion of a vertex of a curve in a $2$-space form is invariant with respect to the action of the pseudo-group of Moebius transformations. Therefore, if we take a curve in the unit disk, its Euclidean and hyperbolic vertices coincide.}
\end{remark}

We now introduce the variational problem and the class of critical curves under consideration. 

\begin{defn} A \emph{$1/2$-elastica} with multiplier $\lambda\in {\mathbb R}$ is a convex curve in $\HH^2$ of class ${\mathcal C}^4$ which is a critical point of the functional
$$\mathcal{B}_{\lambda} : \gamma\longmapsto \int_\gamma\left( \sqrt{|\kappa|}+\lambda\right),$$
with respect to compactly supported variations through convex curves.
\end{defn}

As suggested by this definition, from now on we will only work with convex curves, for which we define the following geometric invariant.

\begin{defn} The \emph{Blaschke invariant} of a convex curve $\gamma$ is
	$$\mu=\sqrt{\kappa}\,,$$
ie. the positive square root of the geodesic curvature of $\gamma$.
\end{defn}

In the following proposition we characterize $1/2$-elasticae with multiplier $\lambda\in\mathbb{R}$ in terms of a vector of $\mathbb{R}^{1,2}$, namely the momentum.

\begin{prop} Let $\gamma$ be a convex curve parameterized by the arc length.  Then,  $\gamma$ is a 1/2-elastica with multiplier $\lambda$
if and only if there exists  a  non-zero vector $\vec{\xi}\in \R^{1,2}$, called the momentum, such that
\begin{equation}\label{Noether}
	\frac{1}{2\mu}\gamma+\frac{\dot{\mu}}{2\mu^2}\dot{\gamma}-\left(\lambda+\frac{\mu}{2}\right)\gamma\times \dot{\gamma}=\vec{\xi}\,,
\end{equation}
where $\mu=\sqrt{\kappa}$ is the Blaschke invariant of $\gamma$.
\end{prop}
\begin{proof} 
Using a standard formula for the variational derivative of functionals depending on $\kappa$  \cite{AGM,AGP,LS2,P},  a convex curve is  critical for ${\mathcal B}_{\lambda}$ if and only if $\mu=\sqrt{\kappa}$ is a solution of the Euler-Lagrange equation
\begin{equation}\label{eulerlagrange}
\frac{\ddot{\mu}}{2\mu^2}=\frac{\dot{\mu}^2}{\mu^3}-\frac{1}{2\mu}-\lambda\mu^2-\frac{\mu^3}{2}\,.
\end{equation}
On the other hand, (\ref{frenet}) implies
$$\frac{d}{ds}\left(\frac{1}{2\mu}\gamma+\frac{\dot{\mu}}{2\mu^2}\dot{\gamma}-\left[\lambda+\frac{\mu}{2}\right]\gamma\times \dot{\gamma} \right)
=\left(\frac{\ddot{\mu}}{2\mu^2}-\frac{\dot{\mu}^2}{\mu^3}+\frac{1}{2\mu}+\lambda\mu^2+\frac{\mu^3}{2} \right)\dot{\gamma}\,,$$
which proves the result
\end{proof}

\begin{remark} \emph{The explicit expression of the momentum is found via a more conceptual approach, which is based on the construction of the Griffith's phase space and the analysis of the moment map for the Hamiltonian action of $\O(1,2)$ on the phase space \cite{GM,G,J,MP,OR}.}
\end{remark}

From (\ref{Noether}) we obtain the \emph{conservation law}
\begin{equation}\label{ode}
	\dot{\mu}^2=-\mu^2\left(\mu^4+4\lambda \mu^3+4\left[\lambda^2-  c\right]\mu^2-1\right),
\end{equation}
where $c=\vec{\xi}\cdot \vec{\xi}$. Note that the constant $c$ may have any real value.

In order to understand the solutions of \eqref{eulerlagrange} and \eqref{ode} we will analyze the orbit types of its phase portrait. 

\subsection{The Modified Invariant Signatures}\label{msignatures} 

We begin by defining the invariant and modified invariant signatures of an hyperbolic curve. Let $\gamma : {\rm I}\subseteq \mathbb{R}\to \HH^2$ be a smooth immersed curve with curvature $\kappa$.

\begin{defn} The (\emph{hyperbolic}) \emph{invariant signature}   of $\gamma$ is the set ${\mathfrak S}_{\gamma}\subset \R^2$  parameterized by $(\kappa,\dot{\kappa})$.  The \emph{modified invariant signature} of a convex curve $\gamma$  is the set  $\widehat{{\mathfrak S}}_{\gamma}\subset \R^2$ parameterized by  $(\mu,\dot{\mu})$, where $\mu =\sqrt{\kappa}$ is the Blaschke invariant.
\end{defn}

\begin{remark} \emph{Our definition of the invariant signature is the hyperbolic  analogue of the corresponding one for plane curves in Euclidean,  affine or projective geometries \cite{COST,  HO, KRV,MNJMIV}. We note that the signature of a convex curve is the image of the modified signature by the contact diffeomorphism ${\mathtt f}:(x,y)\in {\mathbb H}\to (x^2,2xy)\in  {\mathbb H}$ of the positive half-plane $ {\mathbb H}=\{(x,y)\in \R^2\,/\, x>0\}$.  If $\gamma$ is a closed convex curve with non-constant curvature and wave number  ${\bf n}_{\gamma}$,  the point $(\mu(s),\dot{\mu}(s))$, $s\in [0,\omega_{\gamma}]$ runs through the signature ${\bf n}_{\gamma}$-times.  Hence, the wave number is the degree of the branched covering $\gamma(s)\in |[\gamma]|\to (\kappa(s),\dot{\kappa}(s))\in {\mathfrak S}_{\gamma}$ or, in the convex case, the degree of $\gamma(s)\in |[\gamma]|\to (\mu(s),\dot{\mu}(s))\in \widehat{{\mathfrak S}}_{\gamma}$.}
\end{remark}

Let ${\mathbb H}$ be the half-plane $\{(x,y)\in \R^2\,/\, x>0\}$ and  $\vec{X}_{\lambda}\in {\mathfrak X}({\mathbb H})$ be the vector field (see Figure \ref{FIG08}, left)
\begin{equation}\label{phasevector}
	\vec{X}_{\lambda}|_{(x,y)}=y\,\partial_x+2\left(\frac{y^2}{x}-\frac{x}{2}-\lambda x^4-\frac{x^5}{2}\right)\partial_y\,.
\end{equation}
The integral curves of $\vec{X}_{\lambda}$ are the modified signatures of $1/2$-elasticae with multiplier $\lambda\in\mathbb{R}$. Hence, they are parameterized by
$${\bf m}:s\in {\rm I}\subseteq\mathbb{R}\longmapsto (\mu(s),\dot{\mu}(s)),$$
where $\mu$ is a positive solution of (\ref{eulerlagrange}) and ${\rm I}$ is its maximal interval of definition.   From (\ref{ode}) it follows that the modified signature of a $1/2$-elastica with momentum $\vec{\xi}$ is a smooth stratum of ${\mathcal C}^*_{\lambda,c}={\mathcal C}_{\lambda,c}\cap {\mathbb H}$,  where ${\mathcal C}_{\lambda,c}\subset \R^2$ is the singular elliptic curve $y^2+x^2Q_{\lambda,c}(x)=0$ and $Q_{\lambda,c}(x)$ is the quartic polynomial
\begin{equation}\label{quadratic} 
	Q_{\lambda,c}(x)=x^4+4\lambda x^3+4(\lambda^2- c)x^2-1\,.
\end{equation}
Non-constant periodic solutions of (\ref{eulerlagrange}) correspond to the smooth compact connected components of ${\mathcal C}_{\lambda,c}$ contained in ${\mathbb H}$. If $\lambda\ge -2/\sqrt[4]{27}$ there are no smooth compact $1$-dimensional phase curves of  $\vec{X}_{\lambda}$.  Thus, there are no critical curves with non-constant periodic curvature and, a fortiori, closed $1/2$-elasticae. Since our ultimate goal is to investigate closed $1/2$-elasticae, we discard this case. Thus, from now on we assume that $\lambda< -2/\sqrt[4]{27}$ holds. Then, the equilibrium points of 
$\vec{X}_{\lambda}$ are ${\bf m}_{\pm}(\lambda)=(\eta_{\pm}(\lambda),0)$,  where $0< \eta_{-}(\lambda)<\eta_{+}(\lambda)$ are the two real roots of the quartic polynomial
\begin{equation}\label{quarticP} 
	P_{\lambda}(x)=x^4+2\lambda x^3+ 1\,.
\end{equation}
The equilibrium point ${\bf m}_-(\lambda)$ is unstable (a saddle point) while ${\bf m}_+(\lambda)$ is a stable elliptic equilibrium point (a center). 

\begin{remark} \emph{The functions
\begin{equation}\label{eta}
	\eta_{\pm}:(-\infty,  -2/\sqrt[4]{27}\,)\longmapsto \eta_{\pm}(\lambda)\in \R^+\end{equation} 
are continuous and real-analytic on $(-\infty,  -2/\sqrt[4]{27})$.  The function $\eta_-$ is strictly increasing,  $\eta_-(\lambda)>1$  for every $ \lambda\in (-1, -2/\sqrt[4]{27})$, and $\eta_-(\lambda)\le 1$  for every $ \lambda\le -1$.  The function  $\eta_+$ is strictly decreasing. Using Ferrari's formula one can obtain the explicit expressions of the functions $\eta_\pm$. These expressions are rather complicated and, hence, we avoid writing them here, although they will be used in some computations later on.}
\end{remark}

\begin{figure}[h]
	\begin{center}
		\includegraphics[height=4cm,width=4cm]{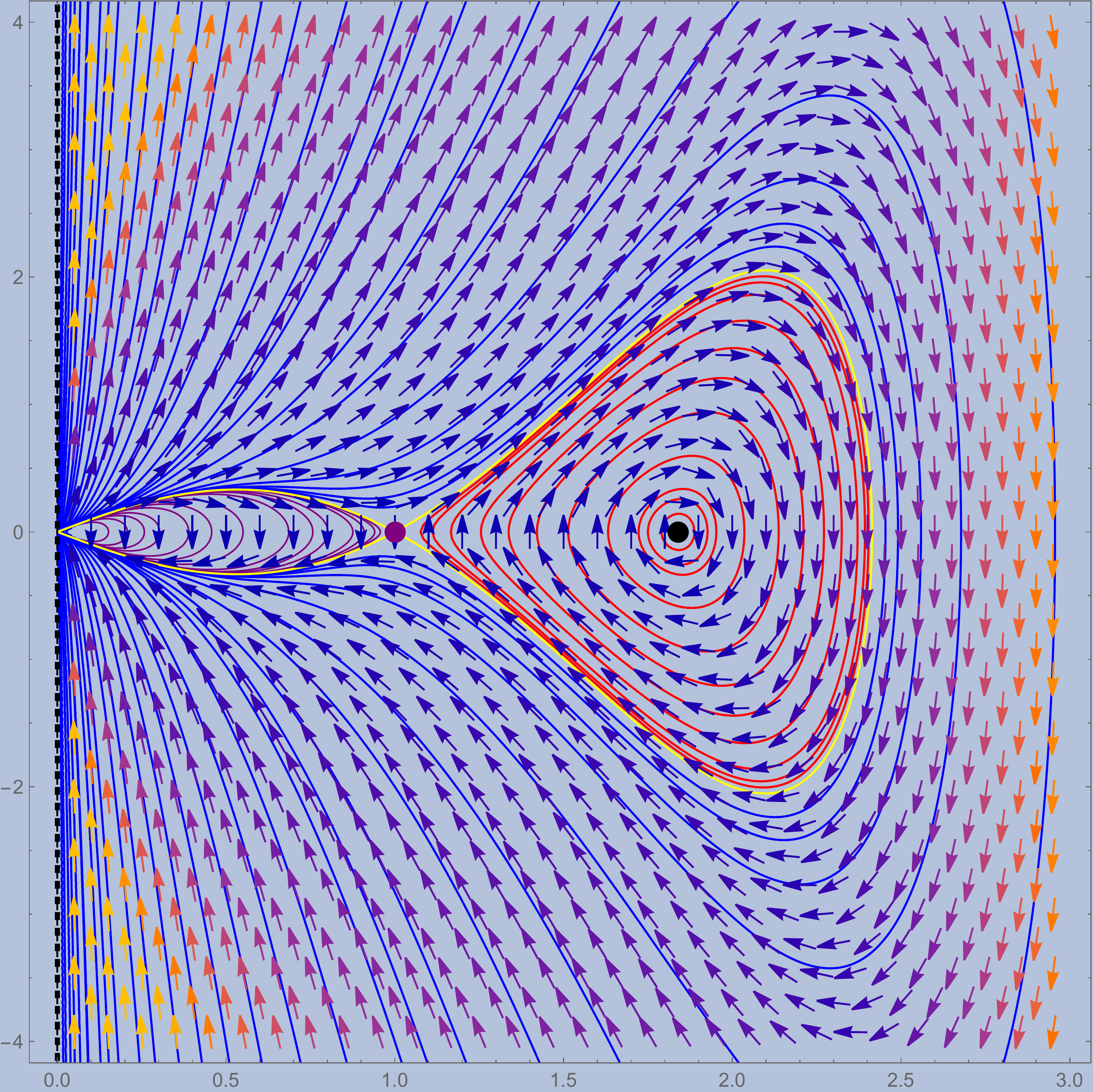}\quad\quad\quad\quad
		\includegraphics[height=4cm,width=4cm]{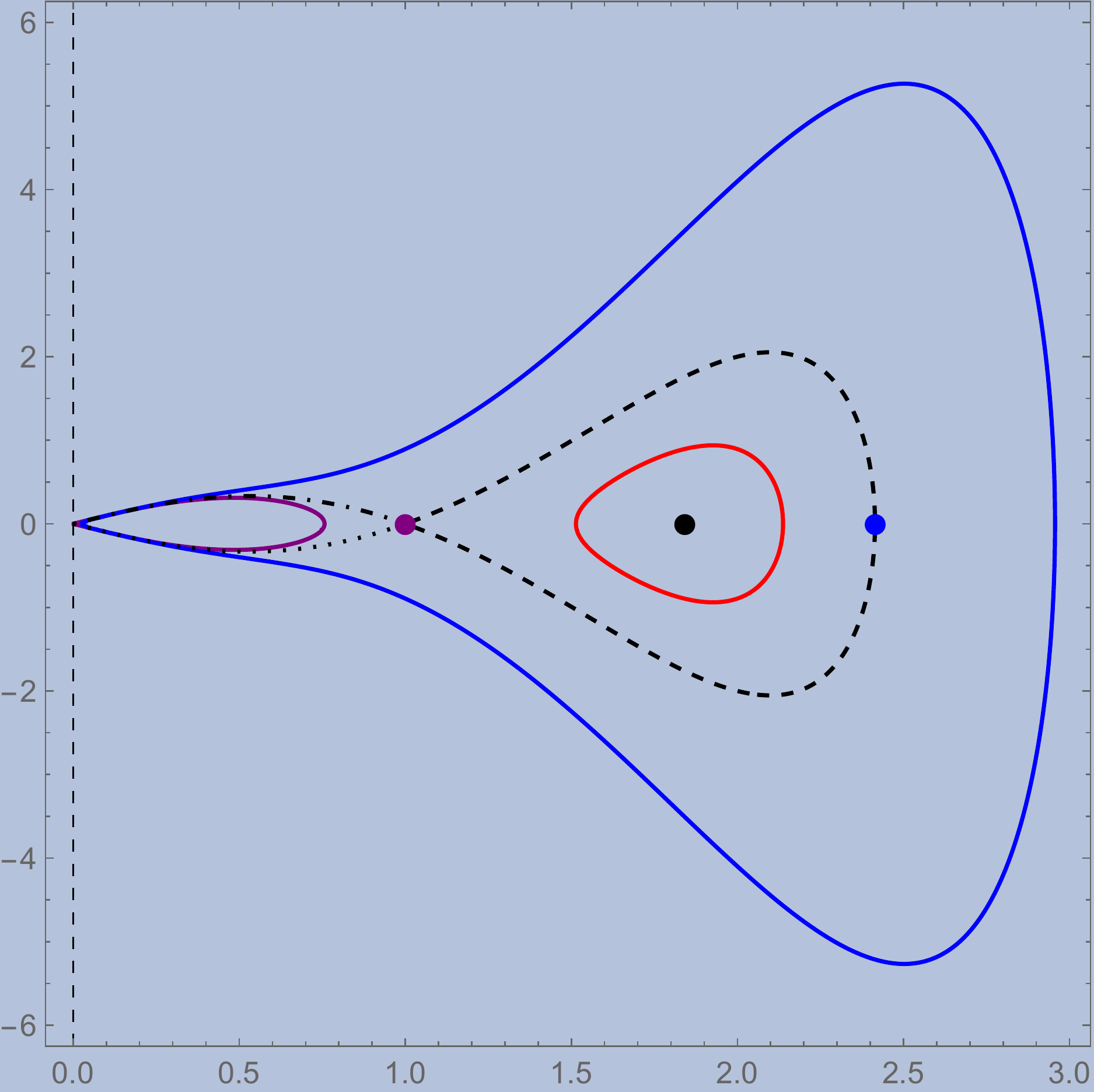}
		\caption{\small{Left: The plot of the vector field $\vec{X}_{-1}$ and its phase portrait. Right: The orbit types of $\vec{X}_{-1}$.
		}} \label{FIG08}
	\end{center}
\end{figure}
 
We now describe the orbit types of the phase portrait of the vector field $\vec{X}_{\lambda}$ (see Figure \ref{FIG08}, right): 
\begin{itemize}
\item The unstable equilibrium point ${\bf m}_-(\lambda)$ and the stable equilibrium point  ${\bf m}_+(\lambda)$. They correspond to the constant solutions 
$\mu =\eta_-(\lambda)$ and $\mu =\eta_+(\lambda)$ of (\ref{eulerlagrange}) (the purple and black points, respectively, in Figure \ref{FIG08}).
\item Closed trajectories (see the red curve depicted on Figure \ref{FIG08}). They correspond to the solutions of  
(\ref{eulerlagrange}) with initial conditions $\mu(0)\in (\eta_-(\lambda),\eta_+(\lambda))$ and $\dot{\mu}(0)=0$.
\item Non-closed phase curves of the first kind (see the blue curve depicted on Figure \ref{FIG08}). They correspond to the solutions of  
(\ref{eulerlagrange}) with initial conditions $\mu(0)\in (m^*_{\lambda},+\infty)$  and $\dot{\mu}(0)=0$  (the point $ (m^*_{\lambda},0)$ is represented in blue in Figure \ref{FIG08}). The origin is the limit point of these integral curves.
\item Non-closed phase curves of the second kind (see the purple curve depicted on Figure \ref{FIG08}). They correspond to the solutions of  
(\ref{eulerlagrange}) with initial conditions $\mu(0)\in (0,\eta_-(\lambda))$ and $\dot{\mu}(0)=0$. The origin is the limit point of these integral curves.
\item The exceptional (non-closed) phase curve of the first kind (the dashed-black curve reproduced in Figure \ref{FIG08}). It corresponds to the solution of (\ref{eulerlagrange}) with initial conditions $\mu(0)=m^*_{\lambda}$  and $\dot{\mu}(0)=0$. The unstable equilibrium is the limit point of this phase curve.
\item The two exceptional (non-closed) phase curves of the second kind (the dotted and dash-dotted black curves reproduced in Figure \ref{FIG08}).  They correspond to the solutions of (\ref{eulerlagrange}) with initial conditions 
$$\mu(0)= \frac{\eta_-(\lambda)}{2},\quad \dot{\mu}(0)=\pm  \frac{\eta_-(\lambda)}{2}\sqrt{-Q_{\lambda,c_+(\lambda)}\left( \frac{\eta_-(\lambda)}{2}\right)}\,.$$
They have two limit points: the origin and the unstable equilibrium point.
\end{itemize}

In the hyperbolic plane $\mathcal{H}^2$ there are three types of curves of positive constant curvature $\kappa$,  depending on whether $\kappa>1$ (elliptic type), $\kappa=1$ (parabolic type) or $\kappa<1$ (hyperbolic type). The only closed ones, namely circles, are those of elliptic type (ie. $\kappa>1$). We then have the following result.

\begin{prop}\label{constcurv}
	For closed $1/2$-elasticae with multiplier $\lambda$ and $\kappa>0$ constant, the following conclusions hold:
	\begin{itemize}
		\item If $\lambda>-2/\sqrt[4]{27}$ there are no $1/2$-elasticae with positive constant curvature and multiplier $\lambda$.
		\item If  $\lambda=-2/\sqrt[4]{27}$ there is one equivalence class of  closed $1/2$-elasticae with positive constant curvature and multiplier $\lambda$.  In this case $\kappa=\sqrt{3}$.
		\item  If $-1<\lambda<-2/\sqrt[4]{27}$ there are two distinct equivalence classes of closed $1/2$-elasticae with positive constant curvature  and multiplier $\lambda$. Their curvatures are $\eta_-(\lambda)^2$ and $\eta_+(\lambda)^2$, respectively.
		\item  If $\lambda \le -1$ there is one equivalence class of  closed $1/2$-elasticae with positive constant curvature $\eta_+(\lambda)^2$ and multiplier $\lambda$.
	\end{itemize}
\end{prop}
\begin{proof} If $\lambda>-2/\sqrt[4]{27}$ the vector field $\vec{X}_{\lambda}$ has no equilibrium points. This implies that (\ref{eulerlagrange}) does not possess constant solutions.  This proves the first assertion. If $\lambda=-2/\sqrt[4]{27}$,   the vector field $\vec{X}_{\lambda}$ possesses a unique equilibrium point, namely ${\bf m}_+(\lambda)={\bf m}_-(\lambda)=(\sqrt[4]{3},0)$. Since $\sqrt[4]{3}>1$, this implies the second assertion.   Suppose now that $\lambda<-2/\sqrt[4]{27}$.   The two critical curves of ${\mathcal B}_{\lambda}$ with constant curvature correspond to the two equilibrium points of the vector field $\vec{X}_{\lambda}$.   Thus,  their curvatures are $\eta_+(\lambda)^2$ and $\eta_-(\lambda)^2$, respectively.   Since $\eta_+(\lambda)>1$ for every $\lambda$, any curve with modified signature ${\bf m}_+(\lambda)$ is a closed circle.  Instead,  $\eta_-(\lambda)>1$ if and only if $\lambda\in  (-1,-2/\sqrt[4]{27})$.  This proves the third and the last assertions.
\end{proof}

\begin{remark} \emph{If $-1<\lambda<-2/\sqrt[4]{27}$,  the unstable equilibrium point of $\vec{X}_{\lambda}$ is the modified signature of a curve with constant curvature of elliptic type (a closed circle).  If $\lambda=-1$,  the unstable equilibrium point  is the modified signature of a curve with constant curvature of parabolic type (in the Poincar\'e disk, a circle minus an ``ideal'' point).  If $\lambda<-1$  the unstable equilibrium point  is the modified signature of a curve with constant curvature of hyperbolic type (a circular arc with two ``ideal'' points).}
\end{remark}

The study of closed $1/2$-elasticae with non-constant curvature can be subdivided into two parts: the analysis of the curves with non-constant periodic curvature and the investigation of the closure conditions for these curves.

\section{B-Curves}\label{B-curves} 

In this section we study $1/2$-elasticae with non-constant periodic curvature. For convenience, we introduce the following terminology.

\begin{defn} Convex curves whose Blaschke invariant $\mu$ is a non-constant periodic solution of (\ref{ode}) are called \emph{$B$-curves}.  A closed B-curve is a \emph{B-string}. 
\end{defn}

\begin{remark} \emph{Note that periodic solutions of (\ref{ode}) do exist if and only if the quartic polynomial $Q_{\lambda,c}$
has four distinct real roots, three positive and one negative,  denoted by
$e_1>e_2>e_3>0>e_4$, respectively. The equations  (\ref{eulerlagrange}), (\ref{ode}),  the multiplier $\lambda$, the constant $c$ and the roots $e_1>e_2>e_3>0>e_4$ of $Q_{\lambda,c}$ depend only on the equivalence class of the B-curve (see the definition below).}
\end{remark}

\begin{defn} Two immersed curves $\gamma$ and $\widetilde{\gamma}$ are \emph{equivalent to each other} if there exist ${\rm A}\in \O(1,2)$ and a smooth strictly increasing function $\psi:{\rm I}\subseteq\R\to \R$  such that $\widetilde{\gamma}(t)={\rm A}\gamma\circ \psi$, for all $t\in {\rm I}\subseteq\R$.  The \emph{equivalence class} of $\gamma$ is denoted by $\langle \gamma\rangle$. The \emph{moduli space} is the set of the equivalence classes.
\end{defn}

\begin{remark} \emph{If $\gamma$ and $\widetilde{\gamma}$ are equivalent to each other and parameterized by arc length the change of parameter $\psi$ is a translation of the independent variable (ie. $\psi(s)=s+s_o$,  where $s_o$ is a constant).}
\end{remark}

It turns out that the moduli space of B-curves can be described in terms of the multiplier $\lambda$ and the root $e_2$ of $Q_{\lambda,c}$.

\begin{prop}\label{ModSp} Let $\gamma$ be a B-curve. The map $\langle \gamma \rangle \to (\lambda(\langle \gamma \rangle),e_2(\langle \gamma \rangle))\in \R^2$ is a bijection 
onto the open domain ${\mathfrak P}=\{(\lambda,e_2)\in \R^2\,/\,   e_2>0,\, e_2^4+2\lambda e_2^3+1<0 \}$ of $\R^2$. 
\end{prop}

\begin{proof} The polynomial $Q_{\lambda,c}$ has four distinct roots  $e_1>e_2>e_3>0>e_4$  if and only if
\begin{equation}\label{rele1e2Lc}
	\begin{cases}
(i) &0<e_2<e_1\,,\\
(ii) &0<e_1^2e_2^3-2e_1-e_2\,,\\
(iii) &e_3=\frac{e_1+e_2+\sqrt{4e_1^3e_2^3+(e_1+e_2)^2}}{2e_1^2e_2^2}\,,\\ 
(iv) & e_4=\frac{-2e_1e_2}{e_1+e_2+\sqrt{4e_1^3e_2^3+(e_1+e_2)^2}}\,,\\
(v) &\lambda = -\frac{e_1^3e_2^2+e_1^2e_2^3+e_1+e_2}{4e_1^2e_2^2}\,,\\
(vi) &c=\frac{-2e_1^5 e_2^5 +e_1^6 e_2^4 + e_1^4 e_2^6 - 2(e_1^4e_2^2+e_1^2e_2^4)+(e_1+e_2)^2}{16e_1^4e_2^4}\,.
\end{cases}
\end{equation}
From  (i), (ii) and (v) it follows that $1+e_2^4+2\lambda e_2^3<0$ and $e_2>0$.   This implies that 
$(\lambda(\langle \gamma \rangle),e_2(\langle \gamma \rangle))$ belongs to ${\mathfrak P}$,  for every equivalence class $\langle\gamma\rangle$ of B-curves.
Conversely, let $(\lambda,e_2)$ be a point of ${\mathfrak P}$.  The inequalities $e_2>0$ and $1+e_2^4+2\lambda e_2^3<0$ imply that  the cubic polynomial
\begin{equation}\label{cb} 
	e_2^2x^3+\left(e_2^3+4e_2^2\lambda\right)x^2+x+e_2
\end{equation}
has a unique real root  $e_1$ strictly bigger than $e_2$.  Define $e_3,e_4$ and $c$ as in (\ref{rele1e2Lc}).  Then,  $e_1$ and $e_2$ satisfy $(ii)$ of (\ref{rele1e2Lc}) and $e_1>e_2>e_3>0>e_4$ are the four roots of $Q_{\lambda,c}$.  Let $\mu$ be the (periodic) solution of (\ref{ode}) such that $\mu(0)=e_2$ and $\gamma$ be a B-curve with $\kappa=\mu^2$.  Then,
$\langle\gamma\rangle\in {\mathfrak P}$ is the unique equivalence class of B-curves such that $\lambda(\langle \gamma \rangle) = \lambda$ and $
e_2(\langle \gamma \rangle)=e_2$. 
\end{proof}

\begin{remark} \emph{From the Cardano's formula\footnote{$\sqrt[n]{\zeta}$ is the determination of the n-th root of $\zeta\in \C$, with a branch cut discontinuity along the negative real axis such that  ${\mathfrak I}(\sqrt[n]{\zeta})>0$, if $\zeta$ belongs to the negative real axis.},  we have
\begin{equation}\label{e1}
	e_1=\frac{1}{3e_2}\left[(e_2+4\lambda)e_2 + {\mathfrak R}\left(\sqrt[3]{-8\left[{\mathtt a}+3\sqrt{3{\mathtt b}}\right]}\right) \right]
\end{equation}
where
\[\begin{split}
	{\mathtt a}&={\mathtt a}(\lambda,e_2)=e_2^6+12\lambda e_2^5+48\lambda^2 e_2^4+64\lambda^3 e_2^3+9e_2^2-18\lambda e_2\,,\\
{\mathtt b}&={\mathtt b}(\lambda,e_2)=e_2^8+12\lambda e_2^7+48\lambda^2 e_2^6+64\lambda^3 e_2^5+2e_2^4-20\lambda e_2^3+4\lambda^2 e_2^2+1\,.
\end{split}
\]
Using (\ref{e1}) and (iii), (iv), (vi) of (\ref{rele1e2Lc}), the roots $e_3$, $e_4$ and the constant $c$ can be expressed as real-analytic  functions of $\lambda$ and $e_2$.  As a rule, this dependence is implied.  If necessary, it will be explicitly indicated.}
\end{remark}

From now on the moduli space of B-curves is identified with ${\mathfrak P}$. For each ${\mathfrak p}=(\lambda,e_2)\in {\mathfrak P}$,  the polynomial $Q_{\lambda,c}$,  $c=c(\lambda,e_2)$, is denoted by $Q_{\mathfrak p}$.  Similarly,
$\mu_{\mathfrak p}$ stands for the solution of (\ref{eulerlagrange}) with initial conditions $\mu_{\frak p}(0)=e_2$, $\dot{\mu}_{\frak p}(0)=0$.  By construction,
$\mu_{\frak p}$ satisfies $\dot{\mu}_{\frak p}^2+Q_{\frak p}(\mu_{\frak p})=0$. 

We will next discuss an approach to obtain the Blaschke invariant $\mu_\mathfrak{p}$ and its associated B-curve $\gamma_\mathfrak{p}$. Let ${\mathfrak p}=(\lambda,e_2)\in {\mathfrak P}$, $e_1,e_3,e_4$ be the other three roots of $Q_{\mathfrak p}$ and
$\omega_{\mathfrak p}$ be the least period of $\mu_{\mathfrak p}$.   We call $\omega: {\mathfrak p}\in {\mathfrak P}\to \omega_{\mathfrak p}\in \R$ the \emph{wavelength function}.  From 256.12 and 340.04 of \cite{BF} we obtain
\begin{equation}\label{bcwlpp2} 
	\omega_{\mathfrak p}=2\int_{e_2}^{e_1} \frac{dx}{x\sqrt{-Q_{\mathfrak p}(x)}}=
	\frac{2{\mathtt g}}{e_1}\left(\frac{{\mathtt a}}{{\mathtt n}}{\rm K}({\mathtt m})-\frac{{\mathtt a}-{\mathtt n}}{{\mathtt n}}\Pi({\mathtt n},{\mathtt m})\right),
\end{equation}
where ${\rm K}$ and $\Pi$ are the complete elliptic integrals of the first and third kind respectively, and
\begin{equation}\label{ellipticmoduli}
	{\mathtt a}=\frac{e_2-e_1}{e_2-e_4},\,\, \,  {\mathtt m}=\frac{(e_1-e_2)(e_3-e_4)}{(e_1-e_3)(e_2-e_4)},\,\,\,   {\mathtt n}=\frac{e_4 {\mathtt a}}{e_1},\,\,\, {\mathtt g}=\frac{2}{\sqrt{(e_1-e_3)(e_2-e_4)}}.
\end{equation}
It then follows from \eqref{ellipticmoduli} that the wavelength depends in a real-analytic fashion on ${\mathfrak p}$. On the other hand, the function $\mu_{\mathfrak p}$ is strictly increasing on $[0,\omega_{{\mathfrak p}}/2]$.  Let $h_{\mathfrak p}:[e_2,e_1]\to [0,\omega_{\mathfrak p}/2]$ be the inverse of $\mu_{\mathfrak p}|_{[0,\omega_{\mathfrak p}/2]}$. From $\dot{\mu}_\mathfrak{p}^2=-\mu_\mathfrak{p}^2 Q_{\mathfrak p}(\mu_\mathfrak{p})$, we have
$$h_{\mathfrak p}(\mu)= \frac{\omega_{\mathfrak p}}{2} - \int^{e_1}_{\mu}\frac{dx}{x\sqrt{-Q_{\mathfrak p}(x)}}\,,\quad\quad\quad \mu \in [e_2,e_1].$$
This integral can be solved in terms of incomplete elliptic integrals and Jacobi's functions (\cite{BF},  $257.12$ and $340.04$).  As a result we obtain
\begin{equation}\label{hfunction}
	h_{\mathfrak p}(\mu)=\frac{\omega_{\mathfrak p}}{2}-\frac{{\mathtt g}}{e_1}\left(\frac{{\mathtt a}}{{\mathtt n}}{\mathtt u}(\mu)-\frac{{\mathtt a}-{\mathtt n}}{{\mathtt n}}\Pi({\mathtt n}, {\rm am}_{{\mathtt m}}({\mathtt u}(\mu)),{\mathtt m})\right),  
\end{equation}
where $\Pi({\mathtt n},-,{\mathtt m})$ is the incomplete elliptic integral of the third kind with parameters ${\mathtt n}$ and ${\mathtt m}$,
$ {\rm am}_{{\mathtt m}}$ is the Jacobi's amplitude with parameter ${\mathtt m}$, and
$${\mathtt u}(\mu)={\rm sn}^{-1}\left(\sqrt{\frac{(e_2-e_4)(e_1-\mu)}{(e_1-e_2)(\mu-e_4)}},{\mathtt m}\right).$$
Then, $\mu_{\mathfrak p}|_{[0,\omega_{\mathfrak p}/2]}=h_{\mathfrak p}^{-1}$.  Since $\mu_{\mathfrak p}$ is even,  this suffices to reconstruct $\mu_{\mathfrak p}$ on the whole real axis. The B-curves with curvature $\kappa_{\mathfrak p}=\mu_{\mathfrak p}^2$ can be numerically evaluated solving the linear system (\ref{frenet}), with appropriate initial conditions.  

\begin{remark} \emph{The algebraic curves $y^2+x^2Q_{\mathfrak p}(x)=0$ possess a singularity at the origin.  This is the geometric reason behind the fact that the Blaschke invariant $\mu_\mathfrak{p}$ cannot be expressed through elliptic functions as in the case of constrained and unconstrained elasticae of the hyperbolic plane \cite{He,LS2}. Therefore, a more practical procedure to build $\mu_{\mathfrak p}$ is to solve numerically the second-order ordinary differential equation \eqref{eulerlagrange} with initial conditions $\mu_\mathfrak{p}(0)=e_2$ and $\dot{\mu}_\mathfrak{p}(0)=0$.}
\end{remark}

There are three possible types of B-curves depending on the causal character of the momentum $\vec{\xi}$: either $\vec{\xi}$ is space-like or light-like or else time-like. Each of these cases carries different signs on the constant $c=\vec{\xi}\cdot\vec{\xi}$ of the conservation law \eqref{ode} and the associated B-curves present essentially different behaviors. Thus, we will distinguish between them.

\begin{defn} We say that a B-curve $\gamma$ is a \emph{BS-curve} (resp., \emph{BL-} or \emph{BT-curve}) if its momentum $\vec{\xi}$ is space-like (resp., light-like or time-like).
\end{defn}

It is also convenient to split the moduli space $\mathfrak{P}$ in three different subdomains depending on whether the B-curves associated with $\mathfrak{p}=(\lambda,e_2)$ are BS-, BL-, or BT-curves (see Figure \ref{FIG4}).

\begin{defn} The open subdomains ${\mathcal S}=\{(\lambda,e_2)\in {\mathfrak P}\, /\,  c(\lambda,e_2)>0\}$, ${\mathcal T}=\{(\lambda,e_2)\in {\mathfrak P}\, /\,   c(\lambda,e_2)<0 \}$ and the separating curve 
${\mathcal L}=\{(\lambda,e_2)\in {\mathfrak P}\, /\, c(\lambda,e_2)=0 \}$ are the \emph{moduli spaces of  BS-, BT- and BL-curves}, respectively.
\end{defn}

\begin{figure}[h]
	\begin{center}
		\includegraphics[height=4cm,width=6cm]{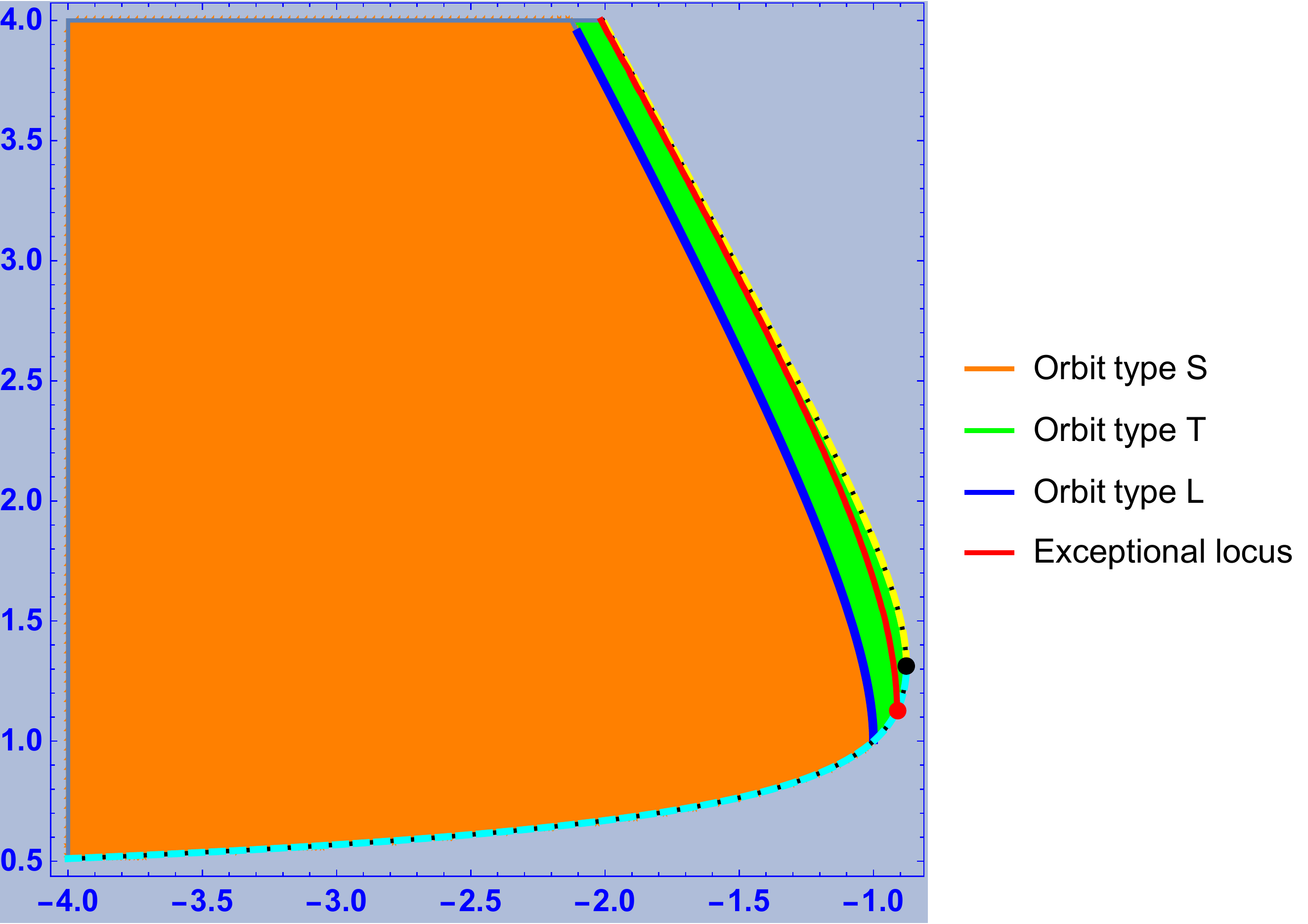}
		\caption{\small{The moduli space $\mathfrak{P}$. The blue curve is ${\mathcal L}$, the orange region is  ${\mathcal S}$ and the green region is ${\mathcal T}$. The red curve is the exceptional locus ${\mathcal E}$.    }}\label{FIG4}
	\end{center}
\end{figure}

\begin{remark}\label{lemma4Modulispaces} \emph{The moduli spaces ${\mathcal S}$, ${\mathcal L}$ and ${\mathcal T}$ can be characterized as follows:
	\begin{itemize}
		\item $(\lambda,e_2)\in {\mathcal S}\, \iff  (\lambda,e_2)\in {\mathfrak P}\,\,\,\, {\rm and}\,\,\,\, e_2^2+2\lambda e_2+1<0$\,,
		\item $(\lambda,e_2)\in {\mathcal L}\, \iff (\lambda,e_2)\in {\mathfrak P}\,\,\,\, {\rm and}\,\,\,\,  e_2^2+2\lambda e_2+1=0$\,,
		\item $(\lambda,e_2)\in {\mathcal T} \iff (\lambda,e_2)\in {\mathfrak P}\,\,\,\, {\rm and}\,\,\,\,  e_2^2+2\lambda e_2+1>0$\,.
	\end{itemize}
The curve ${\mathcal L}$ is also the graph of the function
\begin{equation}\label{br1}
	b_0:\lambda\in (-\infty,-1)\longmapsto -\lambda+\sqrt{\lambda^2-1}\,.
\end{equation}
The boundary $\partial{\mathfrak P}$ is the curve $1 + e_2^4 + 2 e_2^3 \lambda= 0$,  $\lambda<-2/\sqrt[4]{27}$. This curve is the union of the graphs $\partial_{\pm}{\mathfrak P}$ of the functions $\eta_{\pm}$ defined in (\ref{eta}) (the dotted curves colored in blue and yellow in Figure \ref{FIG4}).}
\end{remark} 

We now study each type of B-curves separately. For each case we will obtain explicit parameterizations by quadratures. The explicit expressions of the B-curves are found with a more conceptual approach based on the Marsden-Weinstein reduction method as applied to the Hamiltonian action of $\O(1,2)$ on the Griffith's phase space of the variational problem \cite{GM,G,J,MP,OR}.

\subsection{BL-Curves}\label{OTL}

Let ${\mathfrak p}\in {\mathcal L}$. Then $\mathfrak{p}$ is of the form  ${\mathfrak p}=(\lambda,  e_2(\lambda))$,    where $e_2(\lambda)=-\lambda+\sqrt{\lambda^2-1}$ and $\lambda <-1$. The B-curves associated with $\mathfrak{p}\in\mathcal{L}$ have constant $c=0$ in the conservation law \eqref{ode}. In other words, their momenta are light-like vectors.

\begin{thm}\label{TypeL} Let ${\mathfrak p}=(\lambda,e_2)\in {\mathcal L}$ and $\mu_{\mathfrak p}$ be the solution of (\ref{ode}) with initial condition $\mu_\mathfrak{p}(0)=e_2$ where $c=0$. Define $\Theta_{\mathfrak p}$ by
\begin{equation*}
	\Theta_\mathfrak{p}(s)=-\int _0^s \mu_{\mathfrak p}^2\left(\mu_{\mathfrak p}+2\lambda\right)ds\,.
\end{equation*}
Then,
\begin{equation}\label{eq3OTL}
	\gamma_{{\mathfrak p}}=\frac{1}{2\sqrt{2}\,\mu_{\mathfrak p}}\left(2\Theta_\mathfrak{p}^2+2\mu_{\mathfrak p}^2+1,2\sqrt{2}\,\Theta_\mathfrak{p},  2\Theta_\mathfrak{p}^2+2\mu_{\mathfrak p}^2-1\right)
\end{equation}
is a BL-curve with modulus ${\mathfrak p}$ and momentum  $\vec{\xi}=\frac{1}{\sqrt{2}}(1,0,1)$.
There are no BL-strings.
\end{thm}
\begin{proof} For simplicity, we will omit the subscript $\mathfrak{p}$ throughout this proof. From (\ref{eq3OTL}) it follows that $\gamma\cdot \gamma=-1$ and furthermore we have that
$$\dot{\gamma}\cdot\dot{\gamma}=\frac{\dot{\mu}^2+\mu^6+4\lambda \mu^5+4\lambda^2\mu^4}{\mu^2}.$$
Using (\ref{ode}) we obtain $\dot{\gamma}\cdot\dot{\gamma}=1$.  Then, $\gamma$ is parameterized by arc length.  Computing $\kappa=\ddot{\gamma}\cdot (\gamma\times \dot{\gamma})$ we find
$$\kappa=-(\mu+2\lambda)\ddot{\mu}+2(1+\lambda \mu^{-1})\dot{\mu}^2-\mu^6-6\lambda\mu^5-12\lambda^2\mu^4-8\lambda^3\mu^3.
$$
Then, from  (\ref{eulerlagrange}) and (\ref{ode}) we obtain $\kappa=\mu^2$.  This implies that $\gamma$ is a B-curve with parameters $(\lambda,e_2(\lambda))$. In addition, we have
$$\begin{cases}\gamma|_0=\frac{1}{2\sqrt{2}}	\left(\sqrt{\lambda^2-1}-3\lambda,0,3\sqrt{\lambda^2-1}-\lambda\right),\\
\dot{\gamma}|_0=\left(0,1,0\right),\\
\gamma|_0\times \dot{\gamma}|_0=\frac{1}{2\sqrt{2}}\left(3\sqrt{\lambda^2-1}-\lambda,0,\sqrt{\lambda^2-1}-3\lambda\right).
\end{cases}
$$
Keeping in mind that $\mu(0)=e_2(\lambda)$ and $\dot{\mu}(0)=0$,  we obtain
$$\vec{\xi}=\frac{1}{e_2(\lambda)}\gamma|_0-\left(\lambda+\frac{e_2(\lambda)}{2}\right)\gamma|_0\times \dot{\gamma}|_0=\frac{1}{\sqrt{2}}(1,0,1).$$
We now show that BL-strings do not exist. For $\gamma$ to be periodic, it follows from \eqref{eq3OTL} that $\Theta(\omega)$ needs to be zero where $\omega$ is the least period of $\kappa$, ie. the wavelength. In this case, the least period of $\gamma$ is also $\omega$. On $[0,\omega/2]$ we have from \eqref{ode}
$$ds=\frac{d\mu}{-\mu\sqrt{-(\mu^4+4\lambda \mu^3+4\lambda^2 \mu^2-1)}}\,.$$
Then,
\begin{eqnarray}\label{eq4OTL}
	\Theta(\omega)&=&-\int_0^\omega \mu^2(\mu+2\lambda)\,ds=-2\int_0^{\omega/2}\mu^2(\mu+2\lambda)\,ds\nonumber\\&=&2\int_{e_2}^{e_1}\frac{\mu\left(\mu+2\lambda\right)}{\sqrt{-\left(\mu^4+4\lambda \mu^3+4\lambda^2\mu^2-1\right)}}\,d\mu\,,
\end{eqnarray}
where 
$$e_1=-\lambda+\sqrt{\lambda^2+1}>e_2=-\lambda +\sqrt{\lambda^2-1}\,.$$
The right hand side of (\ref{eq4OTL}) is a standard elliptic integral that can be solved using 257.11, 336.01, 336.02 and 340.02 of \cite{BF}.   As a result we obtain
$$\Theta(\omega) =2 \frac{\sqrt{2}\left(2\lambda^4+2\lambda^2\sqrt{\lambda^2-1}\,\right)}{\left(\lambda^2+\sqrt{\lambda^4-1}\,\right)^{3/2}}\left({\rm E}({\mathtt m}_{\lambda})-{\rm K}({\mathtt m}_{\lambda})\right),
$$
where ${\rm E}$ is the complete elliptic integral of the second kind and
$${\mathtt m}_{\lambda}=\frac{\lambda^2-\sqrt{\lambda^4-1}}{\lambda^2+\sqrt{\lambda^4-1}}\in (0,1)\,.$$
Since ${\rm E}(m)<{\rm K}(m)$ for every $m\in (0,1)$, it follows that
$$\Theta(\omega)<0\,,$$
and so there are no BL-strings. 
\end{proof}

\begin{defn} The BL-curve $\gamma_\mathfrak{p}$ given by \eqref{eq3OTL} is referred to as the \emph{standard BL-curve} with modulus $\mathfrak{p}\in\mathcal{L}$.
\end{defn}

\begin{remark} \emph{Every BL-curve with modulus  ${\mathfrak p}\in {\mathcal L}$ is equivalent to $\gamma_{{\mathfrak p}}$, \eqref{eq3OTL}.  From now on we implicitly assume that the BL-curves in consideration are in their standard form.}
\end{remark}

Let ${\mathfrak p}=(\lambda,-\lambda+\sqrt{\lambda^2-1})\in {\mathcal L}$ and $\gamma_\mathfrak{p}$ be the standard BL-curve with modulus $\mathfrak p$ defined as in (\ref{eq3OTL}). Adopting the Poincar\'e model of the hyperbolic plane, the curve $\gamma_\mathfrak{p}$ is parameterized by
$$\gamma_\mathfrak{p}=\frac{1}{2\Theta^2_\mathfrak{p}+\left(1+\sqrt{2}\,\mu_\mathfrak{p}\right)^2}\left(2\sqrt{2}\,\Theta_\mathfrak{p}, 2\Theta_\mathfrak{p}^2+2\mu_\mathfrak{p}^2-1\right).
$$
The stabilizer of the momentum is the parabolic subgroup
$${\rm P}=\left\{{\rm HP}(t) =  \begin{pmatrix}
1+\frac{t^2}{2}&t&-\frac{t^2}{2}\\
t&1&-t\\
\frac{t^2}{2}&t&1-\frac{t^2}{2}
\end{pmatrix}\,/\, t\in \R \right\}.$$
The monodromy  ${\mathfrak m}$ is the non-trivial element ${\rm HP}(\sqrt{2}\,\Theta_\mathfrak{p}(\omega_\mathfrak{p}))$ of ${\rm P}$.   Let
$\Gamma=\gamma_\mathfrak{p}([0,\omega_\mathfrak{p}])$ be the fundamental arc,  then $|[\gamma_\mathfrak{p}]|=\cup_{n\in \Z}{\mathfrak m}^n(\Gamma)$.  Since the curvature is an even function,  
$|[\gamma_\mathfrak{p}]|$ is invariant by the reflection with respect to the Oy-axis. The segment $\{(0,v)\,/\, -1<v<1\}$ is a slice for the action of ${\rm P}$ on ${\rm D}^2$.  The orbit ${\mathcal O}_{v}$  through $(0,v)$  is the intersection of  ${\rm D}^2$ with the circle passing through $(0,v)$, ${\rm p}=(0,1)$ and tangent to $\partial {\rm D}^2$ at  ${\rm p}$.  Denote by $\mathcal O^-$ and by $\mathcal O^+$ the orbits of ${\rm P}$  through $\gamma_\mathfrak{p}(0)$ and $\gamma_\mathfrak{p}(\omega_\mathfrak{p}/2)$ respectively,  referred to as the \emph{lower and upper osculating circles} of $\gamma_\mathfrak{p}$.   The trajectory $|[\gamma_\mathfrak{p}]|$ is contained in the lunular region of ${\rm D}^2$ bounded by $\mathcal O^+$ and $\mathcal O^-$;  is tangent to $\mathcal O^-$ at $\gamma_\mathfrak{p}(n\omega_\mathfrak{p})$, $n\in \Z$,  and to  $\mathcal O^+$ at $\gamma_\mathfrak{p}(\omega_\mathfrak{p}/2+n\omega_\mathfrak{p})$, $n\in \Z$,  with limit point
$$\lim_{s\to +\infty} \gamma_\mathfrak{p}(s) = \lim_{s\to- \infty} \gamma_\mathfrak{p}(s) = {\rm p}\,.$$
When $\lambda\to -1^-$,   $\mathcal O^-$ and $|[\gamma_\mathfrak{p}]|$ tend to ${\mathcal O}_{2-\sqrt{2}}$. When $\lambda$ decreases, $|[\gamma_\mathfrak{p}]|$ and   $\mathcal O^-$   inflect (ie.  their radii shrink)  and,   when $\lambda\to -\infty$,  tend to ${\rm p}$.  
Figure \ref{FIG12} reproduces the trajectories of BL-curves with $\lambda=-1.01, -1.17,-1.3$  and $-2$, respectively (all numerical values are rounded up to a maximum of two decimals).

\begin{figure}[h]
		\makebox[\textwidth][c]{
		\includegraphics[height=4cm,width=4cm]{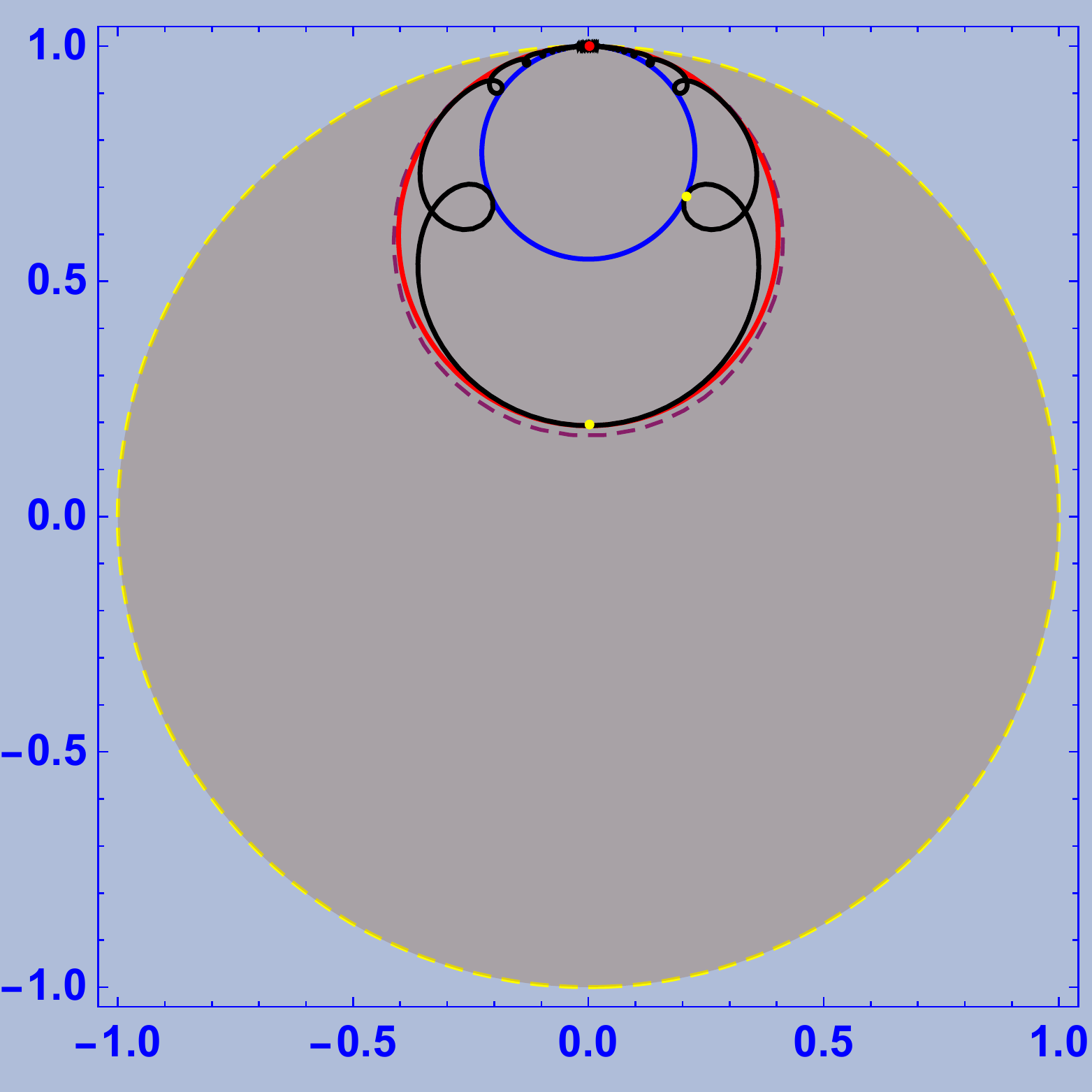}\quad
		\includegraphics[height=4cm,width=4cm]{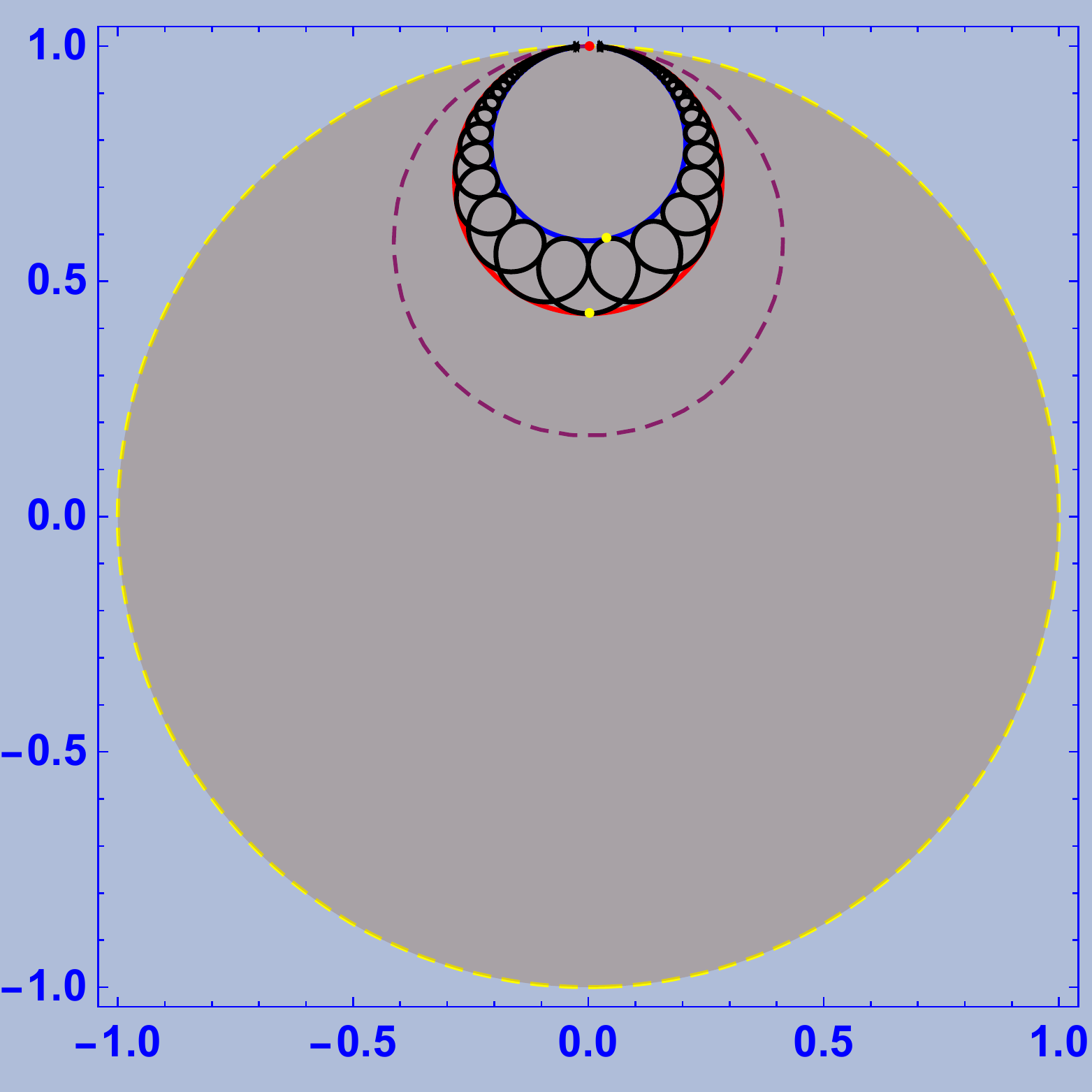}\quad
		\includegraphics[height=4cm,width=4cm]{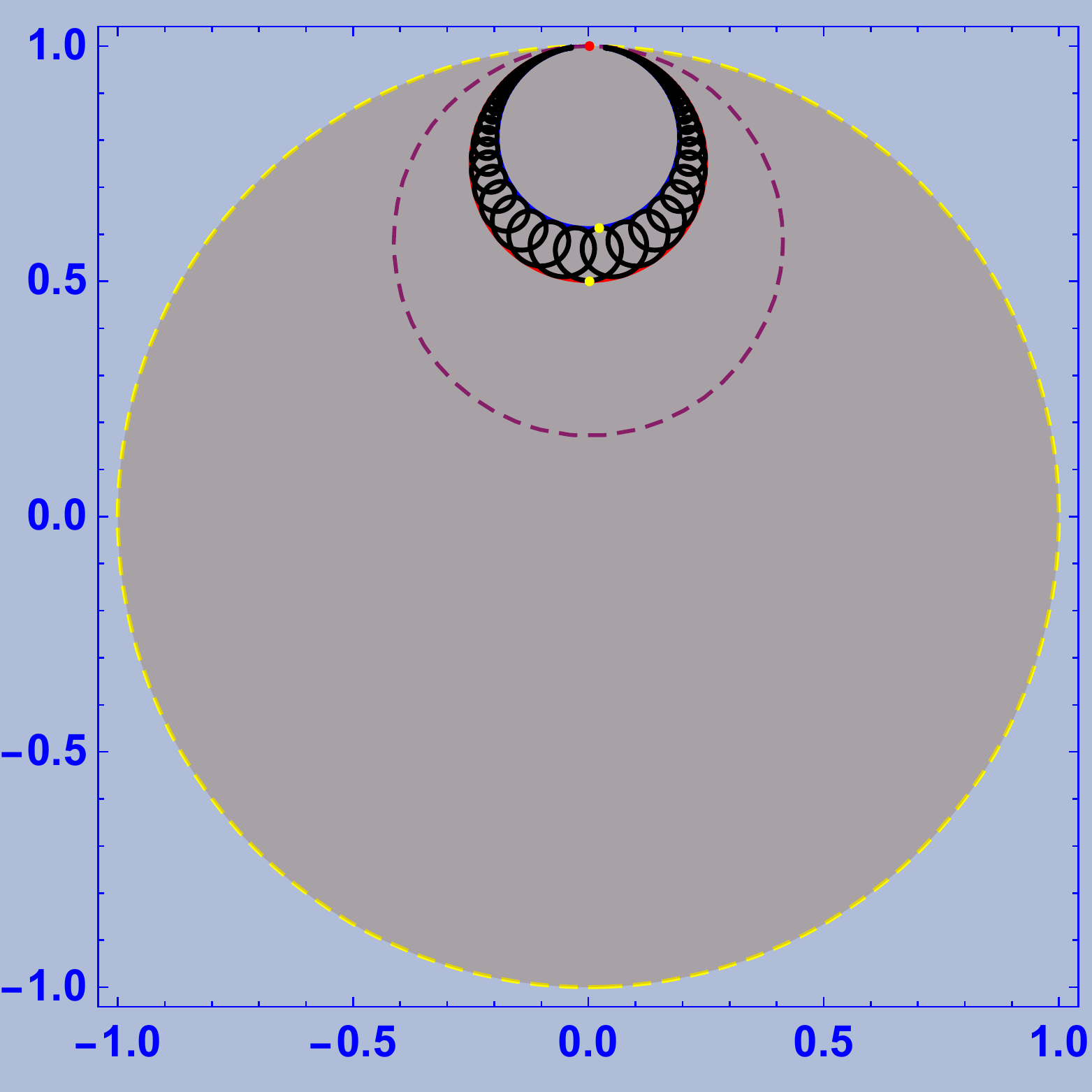}\quad
		\includegraphics[height=4cm,width=4cm]{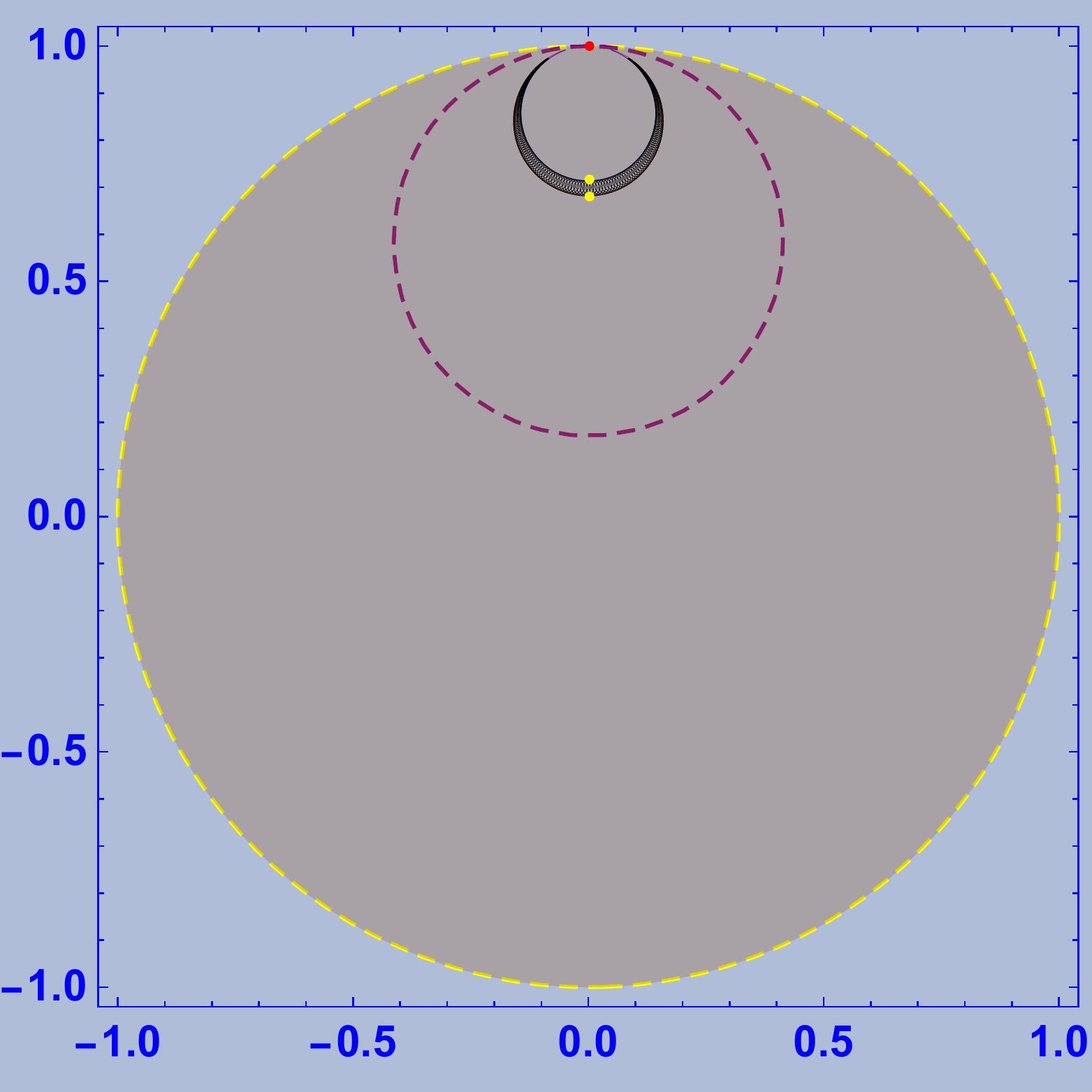}
	}
		\caption{\small{Trajectories of BL-curves with decreasing multiplier $\lambda$. The red and blue curves are the osculating circles. The dashed purple curve is the limit circle  ${\mathcal O}_{2-\sqrt{2}}$.
		}}\label{FIG12}
\end{figure}

\subsection{BS-Curves}\label{OTS}

Let $\mathfrak{p}=(\lambda,e_2)\in\mathcal{S}$. The corresponding B-curves have space-like momentum and $c>0$ holds in the conservation law \eqref{ode}.

\begin{thm}\label{TypeS} Let ${\mathfrak p}=(\lambda,e_2)\in {\mathcal S}$ and $\mu_{{\mathfrak p}}$ be the solution of \eqref{ode} with initial condition $\mu_\mathfrak{p}(0)=e_2$, where $c>0$ is as in $(vi)$ of (\ref{rele1e2Lc}). Define $\Theta_\mathfrak{p}$ by
$$\Theta_\mathfrak{p}(s) = 2\sqrt{c\,} \int_0^s \frac{\mu_{\mathfrak p}^2(\mu_{\mathfrak p}+2\lambda)}{1+4c\mu_{\mathfrak p}^2}ds\,.$$
Then,
\begin{equation}\label{typeS}
	\gamma_{{\mathfrak p}}= \frac{1}{2\sqrt{c}\,\mu_{\mathfrak p}}\left(\sqrt{1+4c\mu_{\mathfrak p}^2\,}\,\cosh(\Theta_\mathfrak{p}), \sqrt{1+4c\mu_{\mathfrak p}^2\,}\,\sinh(\Theta_\mathfrak{p}),1\right)
\end{equation}
is a BS-curve with modulus ${\mathfrak p}$ and momentum $\vec{\xi}=(0,0,-\sqrt{c\,})$. There are no BS-strings.
\end{thm}
\begin{proof} The first part of the proof is analogous to the first part of Theorem \ref{TypeL}. For the sake of brevity, we omit it here.
	
We prove that there are no BS-strings. As customary in our proofs, we avoid explicitly writing the dependence upon $\mathfrak{p}$. By contradiction suppose that $\gamma$ is periodic.  From (\ref{typeS}) it follows that $\Theta(\omega)=0$ and that 
$\omega$ is the least period of the curvature $\kappa$, ie. the wavelength. In this case, this implies that $\omega$ is also the least period of $\gamma$.  Since
$$f(s)=\frac{\mu^2(\mu+2\lambda)}{1+4c\mu^2}$$
is even and periodic with period $\omega$,  the function $\Theta$ is odd,  periodic with period $\omega$, and such that
$$\Theta(\omega)= 2\sqrt{c} \int_0^{\omega} f(s)ds = 2\sqrt{c} \int_{-\omega/2}^{\omega/2} f(s)ds =4\sqrt{c} \int_{0}^{\omega/2} f(s)ds = 2\Theta(\omega/2)\,.$$
Then, $\Theta(\omega/2)=0$.  The function $\mu$ is strictly increasing on $(0,\omega/2)$ and strictly decreasing on $(\omega/2,\omega)$.  Then, 
$\Gamma_1=\gamma([0,\omega/2])$ and $\Gamma_2=\gamma([\omega/2,\omega])$ are simple closed arcs with distinct boundary points 
\[\begin{split}\vec{p}_0&=\gamma(0)=  \frac{1}{2\sqrt{2}\,e_2}\left(\sqrt{1+4c e_2^2},0,1\right),\\
 \vec{p}_2&= \gamma(\omega/2)=  \frac{1}{2\sqrt{2}\,e_1}\left(\sqrt{1+4 c e_1^2},0,1\right),\end{split}\]
such that $\gamma(\R)=\Gamma_1\cup \Gamma_2$. Taking into account that  $\mu(0)=e_2<-2\lambda<e_1=\mu(\omega/2)$ and recalling that  $\mu$ is strictly increasing on $(0,\omega/2)$ and strictly decreasing on $(\omega/2,\omega)$,  the equation $\mu+2\lambda=0$ possesses exactly one solution $s'$ in the interval $(0,\omega/2)$ and exactly one solution $s''$ in the interval $(\omega/2,\omega)$.  Thus,
$$\begin{cases} \mu(s)+2\lambda < 0,\quad  s\in [0,s'),\\
\mu(s)+2\lambda > 0,\quad  s\in (s',s''),\\
 \mu(s)+2\lambda < 0,\quad  s\in (s'',\omega]\,.
\end{cases}
$$
This implies that $\Theta$ is strictly decreasing on $(0,s')$, strictly increasing on $(s',s'')$ and strictly decreasing on $(s'',\omega)$.  Since $\Theta(0)=\Theta(\omega/2)=\Theta(\omega)=0$,  we have
$$\Theta(s)<0,\,\,  \forall s\in (0,\omega/2),\quad \Theta(s)>0,\,\,  \forall s\in (\omega/2,\omega).$$
Therefore,  
$$\Gamma_1\setminus\{\vec{p}_1,\vec{p}_2\}\subset \{\vec{x}\in \R^{1,2}\,/\, x^2<0\},\quad \Gamma_2\setminus\{\vec{p}_1,\vec{p}_2\}\subset \{\vec{x}\in \R^{1,2}\,/\, x^2>0\}.$$  
Hence $\gamma$ is a simple closed curve with two vertices, namely, $\gamma(0)$ and $\gamma(\omega/2)$. This conclusion contradicts the Four Vertex Theorem (Theorem \ref{fourvertices}). Consequently, $\gamma$ cannot be closed.
\end{proof}

\begin{defn} The BS-curve $\gamma_\mathfrak{p}$ given by \eqref{typeS} is referred to as the \emph{standard BS-curve} with modulus $\mathfrak{p}\in\mathcal{S}$.
\end{defn}

\begin{remark} \emph{Every BS-curve with modulus  ${\mathfrak p}\in {\mathcal S}$ is equivalent to $\gamma_{{\mathfrak p}}$, \eqref{typeS}. From now on we implicitly assume that the BS-curves in consideration are in their standard form.}
\end{remark}

Let ${\mathfrak p}=(\lambda,e_2)\in {\mathcal S}$ and $\gamma_\mathfrak{p}$ be the standard BS-curve with modulus $\mathfrak p$. Resorting to the  Poincar\'e model, the parameterization of $\gamma_\mathfrak{p}$ is
$$\gamma_\mathfrak{p}=\frac{1}{2\sqrt{c}\,\mu_\mathfrak{p}+\sqrt{1+4c\mu_\mathfrak{p}^2}\,\cosh(\Theta_\mathfrak{p})}\left(\sqrt{1+4 c \mu_\mathfrak{p}^2}\,\sinh(\Theta_\mathfrak{p}),1 \right).
$$
The stabilizer of the momentum is the subgroup
$$\O(1,1)=\left\{{\rm HR}(t) =  \begin{pmatrix}
\cosh(t)&\sinh(t)&0\\
\sinh(t)&\cosh(t)&0\\
0&0&1
\end{pmatrix}\,/\, t\in \R \right\}\,.$$
The monodromy  ${\mathfrak m}$ is the non-trivial element ${\rm HR}(\Theta_\mathfrak{p}(\omega_\mathfrak{p}))$ of $\O(1,1)$.   Let
$\Gamma=\gamma_\mathfrak{p}([0,\omega_\mathfrak{p}])$ be the fundamental arc,  then $|[\gamma_\mathfrak{p}]|=\cup_{n\in \Z}{\mathfrak m}^n(\Gamma)$.  Since the curvature is an even function,  
$|[\gamma_\mathfrak{p}]|$ and $\Gamma$ are invariant by the reflection with respect to the Oy-axis. The segment $\{(0,v)\,/\, -1<v<1\}$ is a slice for the action of $\O(1,1)$ on ${\rm D}^2$.  The orbit ${\mathcal O}_{v}$  through $(0,v)$  is the intersection of  ${\rm D}^2$ with the circle passing through $(0,v)$, ${\rm p}_{+}=(1,0)$ and ${\rm p}_{-}=(-1,0)$.  Denote by $\mathcal O^-$ and by $\mathcal O^+$ the orbits of $\O(1,1)$  through $\gamma_\mathfrak{p}(0)$ and $\gamma_\mathfrak{p}(\omega_\mathfrak{p}/2)$ respectively,  referred to as the \emph{lower and upper osculating circular arcs} of $\gamma_\mathfrak{p}$.   The trajectory $|[\gamma_\mathfrak{p}]|$ is contained in the lunular region of ${\rm D}^2$ bounded by $\mathcal O^+$ and $\mathcal O^-$;  is tangent to $\mathcal O^+$ at $\gamma_\mathfrak{p}(n\omega_\mathfrak{p})$, $n\in \Z$,  and to  $\mathcal O^-$ at $\gamma_\mathfrak{p}(\omega_\mathfrak{p}/2+n\omega_\mathfrak{p})$, $n\in \Z$,  with limit points
$$\lim_{s\to +\infty} \gamma_\mathfrak{p}(s) = {\rm p}_{+}\,,\quad \quad \lim_{s\to- \infty} \gamma_\mathfrak{p}(s) = {\rm p}_{-}\,.$$

Choose $\lambda\in (-\infty,-1)$.  Put ${\rm L}_{\lambda}=(\eta_-(\lambda),b_0(\lambda))$ and consider the $1$-parameter family 
$\{\gamma_\mathfrak{p}\}_{e_2\in {\rm L}_{\lambda}}$ of BS-curves with multiplier $\lambda$. Recall that the graph of the function $\eta_-$ \eqref{eta} is the lower boundary of the moduli space $\mathcal{S}$, while the graph of $b_0$ \eqref{br1} is the upper boundary. The function
$$\upsilon_{\lambda}^+:e_2\in {\rm L}_{\lambda}\longmapsto   \frac{1}{2\sqrt{c\,}e_2+\sqrt{1+4ce_2^2}}\in \R\,,$$ 
where $c=c(\lambda,e_2)>0$ is as in (vi) of \eqref{rele1e2Lc}, is convex, attains the minimum at $e_2=-\lambda$, and satisfies
 $$\lim_{e_2\to \eta_-(\lambda)^+} \upsilon_{\lambda}(e_2)=\upsilon_{\lambda}^*\,,\quad\quad  \lim_{e_2\to b_0(\lambda)^+} \upsilon_{\lambda}(e_2)=1\,,$$
where
$$
\upsilon_{\lambda}^*= \frac{1}{2\sqrt{c(\lambda,\eta_-(\lambda))}\,\eta_-(\lambda)+\sqrt{1+4c(\lambda,\eta_-(\lambda))\eta_-(\lambda)^2}}<1\,.
$$
Let  $e^*_2(\lambda)$ be the unique element of ${\rm L}_{\lambda}$  such that $\upsilon^+_{\lambda}(e^*_{2}(\lambda))=\upsilon_{\lambda}^*$.

We are now in a position to describe the main features of the kinematics of the $1$-parameter family $\{\gamma_\mathfrak{p}\}_{e_2\in {\rm L}_{\lambda}}$. When $e_2$ varies in the interval ${\rm L}_{\lambda}$ the curves of the family tend to two asymptotic positions and evolve in five intermediate stages with contracting and expanding phases:
\begin{itemize}
\item When $e_2\to \eta_-(\lambda)^{+}$,  the osculating arc ${\mathcal O}^+$  and $|[\gamma_\mathfrak{p}]|$ tend to the asymptotic position   ${\mathcal O}_{\upsilon^*_{\lambda}}$.
\item When $e_2\in (\eta_-(\lambda), -\lambda)$,  the radius of circular arc ${\mathcal O}^+$ increases (ie.  ${\mathcal O}^+$ and $|[\gamma_\mathfrak{p}]|$ deflect). The evolution of the curve is in a contracting phase.
\item When $e_2=-\lambda$  the radius of the osculating circular arc ${\mathcal O}^+$ assume its maximum value.
\item When $e_2\in ( -\lambda,e^*_2(\lambda))$,  the radius of ${\mathcal O}^+$ decreases (ie.  ${\mathcal O}^+$ and $|[\gamma_\mathfrak{p}]|$ inflect). The evolution of the curve is in an expanding phase.
\item When $e_2=e^*_2(\lambda)$,  the osculating arc ${\mathcal O}^+$  returns to the limit position ${\mathcal O}_{\upsilon^*_{\lambda}}$.
\item When $e_2\in ( e^*_2(\lambda),\eta_-(\lambda))$,  the radius of ${\mathcal O}^+$ decreases (ie.  ${\mathcal O}^+$ and $|[\gamma_\mathfrak{p}]|$ inflect). The evolution of the curve is in an expanding phase.
\item When $e_2\to b_0(\lambda)^-$,   $|[\gamma_\mathfrak{p}]|$ and its osculating arc ${\mathcal O}^+$ tend to the upper hemicycle $\partial {\rm D}^2\cap \{(u,v)\in \R^2\,/\, v>0\}$ of the ideal boundary.
\end{itemize}
Figure \ref{FIG14} represents all stages of the evolution of the BS-curves of the family $\{\gamma_\mathfrak{p}\}_{e_2\in {\rm L}_{\lambda}}$, for the multiplier $\lambda=-1.1$. The trajectory of the BS-curves are colored in black, the upper osculating arcs $\mathcal{O}^+$ in red and the lower ones $\mathcal{O}^-$ in blue. The asymptotic limit position ${\mathcal O}_{\upsilon^*_{\lambda}}$ is the purple dashed circular arc and the asymptotic limit position as $e_2\to b_0(\lambda)^-$ is the dashed dark-red upper hemicycle of the ideal boundary. In the first two images ($e_2=0.92$ and $e_2=1$, respectively) two BS-curves in their contracting phases are reproduced. In the third image we show the BS-curve when it reaches its maximally contracted position and the radius of the upper osculating arc $\mathcal{O}^+$ reaches its maximum ($e_2=-\lambda=1.1$). A subsequent BS-curve in its expanding phase is shown in the fourth image ($e_2=1.2$). The first image of the second row ($e_2=e_2^*(\lambda)\approx 1.28$) reproduces the BS-curve when the upper osculating arc $\mathcal{O}^+$ returns to the asymptotic position ${\mathcal O}_{\upsilon^*_{\lambda}}$. The last three images ($e_2=1.54, 1.55$ and $1.56$) depict BS-curves in their expanding phases evolving toward the maximally expanded asymptotic limit (the upper hemicycle of the ideal boundary).
 
\begin{figure}[h]
	\makebox[\textwidth][c]{
		\includegraphics[height=4cm,width=4cm]{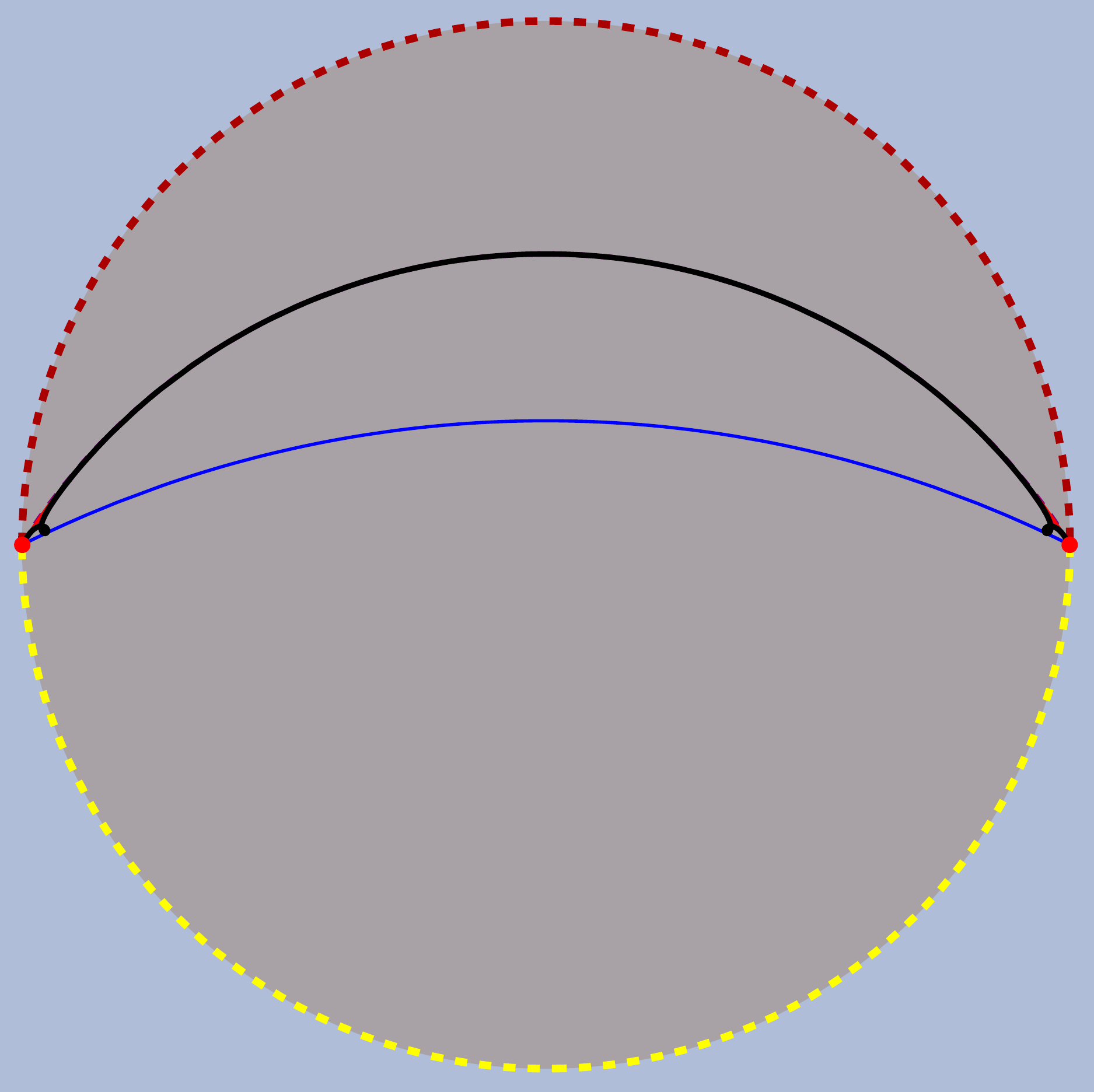}\quad
		\includegraphics[height=4cm,width=4cm]{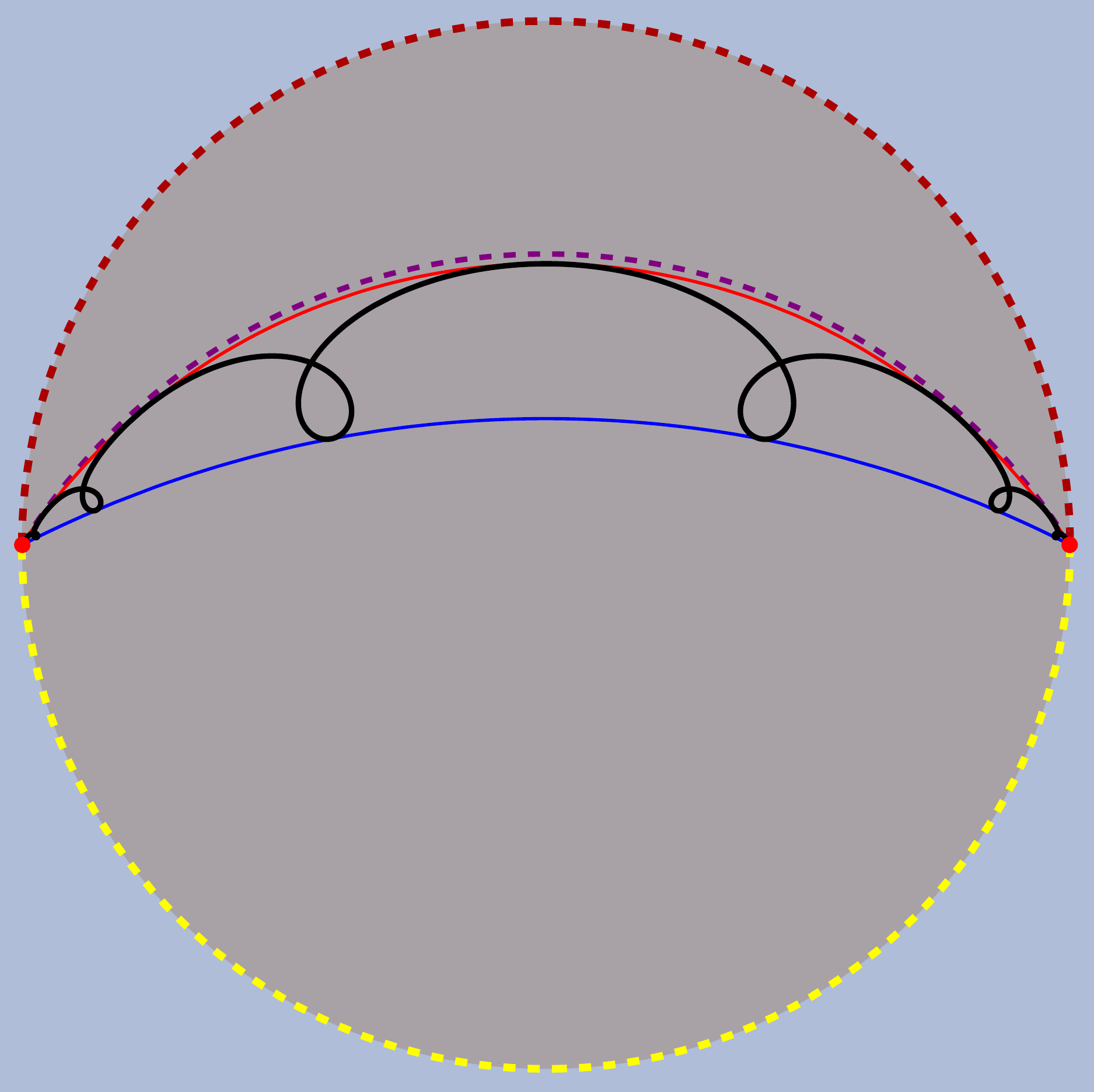}\quad
		\includegraphics[height=4cm,width=4cm]{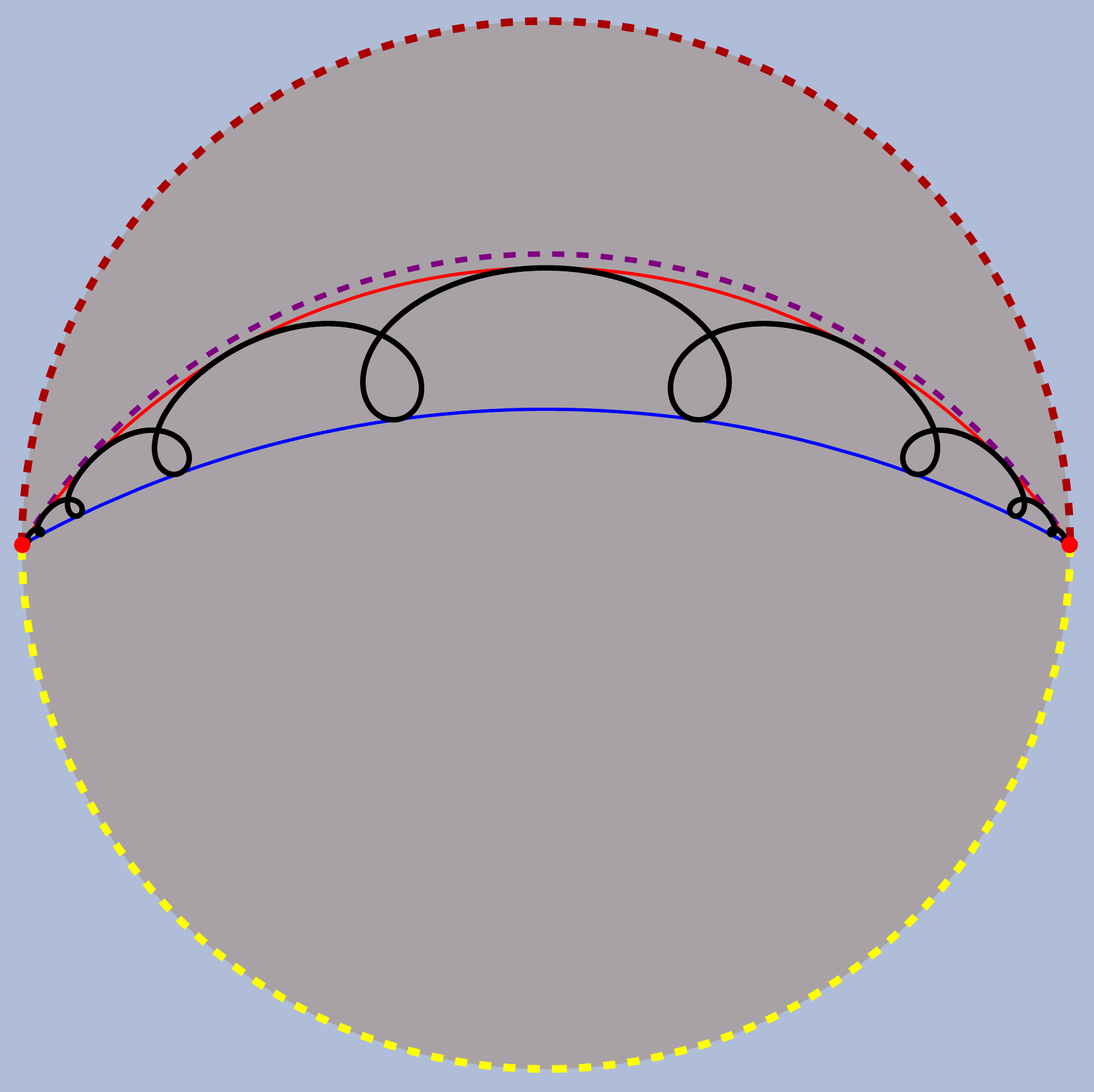}\quad
		\includegraphics[height=4cm,width=4cm]{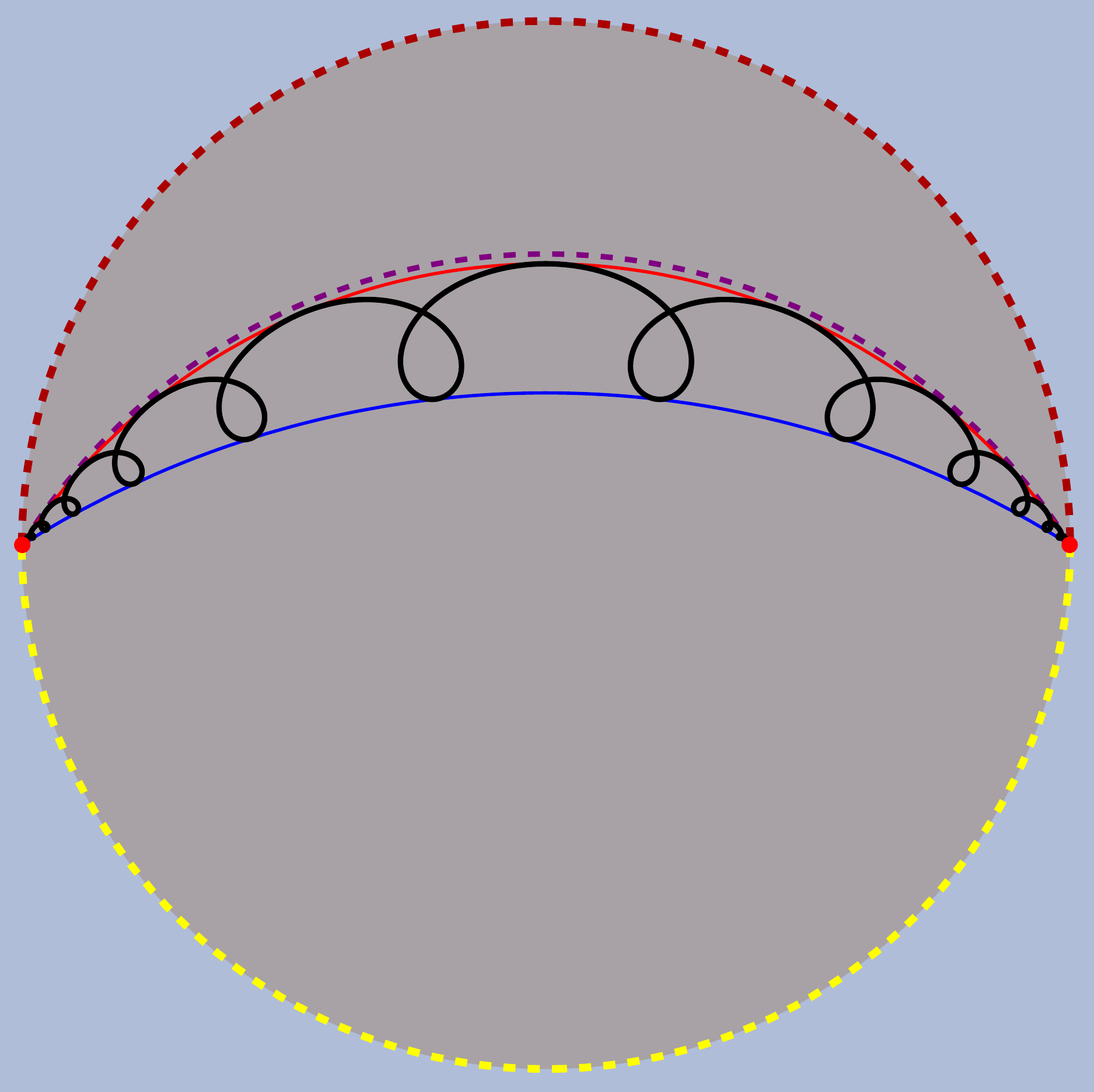}}\\\vspace{0.3cm}
	\makebox[\textwidth][c]{
	\hspace{0.01cm}	\includegraphics[height=4cm,width=4cm]{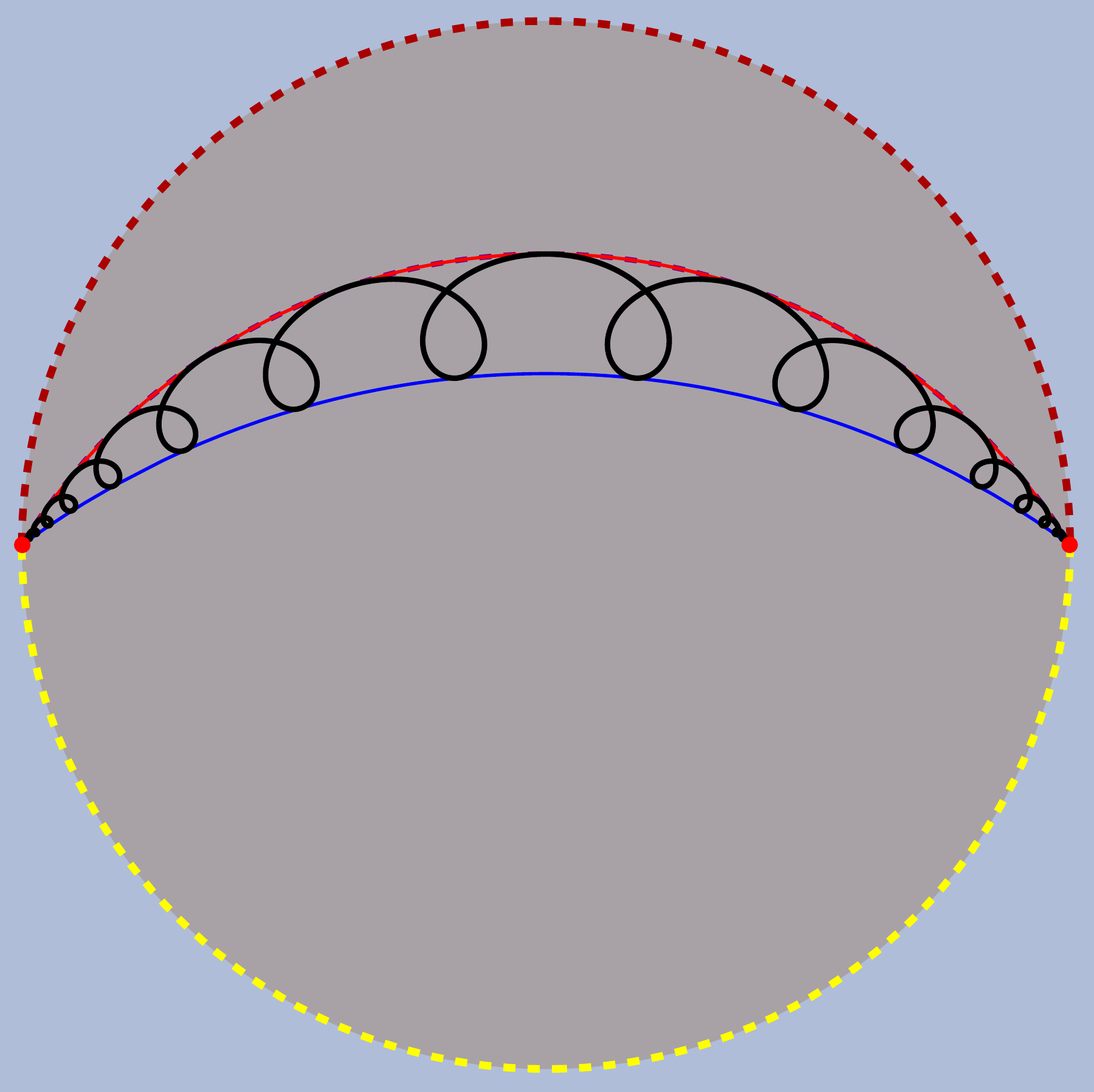}\quad
		\includegraphics[height=4cm,width=4cm]{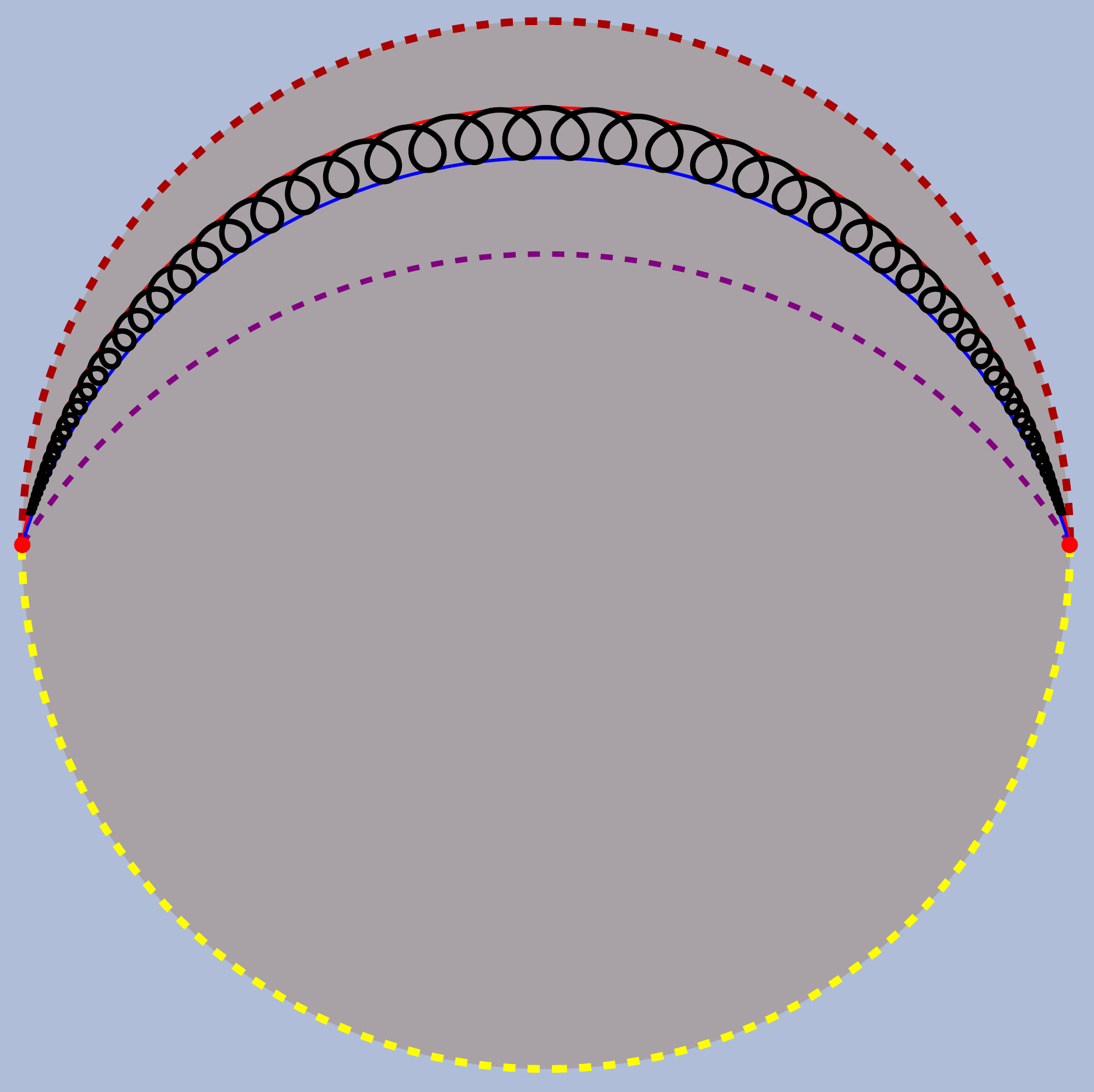}\quad
		\includegraphics[height=4cm,width=4cm]{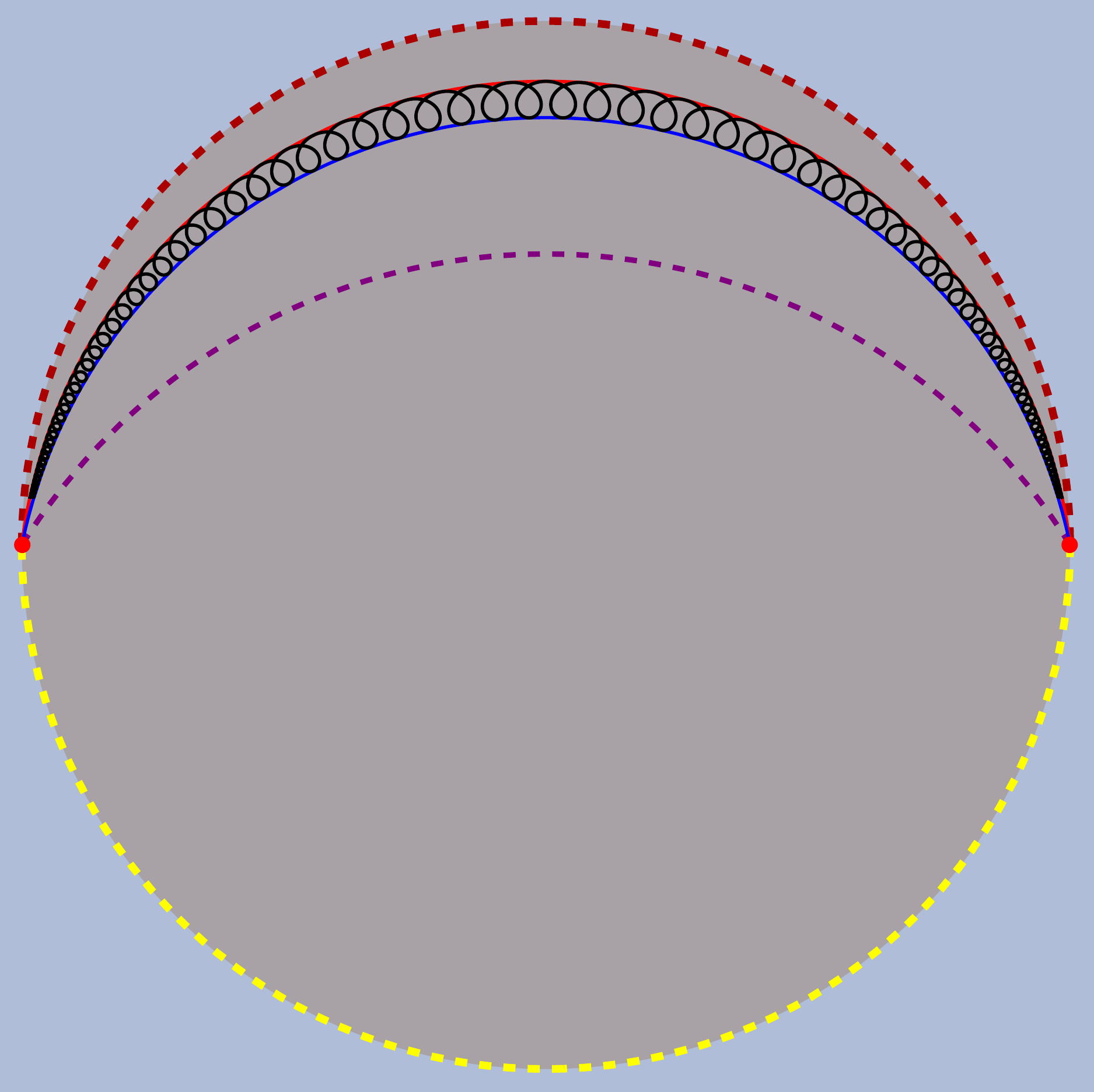}\quad
		\includegraphics[height=4cm,width=4cm]{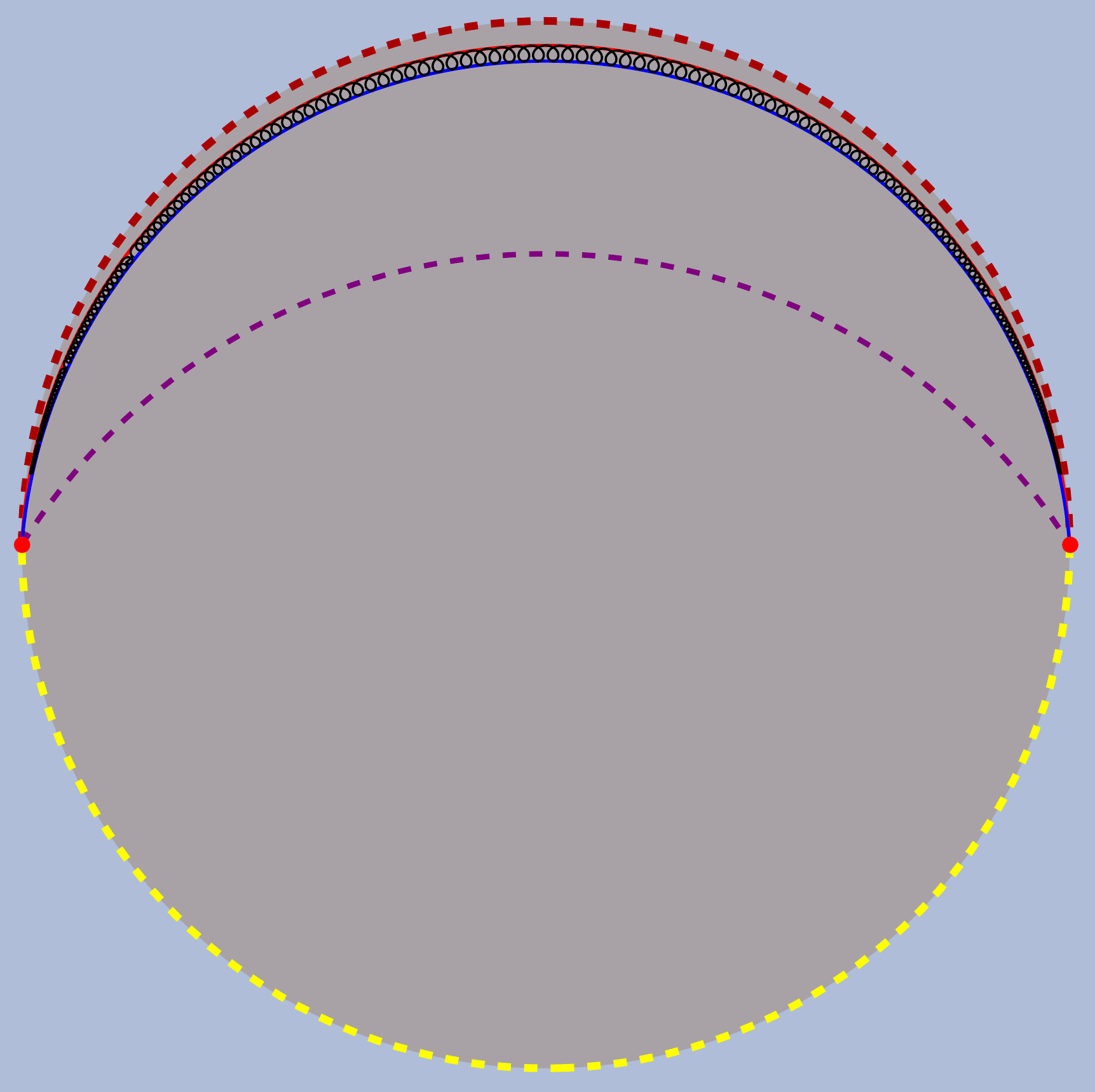}
	}
	\caption{\small{Trajectories of BS-curves with multiplier $\lambda=-1.1$ and increasing values of $e_2\in {\rm L}_\lambda$.
	}}\label{FIG14}
\end{figure}	
 
\subsection{BT-Curves}\label{OTT} 

Let $\mathfrak{p}=(\lambda,e_2)\in\mathcal{T}$. The B-curves associated with $\mathfrak{p}\in\mathcal{T}$ have time-like momentum. Their Blaschke invariant $\mu_\mathfrak{p}$ is the solution of \eqref{ode} with $c<0$ and $\mu_\mathfrak{p}(0)=e_2$. This case is analogous to the spherical one. Thus, we proceed as in \cite{MP} by introducing the exceptional locus and by defining the radial and angular functions, whose behavior will be essential in the description of the BT-curves.

\begin{defn} The curve ${\mathcal E}=\{(\lambda,e_2)\in {\mathfrak P}\,/\,  1+4 c e_1^2=0\}$ contained in the moduli space ${\mathfrak P}$ is the \emph{exceptional locus}.
\end{defn}

\begin{remark}\label{lemma5Modulispaces} \emph{The exceptional locus ${\mathcal E}$ is characterized by each one of the following equations:
		\begin{equation}\label{eqexlc}
			\begin{cases} 
				(i)\,  \,\, &e_1+2\lambda=0\,,\\
				(ii) \,\,\,& 1+16c\lambda^2=0\,,\\
				(iii)\, \,\,  & e_1^2e_2^3-e_1^3e_2^2+e_1+e_2=0\,,
			\end{cases}
		\end{equation}
		or as the graph of the function $c:(-\infty,-\sqrt[4]{\varphi^5}/2)\to   (\sqrt[4]{\varphi},+\infty)$ defined by
		\begin{equation}\label{c}c(\lambda)=-\frac{2}{3}\lambda+\frac{1}{3 \lambda^2}{\mathfrak R}\left(\sqrt[3]{-69\lambda^9+72\lambda^5-3i\lambda^3\sqrt{768\lambda^8+528\lambda^4+3}}\right),\end{equation}
		where $\varphi=(1+\sqrt{5})/2$ is the golden ratio. See Figures \ref{FIG4} and \ref{FIG20Moduli}.}
\end{remark}

Let $\eta_{\pm}$ and $c$ be as in  (\ref{eta}) and (\ref{c}), respectively.  Define  $a: (-\infty, -2/\sqrt[4]{27})\to \R$ by
$$a(\lambda)=\begin{cases}\sqrt{\lambda^2-1}-\lambda\,,\quad\quad  \lambda \le -1\,,\\
\eta_-(\lambda)\,,\quad\quad\quad\quad\,\,\,\, -1<\lambda < -2/\sqrt[4]{27}\,.
\end{cases} 
$$
Then, $a(\lambda)<\eta_+(\lambda)$, for all $\lambda<-2/\sqrt[4]{27}$, and $a(\lambda)<c(\lambda)<\eta_{+}(\lambda)$, for all $\lambda<-\sqrt[4]{\varphi^5}/2$. We decompose ${\mathcal T}$ as the disjoint union ${\mathcal T}={\mathcal T}_-\cup {\mathcal E}\cup  {\mathcal T}_+$
of the exceptional locus ${\mathcal E}$ with a lower domain (the green region of Figure \ref{FIG20Moduli}),
\[{\mathcal T}_-=\{(\lambda,e_2)\in {\mathcal T}\,/\,  \lambda <-\sqrt[4]{\varphi^5}/2,\,\,  a(\lambda)<e_2<c(\lambda) \}\,,\] 
and an upper domain (the brown domain of Figure \ref{FIG20Moduli}),
\[\begin{split} {\mathcal T}_+=& \{(\lambda,e_2)\in {\mathcal T}\,/\, \lambda <-\sqrt[4]{\varphi^5}/2, c(\lambda)\le e_2<\eta_+(\lambda)\}\,\,\cup\\
&  \{(\lambda,e_2)\in {\mathcal T}\,/\,  -\sqrt[4]{\varphi^5}/2\le\lambda < -2/\sqrt[4]{27},\,  a(\lambda)<e_2<\eta_+(\lambda)\}\,.
\end{split}\]

\begin{figure}[h]
\begin{center}
\includegraphics[height=4cm,width=6cm]{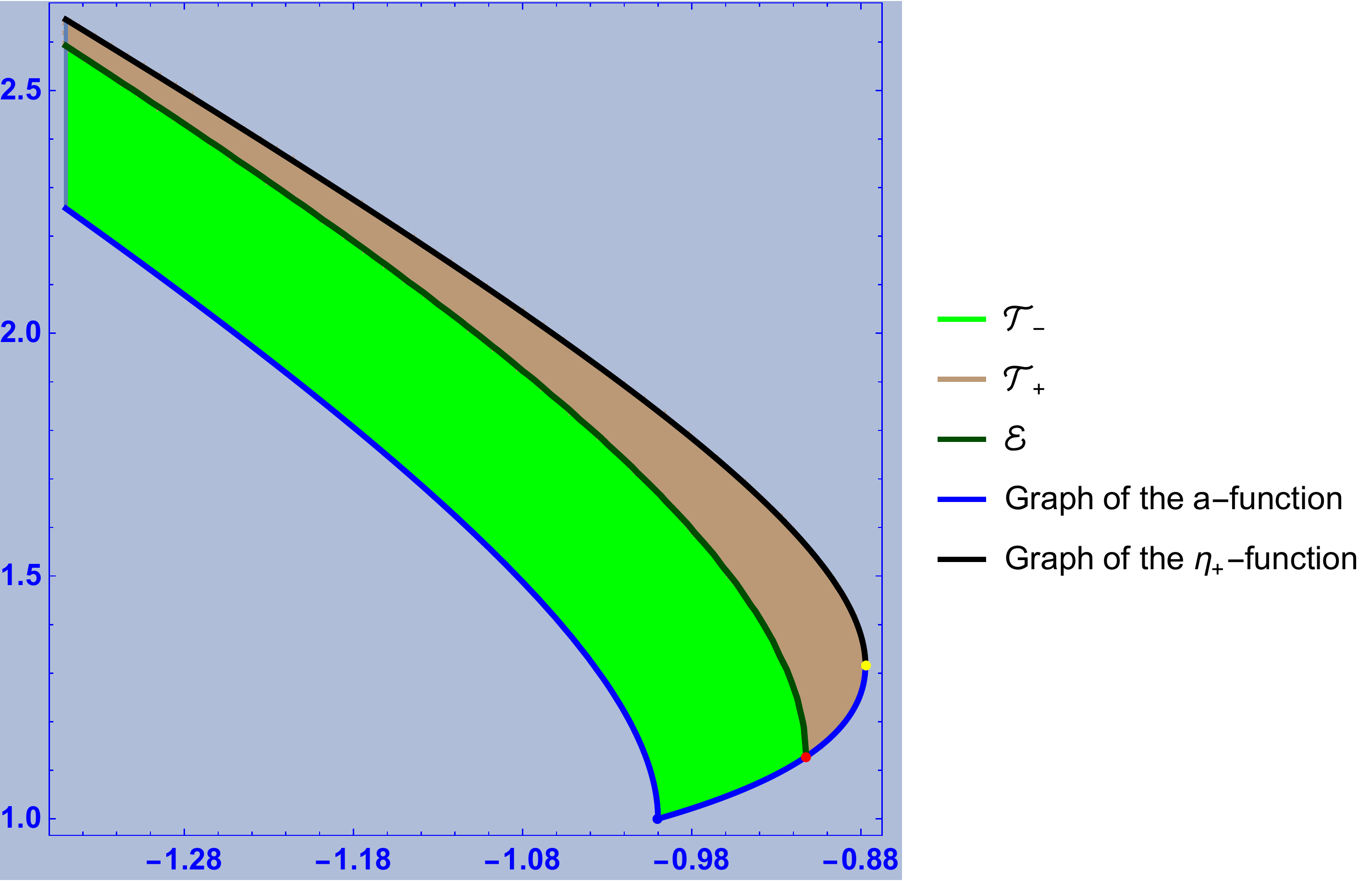}
\caption{\small{The decomposition of the moduli space $\mathcal{T}$ as ${\mathcal T}={\mathcal T}_-\cup {\mathcal E}\cup {\mathcal T}_+$.
}}\label{FIG20Moduli}
\end{center}
\end{figure}

\begin{defn} Let ${\mathfrak p}=(\lambda,e_2)\in {\mathcal T}$ and $\mu_{{\mathfrak p}}$ be the solution of (\ref{ode}) with initial condition $\mu_\mathfrak{p}(0)=e_2$ where $c<0$ is as in (vi) of \eqref{rele1e2Lc}. The \emph{radial function} $\varrho_{{\mathfrak p}}$ is defined by:
\begin{itemize}
\item If ${\mathfrak p}\notin {\mathcal E}$,
$$\varrho_{{\mathfrak p}}= \frac{\sqrt{1+4c\mu_{{\mathfrak p}}^2}\,}{2\sqrt{|c|}\,\mu_{{\mathfrak p}}}\,.$$
\item  If ${\mathfrak p} \in {\mathcal E}$,  
$$\varrho_{{\mathfrak p}}= \begin{cases} \,\,\,\, \frac{\sqrt{1+4c\mu_{{\mathfrak p}}^2}\,}{2\sqrt{|c|}\mu_{{\mathfrak p}}}\,,
	\quad\quad {\rm on}\,\,\,  [-\frac{\omega_{{\mathfrak p}}}{2}\,, \frac{\omega_{{\mathfrak p}}}{2}],\\
	-\frac{\sqrt{1+4c\mu_{{\mathfrak p}}^2}\,}{2\sqrt{|c|}\mu_{{\mathfrak p}}}\,, \quad\quad  {\rm on}\,\,\,  
	[\frac{\omega_{{\mathfrak p}}}{2}, \frac{3\omega_{{\mathfrak p}}}{2}]\,.
\end{cases}$$
\end{itemize}
\end{defn}

\begin{figure}[h]
	\begin{center}
		\includegraphics[height=4cm,width=6cm]{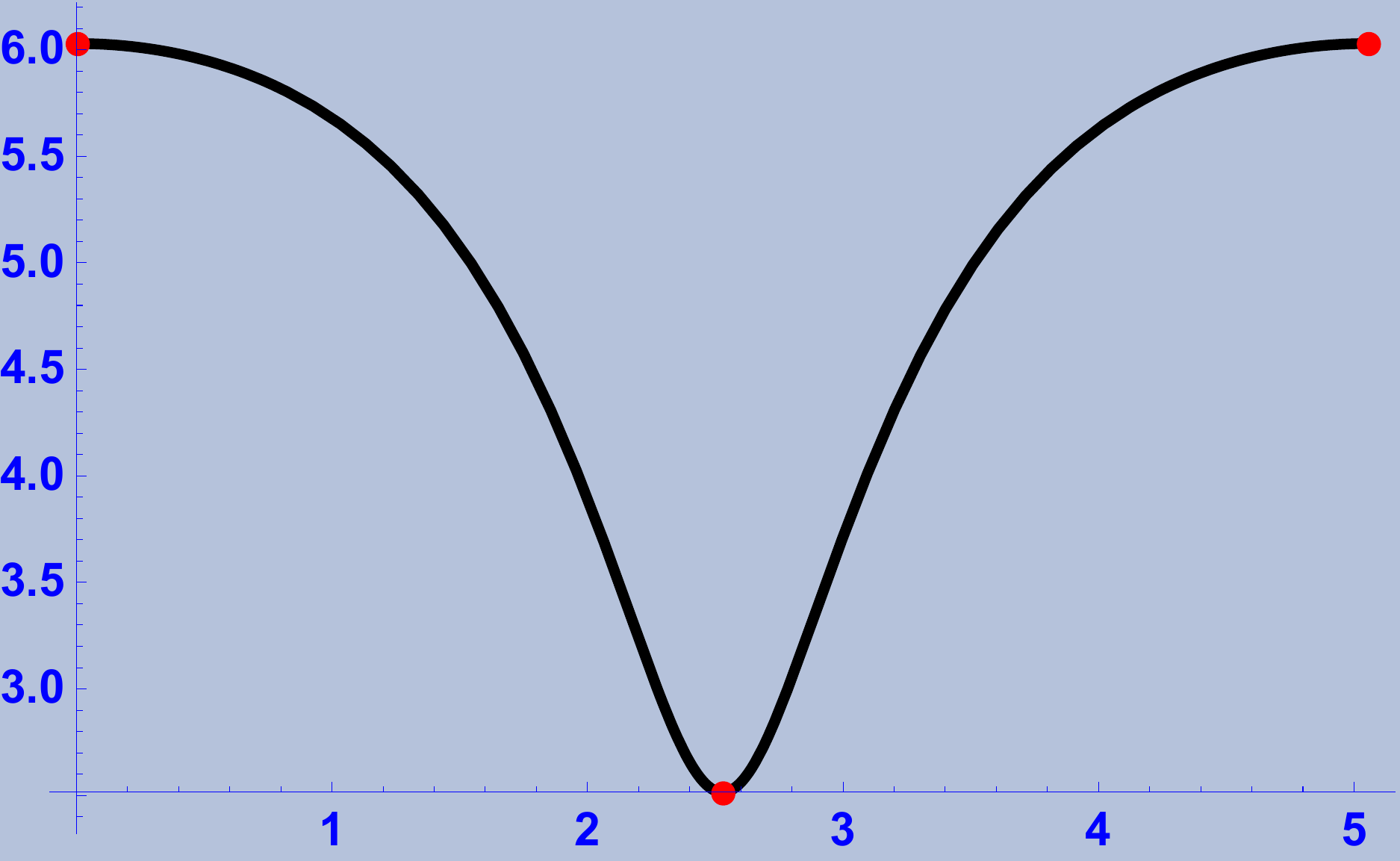}\quad\,\,\,
		\includegraphics[height=4cm,width=6cm]{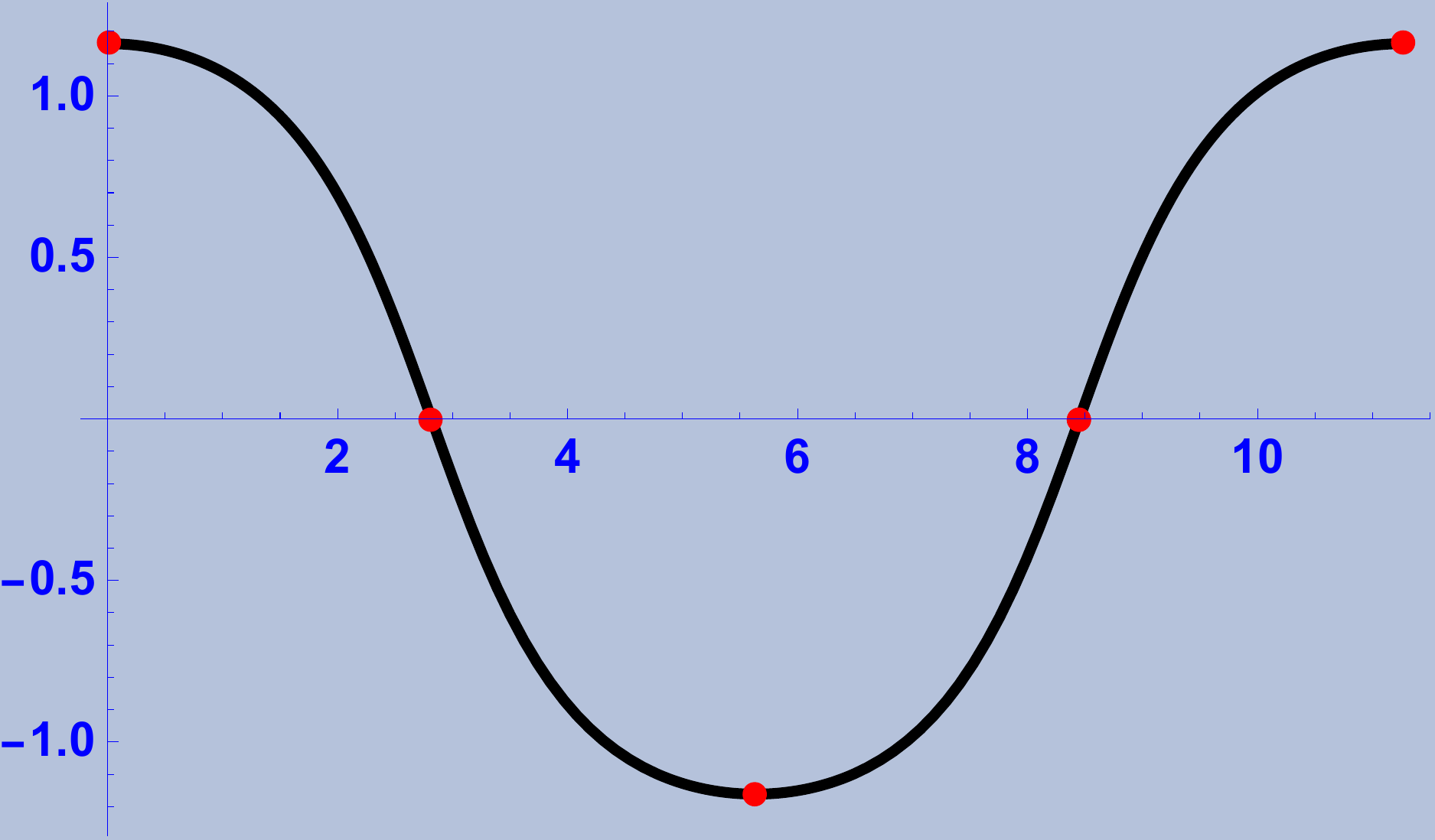}
		\caption{\small{The graphs of the radial functions $\varrho_{\mathfrak p}$ with ${\mathfrak p}=(\lambda,e_2)\in\mathcal{T}$. Left: Non-exceptional case $\lambda=-0.99$ and $e_2=1.05$ represented from $0$ to the least period $\omega_{\mathfrak p}$. Right: Exceptional case $\lambda=-0.91$ and $e_2=1.19$ represented from $0$ to the least period $2\omega_\mathfrak{p}$.
		}}\label{FIG19}
	\end{center}
\end{figure}

We now describe the main features of the radial functions (see Figure \ref{FIG19}):
\begin{itemize}
\item The radial function $\varrho_{{\mathfrak p} }$ is real-analytic, even and periodic,  for every ${\mathfrak p} \in {\mathcal T}$.  
\item If ${\mathfrak p} \notin  {\mathcal E}$,  $\varrho_{{\mathfrak p}}>0$ with least period  $\omega_{{\mathfrak p}}$.  It attains the maximum at
$h\omega_{{\mathfrak p} }$ and the minimum at $\omega_ {{\mathfrak p} }/2+h\omega_{{\mathfrak p} }$, $h\in \Z$. It is strictly decreasing on 
$[h\omega_{{\mathfrak p} }, \omega_{{\mathfrak p} }/2+h\omega_{\mathfrak p}]$ and strictly increasing on $[ \omega_{{\mathfrak p} }/2+h\omega_{{\mathfrak p} }, (h+1)\omega_{{\mathfrak p} }]$, $h\in \Z$.
\item If ${\mathfrak p} \in  {\mathcal E}$,  $\varrho_{{\mathfrak p}}$ is periodic, with least period $2\omega_{{\mathfrak p} }$. It is strictly positive on
$(\omega_{{\mathfrak p} }/2+h\omega_{{\mathfrak p} },  \omega_{{\mathfrak p} }/2+(h+1)\omega_{{\mathfrak p} })$,  $h\in \Z$ and ${\rm mod}(h,2)=1$, 
and strictly negative on $(\omega_{{\mathfrak p} }/2+h\omega_{{\mathfrak p} },  \omega_{{\mathfrak p} }/2+(h+1)\omega_{{\mathfrak p} })$, 
 $h\in \Z$ and ${\rm mod}(h,2)=0$.  It is strictly decreasing on  $[h\omega_{{\mathfrak p} }, (h+1)\omega_{{\mathfrak p} }]$, ${\rm mod}(h,2)=0$,  and strictly increasing on $[h\omega_{{\mathfrak p} }, (h+1)\omega_{{\mathfrak p} }]$, ${\rm mod}(h,2)=1$.
\end{itemize}

\begin{defn} Let ${\mathfrak p}=(\lambda,e_2)\in {\mathcal T}$ and $\mu_{{\mathfrak p}}$ be the solution of (\ref{ode}) with initial condition $\mu_\mathfrak{p}(0)=e_2$ where $c<0$ is as in (vi) of \eqref{rele1e2Lc}. The \emph{angular function} $\Theta_\mathfrak{p}$ is defined by:
\begin{itemize}
	\item If ${\mathfrak p}\notin {\mathcal E}$,
	$$\Theta_{{\mathfrak p}}(s)=2\sqrt{|c|}\int_0^s\frac{\mu_{{\mathfrak p}}^2\left(\mu_{{\mathfrak p}}+2\lambda\right)}{1+4 c \mu_{{\mathfrak p}}^2}\,ds\,.$$
	\item  If ${\mathfrak p} \in {\mathcal E}$,  
	$$\Theta_{{\mathfrak p}}(s)=-8\sqrt{|c|}\,\lambda^2\int_0^s\frac{\mu_{{\mathfrak p}}^2}{\mu_{{\mathfrak p}}-2\lambda}\,ds\,.$$
\end{itemize}
\end{defn}

\begin{figure}[h]
	\begin{center}
		\includegraphics[height=4cm,width=6cm]{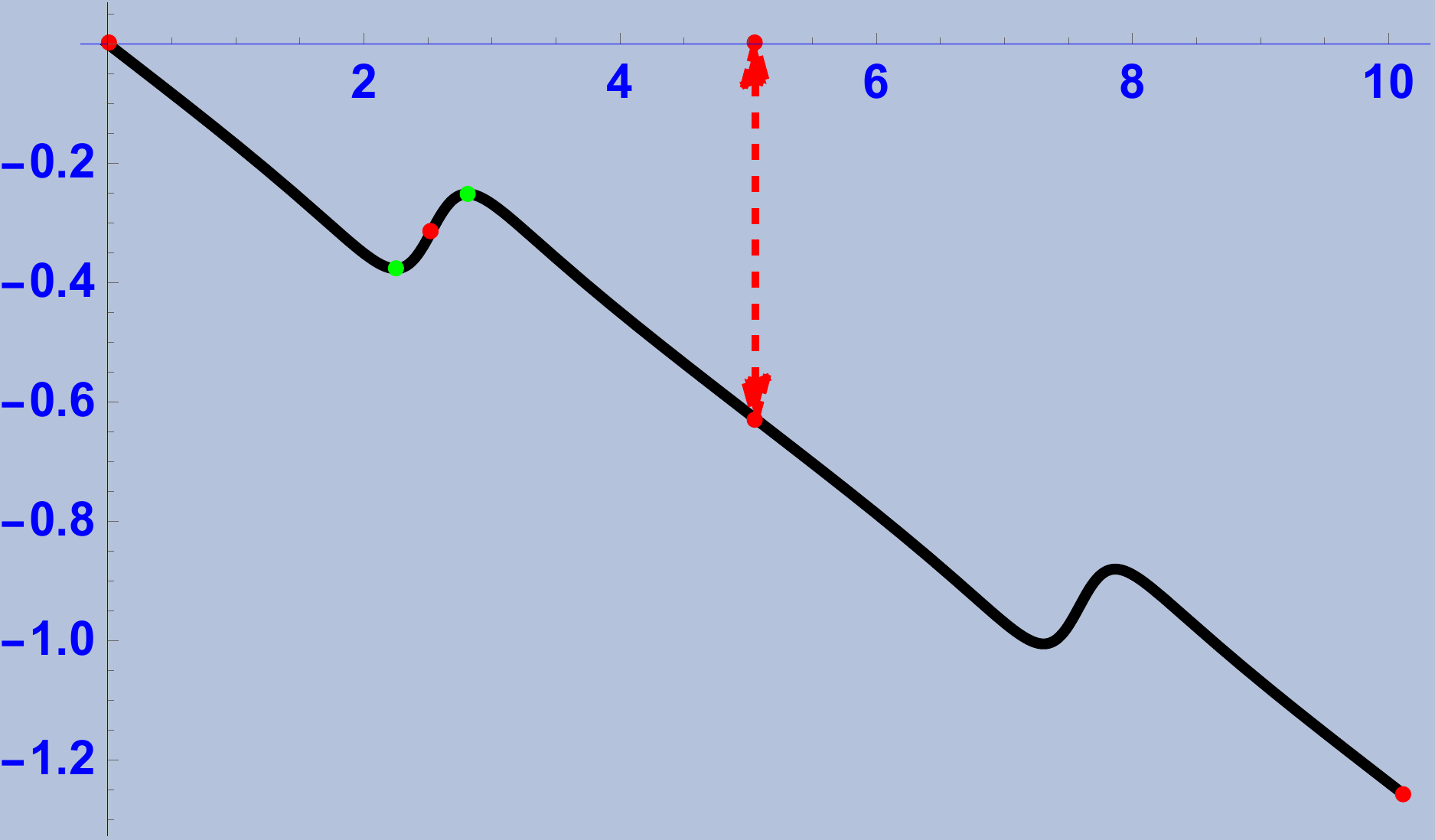}\quad\,\,\,
		\includegraphics[height=4cm,width=6cm]{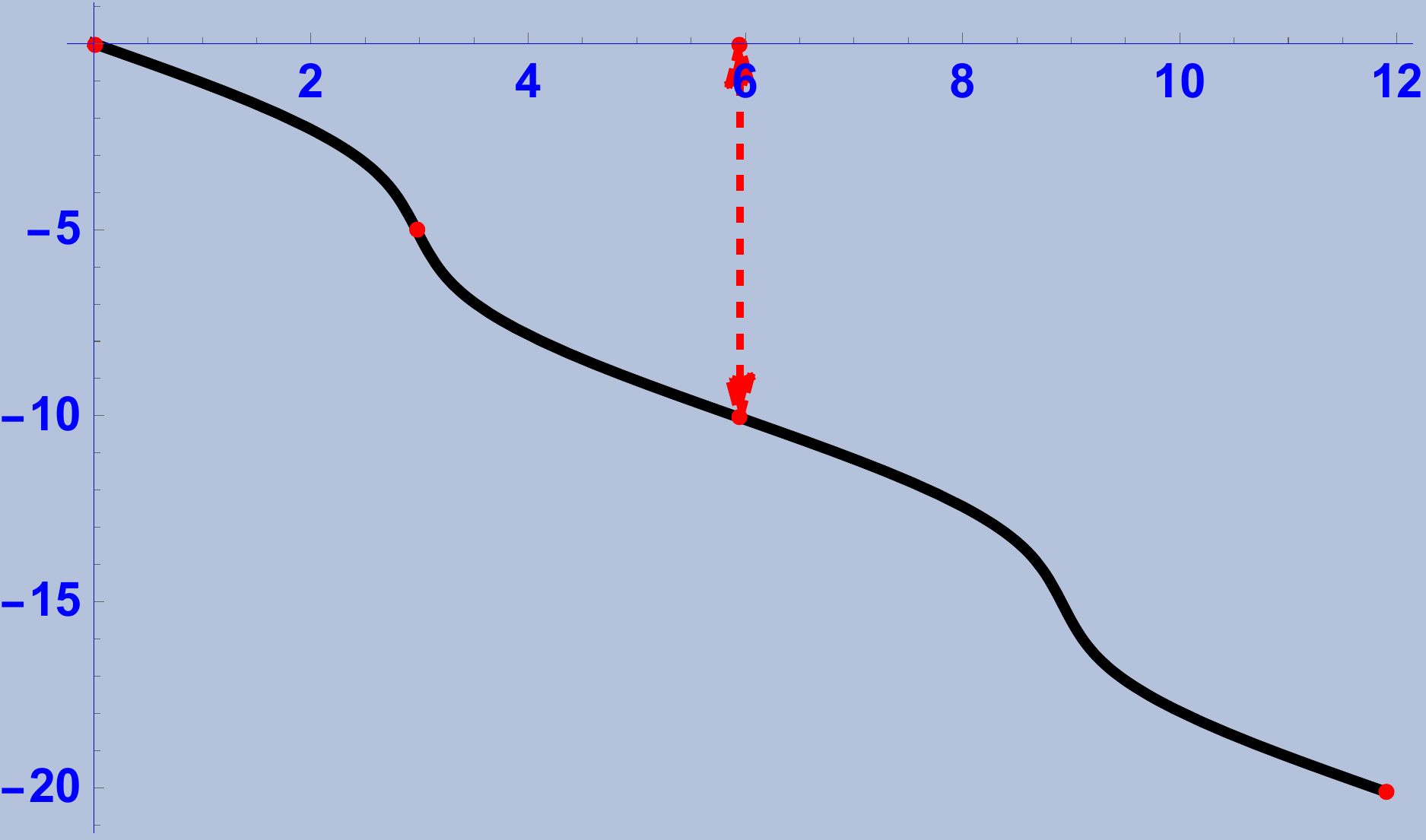}
		\caption{\small{The graphs of the angular functions $\Theta_{\mathfrak p}$ with ${\mathfrak p}=(\lambda,e_2)\in\mathcal{T}$, represented from $0$ to two least periods $2\omega_{\mathfrak p}$. Left: Case ${\mathfrak p}\in\mathcal{T}_-$ where $\lambda=-0.99$ and $e_2=1.05$. Right: Case $\mathfrak{p}\in\mathcal{T}_+$ where $\lambda=-0.9$ and $e_2=1.24$. The dashed arrow represents the value $\Theta_{\mathfrak p}(\omega_{\mathfrak p})$. If the length of this arrow is a rational multiple of $2\pi$, the BT-curve $\gamma_{\mathfrak p}$ will be closed (see Section \ref{BT-Strings}).
		}}\label{FIG20}
	\end{center}
\end{figure}

We now describe the main features of the angular functions (see Figure \ref{FIG20}):
\begin{itemize}
\item The angular function $\Theta_{{\mathfrak p}}$ is an odd quasi-periodic function with least quasi-period $\omega_{{\mathfrak p}}$.
\item The function $\Theta_{{\mathfrak p}}$ tends to $-\infty$ when $s\to +\infty$ and to $+\infty$ when $s\to -\infty$.
\item If ${\mathfrak p}\in {\mathcal E}\cup {\mathcal T}_+$, 
$\Theta_{{\mathfrak p}}$  is strictly decreasing,  with ascending inflection points at $\omega_{{\mathfrak p}}/2+h\omega_{{\mathfrak p}}$, $h\in \Z$ and descending inflection points at $\omega_{{\mathfrak p}}+h\omega_{{\mathfrak p}}$, $h\in \Z$.
\item If ${\mathfrak p}\in {\mathcal T}_-$,   $\Theta_{{\mathfrak p}}$ has a relative minimum at $s^*_{{\mathfrak p}}\in (0,\omega_{{\mathfrak p}}/2)$ and a relative maximum at  $s^{**}_{{\mathfrak p}}=\omega_{{\mathfrak p}}-s^*_{{\mathfrak p}}$.  There are no other critical points in the interval $[0,\omega_{{\mathfrak p}}]$.   $\Theta_{{\mathfrak p}}$ has descending inflection points at $\omega_{{\mathfrak p}}/2+h\omega_{{\mathfrak p}}$, $h\in \Z$ and ascending inflection points at $\omega_{{\mathfrak p}}+h\omega_{{\mathfrak p}}$, $h\in \Z$.
\end{itemize}

Once the radial and angular functions have been introduced and their behavior described, we can proceed with the parameterization of BT-curves.

\begin{thm}\label{IntqS} Let ${\mathfrak p}=(\lambda,e_2)\in {\mathcal T}$ and $\mu_{{\mathfrak p}}$ be the solution of \eqref{ode} with initial condition $\mu_\mathfrak{p}(0)=e_2$ where $c<0$ is as in $(vi)$ of (\ref{rele1e2Lc}). Let $\varrho_{{\mathfrak p}}$ and $\Theta_{{\mathfrak p}}$ be the radial and angular functions defined above, respectively.
Then,
\begin{equation}\label{eq1OTSBis}
	\gamma_{{\mathfrak p}}=\left(\frac{1}{2\sqrt{|c|}\,\mu_{{\mathfrak p}}},-\varrho_{{\mathfrak p}}\cos(\Theta_{{\mathfrak p}}), \varrho_{{\mathfrak p}}\sin(\Theta_{{\mathfrak p}})\right)
\end{equation}
is a BT-curve with modulus $\mathfrak p$ and momentum
$\vec{\xi}=(\sqrt{|c|},0,0)$.
\end{thm}
\begin{proof} The proof of this theorem is analogous to the previous cases (see Theorem \ref{TypeL}).
\end{proof}

\begin{defn} The BT-curve $\gamma_\mathfrak{p}$ given by \eqref{eq1OTSBis} is referred to as the \emph{standard BT-curve} with modulus $\mathfrak{p}\in\mathcal{T}$.
\end{defn}

\begin{remark} \emph{Every BT-curve with modulus  ${\mathfrak p}\in {\mathcal T}$ is equivalent to $\gamma_{{\mathfrak p}}$, \eqref{eq1OTSBis}. From now on we implicitly assume that the BT-curves in consideration are in their standard form.}
\end{remark}

Let $\mathfrak{p}=(\lambda,e_2)\in\mathcal{T}$ and $\gamma_\mathfrak{p}$ be the standard BT-curve with modulus $\mathfrak{p}$. In the  Poincar\'e model,  $\gamma_{{\mathfrak p}}$ can be written as
\begin{equation}\label{TypeSDisk}
		\gamma_{{\mathfrak p}}=\frac{2\sqrt{|c|}\,\varrho_{{\mathfrak p}}\mu_\mathfrak{p}}{1+2\sqrt{|c|}\,\mu_{{\mathfrak p}}}(-\cos(\Theta_{{\mathfrak p}}),\sin(\Theta_{{\mathfrak p}})).
\end{equation}
Thus,  $|[\gamma_{{\mathfrak p}}]|$  is contained in the region  
$${\mathtt D}^2_{{\mathfrak p}}=\{\zeta \in {\rm D}^2/  {\rm r}'_{{\mathfrak p}}\le  |\zeta|\le {\rm r}''_{{\mathfrak p}}\},$$
where
$$\begin{cases}
	r'_{{\mathfrak p}}={\rm min}{\big\lvert}\frac{2\sqrt{|c|}\,\varrho_{{\mathfrak p}}\mu_\mathfrak{p}}{1+2\sqrt{|c|}\mu_{{\mathfrak p}}}{\big \lvert}=\frac{\sqrt{1+4c e_1^2}}{1+2\sqrt{|c|}e_1},\\
	r''_{{\mathfrak p}}={\rm max}{\big\lvert}\frac{2\sqrt{|c|}\,\varrho_{{\mathfrak p}}\mu_\mathfrak{p}}{1+2\sqrt{|c|}\mu_{{\mathfrak p}}}{\big \rvert}=\frac{\sqrt{1+4ce_2^2}}{1+2\sqrt{|c|}e_2}\,.
\end{cases}
$$
From this we see that $r'_\mathfrak{p}=0$ if and only if $1+4c e_1^2=0$, or in other words using the definition of the exceptional locus, if and only if $\mathfrak{p}\in\mathcal{E}$. Consequently, the trajectory $|[\gamma_{{\mathfrak p}}]|$  passes through the origin if and only if ${\mathfrak p}\in {\mathcal E}$.  In this case ${\mathtt D}_{{\mathfrak p}}^2$ is a disk. Otherwise ${\mathtt D}_{{\mathfrak p}}^2$ is an annulus and the outer and inner circles bounding ${\mathtt D}_{{\mathfrak p}}^2$,  denoted by $\mathcal{O}^{\pm}$ are said the \emph{outer and inner osculating circles}. The trajectory $|[\gamma_{{\mathfrak p}}]|$ is tangent to $\mathcal{O}^{+}$ at  $\gamma_{{\mathfrak p}}(h\omega_{{\mathfrak p}})$, $h\in \Z$ and is tangent to  $\mathcal{O}^{-}$ at $\gamma_{{\mathfrak p}}(\omega_{{\mathfrak p}}/2+h\omega_{{\mathfrak p}})$, $h\in \Z$. From (\ref{eq1OTSBis}) it follows that $\gamma_\mathfrak{p}$ is closed if and only if $\Theta_\mathfrak{p}(\omega_\mathfrak{p})/2\pi\in {\mathbb Q}$. Otherwise $|[\gamma_{{\mathfrak p}}]|$ is a dense subset of 
${\mathtt D}_{{\mathfrak p}}^2$. Several closed examples will be illustrated in Figures \ref{FIG22} and \ref{FIG23}.

\section{BT-Strings}\label{BT-Strings}

In this section we will study the existence of BT-strings and the kinematics of their corresponding isomonodromic family. As mentioned above a BT-curve $\gamma_\mathfrak{p}$ will be closed if and only if its angular function $\Theta_\mathfrak{p}$ evaluated at the wavelength $\omega_{\mathfrak p}$ is a rational multiple of $2\pi$ (the quantity $\Theta_{\mathfrak p}(\omega_{\mathfrak p})$ is represented in Figure \ref{FIG20} by the dashed arrow). In order to analyze this closure condition we will define a couple of functions which will appear in an essential way.

\begin{defn} Let $\mathfrak{p}\in\mathcal{T}$ and $\gamma_\mathfrak{p}$ be the BT-curve with modulus $\mathfrak{p}$. The \emph{period map} ${\mathcal P}:{\mathcal T}\to \R$ is defined by
\begin{equation*} {\mathcal P}(\mathfrak p)=\begin{cases}  -\frac{1}{2\pi}\Theta_{\mathfrak p}(\omega_{\mathfrak p})+1\,,\quad\quad\,\,{\rm if}\,\,\, \mathfrak p\in {\mathcal T}_-\,,\\
-\frac{1}{2\pi}\Theta_{\mathfrak p}(\omega_{\mathfrak p})+\frac{1}{2}\,,\quad\quad\,{\rm if}\,\,\, \mathfrak p\in {\mathcal E}\,,\\
-\frac{1}{2\pi}\Theta_{\mathfrak p}(\omega_{\mathfrak p})\,,\quad \quad\quad \quad {\rm if}\,\,\,\mathfrak p\in {\mathcal T}_+\,,
\end{cases}
\end{equation*}
where $\Theta_\mathfrak{p}$ is the angular function and $\omega_\mathfrak{p}$ is the wavelength, ie. the least period of the Blaschke invariant of $\gamma_\mathfrak{p}$.
\end{defn}

\begin{remark}\label{4.2} \emph{Since ${\mathcal P}({\mathfrak p})\cong - \Theta_\mathfrak{p}(\omega_\mathfrak{p})/2\pi\,\, {\rm mod}({\mathbb Q})$ we conclude that $\gamma_\mathfrak{p}$ is closed (ie. a BT-string) if and only if ${\mathcal P}({\mathfrak p})\in {\mathbb Q}$. Hence, in order to prove the existence of BT-strings, it suffices to check that, for fixed multiplier $\lambda$, the image of $\mathcal{P}\equiv\mathcal{P}_\lambda$ is not constant, ie. $\mathcal{P}_\lambda(e_2)$ attains different values as $e_2$ varies in ${\rm I}_{\lambda}=(a(\lambda),\eta_+(\lambda))$. Recall that the graph of the function $a$ is the lower boundary of the moduli space $\mathcal{T}$, while the graph of $\eta_+$ is the upper boundary.} 
\end{remark}

\begin{defn}\label{4.3} Let $\mathfrak{p}\in\mathcal{T}$ and $\gamma_\mathfrak{p}$ be the BT-curve with modulus $\mathfrak{p}$. The function $\chi: (-\infty, -2/\sqrt[4]{27}\,)\to \R$ is defined by 
	\begin{equation*}\label{chifn}\chi(\lambda) =  \frac{\eta_+(\lambda)^4-1}{\sqrt{\eta_+(\lambda)^8-4\eta_+(\lambda)^4+3}}\in \R\,,
	\end{equation*}
	where $\eta_+(\lambda)$ is the function defined in \eqref{eta}.
\end{defn}

\begin{remark} \emph{The function $\chi$ is strictly increasing and tends to $1$ when $\lambda\to -\infty$ and to $+\infty$ when $\lambda\to -2/\sqrt[4]{27}\,^-$.} 
\end{remark}

We denote by ${\rm J}_{\lambda}$ the open interval 
$${\rm J}_{\lambda}=\begin{cases} (1,\chi(\lambda))\,,\quad\quad\quad\quad \lambda\le -1\,,\\
	(\chi(\lambda),+\infty)\,,\quad\quad\,\,\,  -1<\lambda <-2/\sqrt[4]{27}\,.
\end{cases}$$

The following result shows the existence of infinitely many BT-strings for every multiplier $\lambda<-2/\sqrt[4]{27}$.

\begin{thm}\label{existence} For every $\lambda<-2/\sqrt[4]{27}$ and every $q=m/n\in {\rm J}_{\lambda}\cap{\mathbb Q}$ there exist ${\mathfrak p}=(\lambda,e_2)\in {\mathcal T}$ such that $\gamma_{\mathfrak p}$ is a BT-string with ${\mathcal P}({\mathfrak p})=q$. 
\end{thm}
\begin{proof} The proof of this result is a direct consequence of the following claim involving the limits of the period map:
\\

\noindent {\bf Claim.} \emph{Let ${\rm I}_{\lambda}=(a(\lambda),\eta_+(\lambda))$. The period map ${\mathcal P}_{\lambda}:{\rm I}_{\lambda}\to \R$ defined by ${\mathcal P}_{\lambda}(e_2)={\mathcal P}(\lambda,e_2)$ is continuous. Moreover:
	\begin{itemize}
		\item If $\lambda\leq -1$,
		$$\lim_{e_2\to a(\lambda)^+}{\mathcal P}_{\lambda}(e_2)=1\,,\quad\quad\quad \lim_{e_2\to \eta_+(\lambda)^-}{\mathcal P}_{\lambda}(e_2)=\chi(\lambda)\,.$$
		\item If $-1<\lambda<-2/\sqrt[4]{27}$,
		$$\lim_{e_2\to a(\lambda)^+}{\mathcal P}_{\lambda}(e_2)=+\infty\,,\quad\quad\quad \lim_{e_2\to \eta_+(\lambda)^-}{\mathcal P}_{\lambda}(e_2)=\chi(\lambda)\,.$$
	\end{itemize}}

\noindent For the sake of clarity, the proof of this claim will be postponed to the appendix.  
\end{proof}

\begin{figure}[h]
	\makebox[\textwidth][c]{
		\includegraphics[height=4cm,width=4cm]{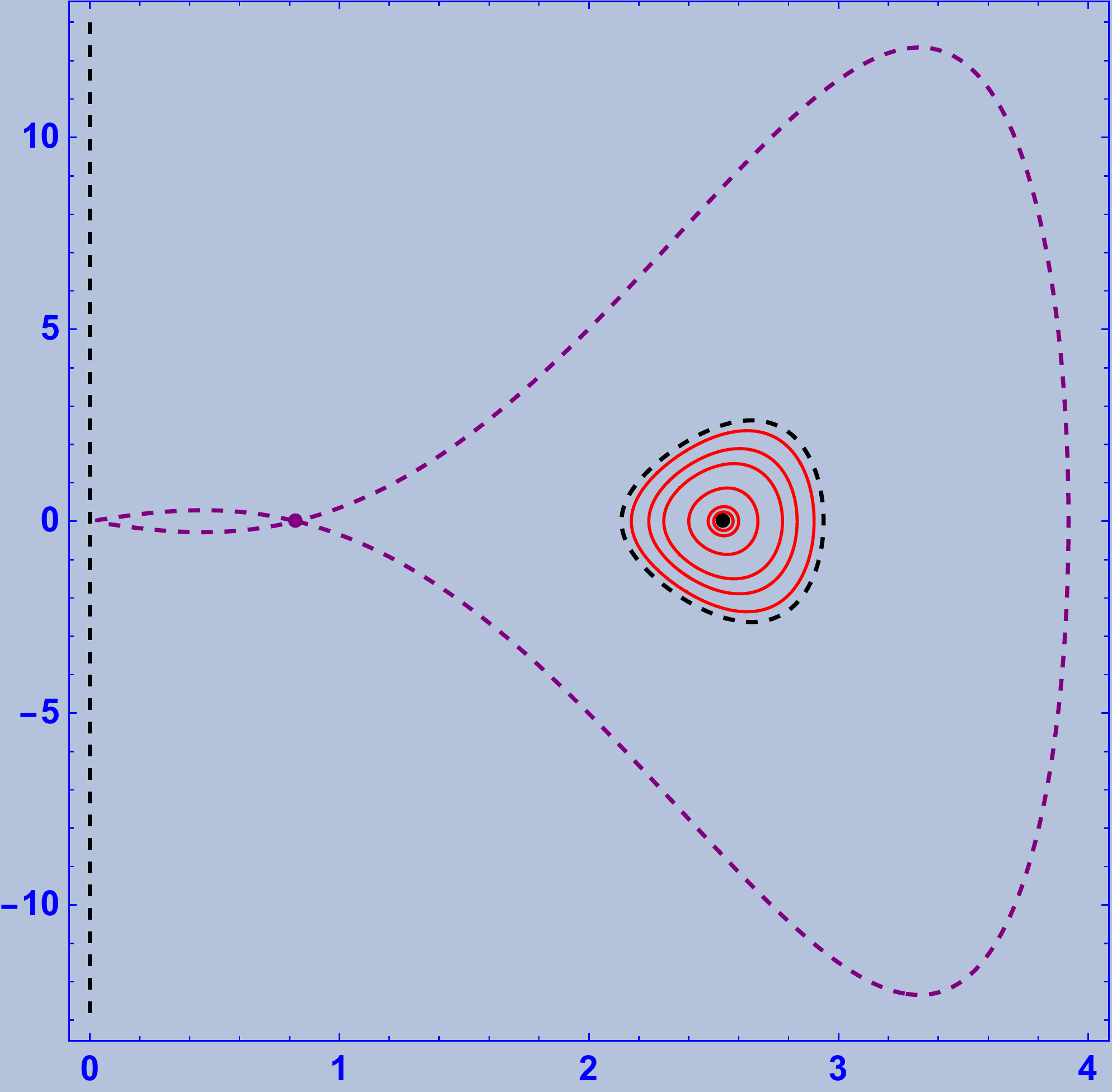}\quad
		\includegraphics[height=4cm,width=4cm]{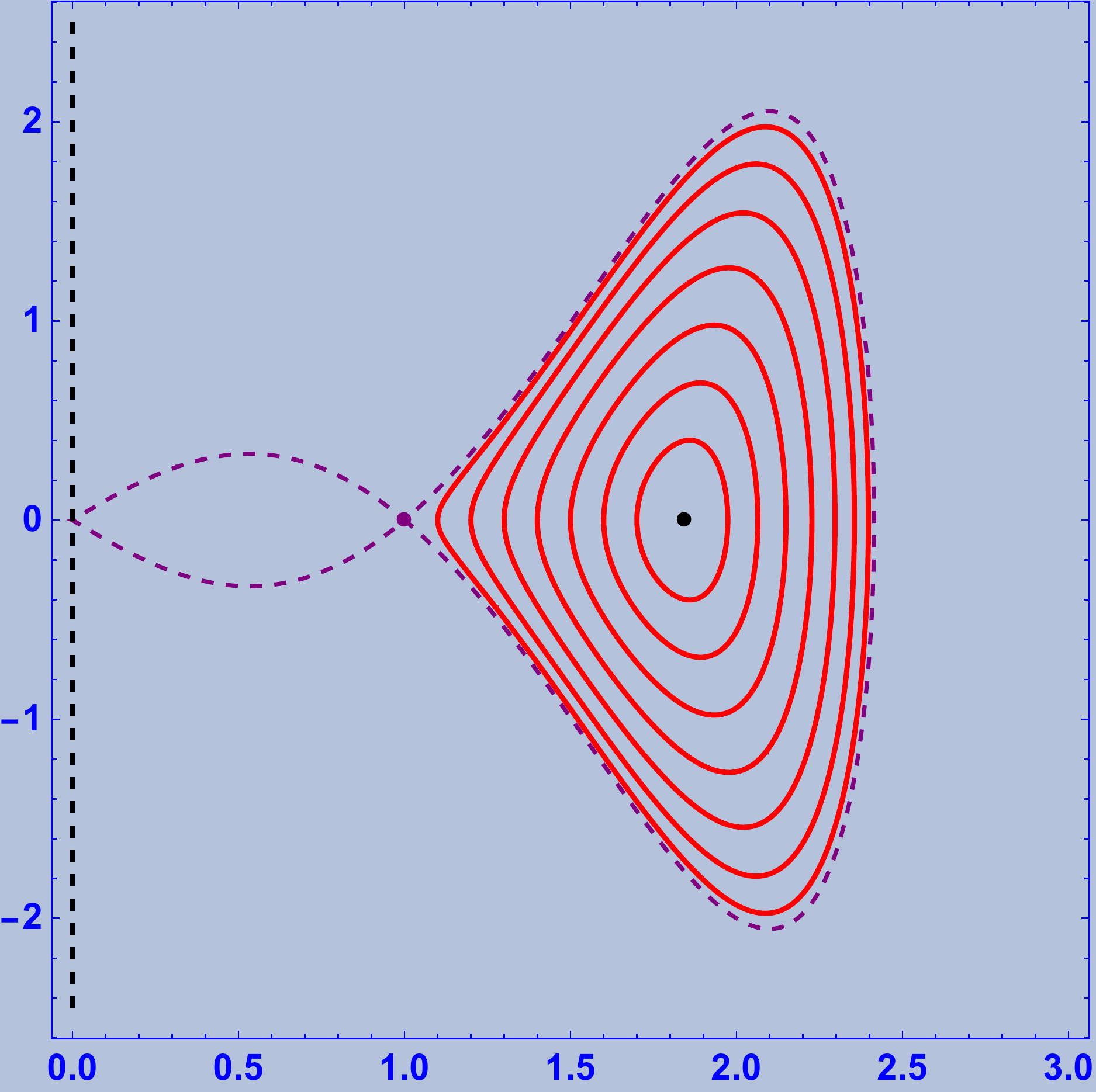}\quad
		\includegraphics[height=4cm,width=4cm]{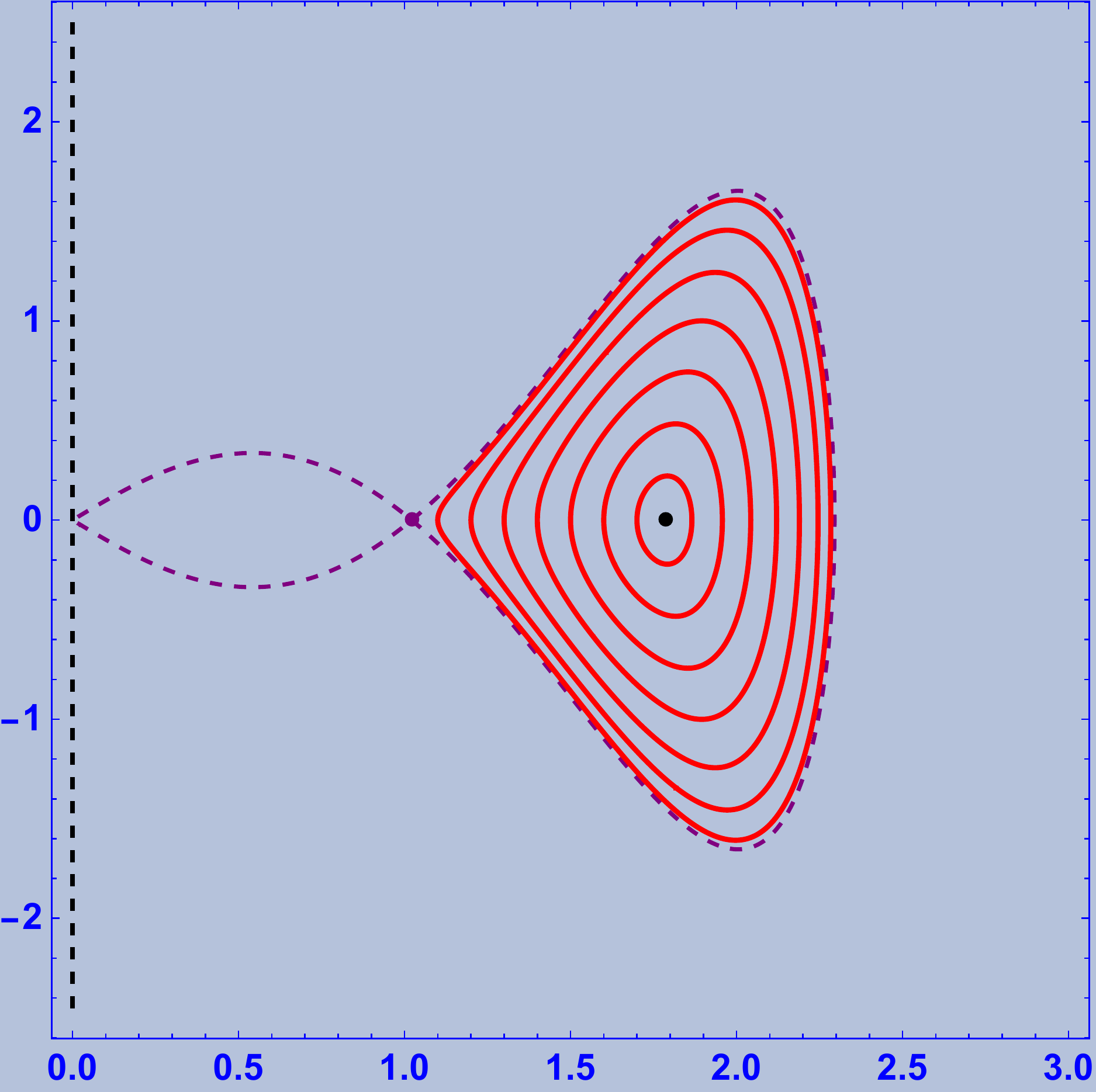}}
	\caption{\small{Left: The modified invariant signatures $\hat{\mathfrak S}_{\lambda,e_2}$, $e_2\in (a(\lambda),\eta_+(\lambda))$ (red) and the signature $\hat{\mathfrak S}_{\lambda,a(\lambda)}$ (dashed-black), for $\lambda=-1.3$. Center: The modified invariant signatures $\hat{\mathfrak S}_{-1,e_2}$, $e_2\in (1,\eta_+(\lambda))$ (red) and the signature $\hat{\mathfrak S}_{-1,1}$ (dashed-black). Right: The modified invariant signatures $\hat{\mathfrak S}_{\lambda,e_2}$, $e_2\in (a(\lambda),\eta_+(\lambda))$ (red) and the signature $\hat{\mathfrak S}_{\lambda,a(\lambda)}$ (dashed-black), for $\lambda=-0.98$.}}\label{FIGFP1}
\end{figure}

We anticipate some properties of the  functions ${\mathcal P}_{\lambda}:e_2\to {\mathcal P}(\lambda,e_2)$ that will be proved in the appendix as part of the proof of the existence of BT-strings.  The purpose is  to highlight the interrelationships between the behavior of the functions  ${\mathcal P}_{\lambda}$ and the geometry of the modified invariant signatures of critical curves with multiplier $\lambda$. Denote by $\hat{\mathfrak S}_{\lambda,e_2}$ the modified invariant signature of a B-curve with modulus $\mathfrak{p}=(\lambda,e_2)\in {\mathfrak P}$. There are three cases to be considered:
\begin{enumerate}  
\item Case $\lambda<-1$. Then $(\lambda, a(\lambda))\in {\mathcal L}$  is the modulus of a BL-curve whose modified invariant signature 
$\hat{\mathfrak S}_{\lambda,a(\lambda)}$ is closed.  The modified invariant signatures $\hat{\mathfrak S}_{\lambda,e_2}$, $e_2\in {\rm I}_{\lambda}$ are contained in the internal region bounded by $\hat{\mathfrak S}_{\lambda,a(\lambda)}$. When $e_2\to a(\lambda)^+$ the signature $\hat{\mathfrak S}_{\lambda,e_2}$ tends to $\hat{\mathfrak S}_{\lambda,a(\lambda)}$ and when $e_2\to \eta_+(\lambda)^-$,  $\hat{\mathfrak S}_{\lambda,e_2}$ contracts, tending to the stable equilibrium point $(\eta_+(\lambda),0)$ (see the picture on the left of Figure \ref{FIGFP1}). The function ${\mathcal P}_{\lambda}$ is strictly increasing (this assertion is only supported by the experimental evidence).  When $e_2\to a(\lambda)^+$,   ${\mathcal P}_{\lambda}$ tends to $1$ and when $e_2\to \eta_+(\lambda)^-$,  ${\mathcal P}_{\lambda}$ tends to $\chi(\lambda)$.
\item Case $\lambda=-1$.  Then $a(-1)=1$ and the signature  $\hat{\mathfrak S}_{-1,1}$  is the unstable equilibrium point $(1,0)$. Let $\hat{\mathfrak S}^*_{-1}$ be the  exceptional modified invariant signature of the first kind.  Then $\hat{\mathfrak S}^*_{-1}\cup\{(1,0)\}$ is a simple (singular) closed curve.  The region bounded by $\hat{\mathfrak S}^*_{-1}\cup\{(1,0)\}$  contains the stable equilibrium point $(\eta_+(-1),0)$.  The signatures $\hat{\mathfrak S}_{-1,e_2}$, $e_2\in {\rm I}_{-1}$ are contained in the internal region bounded by $\hat{\mathfrak S}^*_{-1}\cup\{(1,0)\}$.  When $e_2\to 1^+$,  $\hat{\mathfrak S}_{\lambda,e_2}$ tends to $\hat{\mathfrak S}^*_{-1}\cup\{(1,0)\}$ and, when $e_2\to \eta_+(\lambda)^-$,  $\hat{\mathfrak S}_{\lambda,e_2}$ contracts  to the stable equilibrium point (see the picture on the center of Figure \ref{FIGFP1}). The function ${\mathcal P}_{\lambda}$ is strictly increasing (again, this fact follows from experimental evidence). When $e_2\to 1^+$,  ${\mathcal P}_{\lambda}$ tends to $1$ and when $e_2\to \eta_+(-1)^-$,  ${\mathcal P}_{\lambda}$ tends to $\chi(-1)$.
\item Case $-1<\lambda<-2/\sqrt[4]{27}$. Then $a(\lambda)>1$ and the signature  $\hat{\mathfrak S}_{\lambda,a(\lambda)}$ is the unstable equilibrium point $(a(\lambda),0)$. Let $\hat{\mathfrak S}^*_{\lambda}$ be the exceptional modified invariant signature of the first kind. Then $\hat{\mathfrak S}^*_{\lambda}\cup\{(a(\lambda),0)\}$ is a simple (singular) closed curve. As in the previous case, the region bounded by $\hat{\mathfrak S}^*_{\lambda}\cup\{(a(\lambda),0)\}$ contains the stable equilibrium point and the signatures $\hat{\mathfrak S}_{\lambda,e_2}$, $e_2\in {\rm I}_{\lambda}$.  When $e_2\to a(\lambda)^+$,  $\hat{\mathfrak S}_{\lambda,e_2}$ tends to $\hat{\mathfrak S}^*_{\lambda}\cup\{(1,0)\}$ and, when $e_2\to \eta_+(\lambda)^-$,  $\hat{\mathfrak S}_{\lambda,e_2}$ contracts, to the stable equilibrium point (see the picture on the right of Figure \ref{FIGFP1}). When $e_2\to \eta_+(\lambda)^-$,  ${\mathcal P}_{\lambda}$ tends to $\chi(\lambda)$ and ${\mathcal P}_{\lambda}=O\left(\log(4/\sqrt{e_2-a(\lambda)})\right)$ as $e_2\to a(\lambda)^+$.
\end{enumerate}

\begin{remark} \emph{Unlike in the other cases, if $\lambda>-1$, the function ${\mathcal P}_{\lambda}$ is not necessarily monotonous. Experimental evidence suggests the existence of a value of $\lambda$ $(\lambda _*  \approx -0.98147772)$ such that ${\mathcal P}_{\lambda}$ has a unique minimum point in the interval $(a(\lambda),\eta_+(\lambda))$, for every $\lambda \in (-1,\lambda_*)$ and is strictly decreasing, for every $\lambda\in [\lambda_*,-2/\sqrt[4]{27})$ (see Figure \ref{FIGFP3}). This is a huge difference with the spherical case, for which experiments suggest that $\mathcal{P}_\lambda$ is always monotonous \cite{MP}.}
\end{remark}

\begin{figure}[h]
	\makebox[\textwidth][c]{
		\includegraphics[height=4cm,width=4cm]{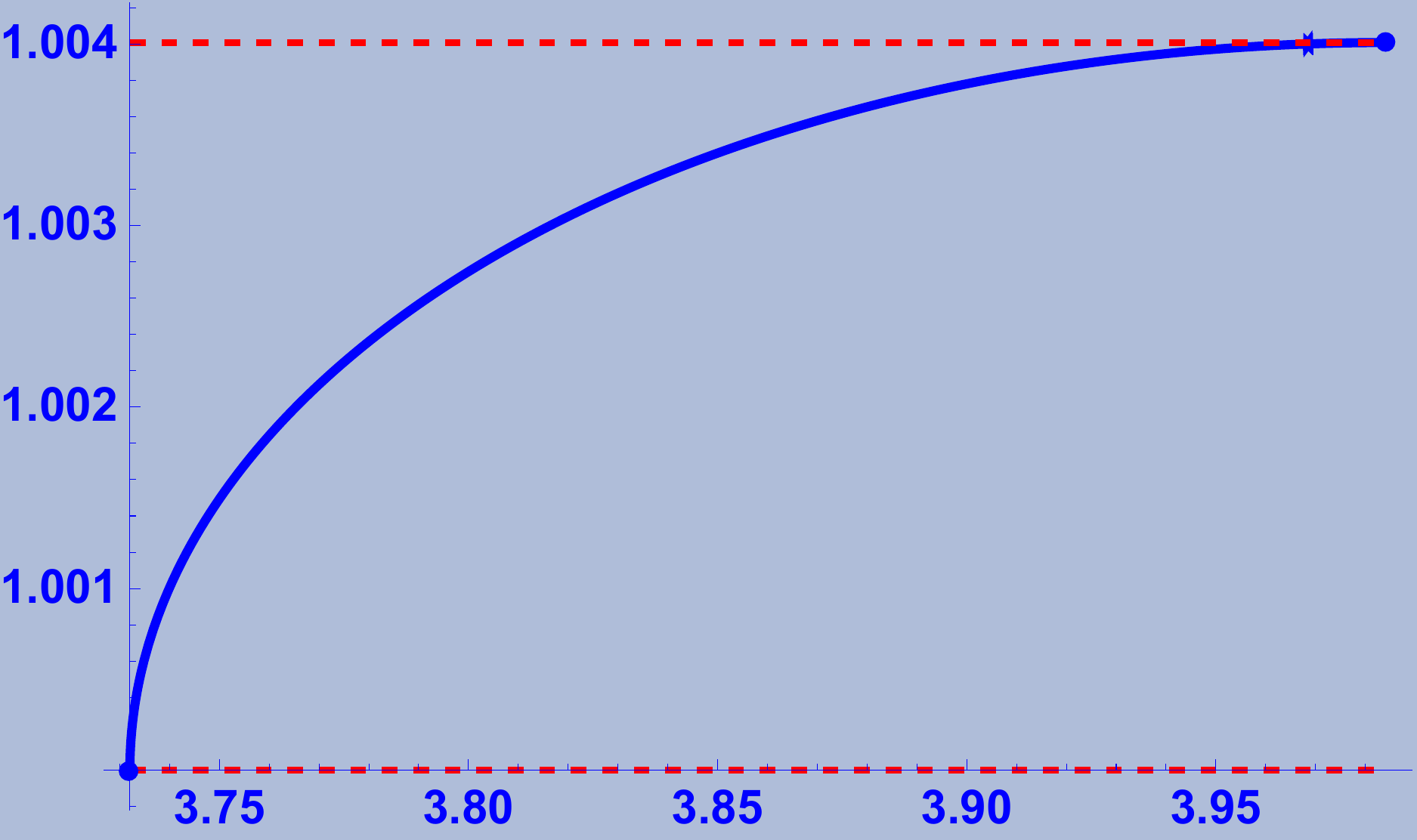}\quad
		\includegraphics[height=4cm,width=4cm]{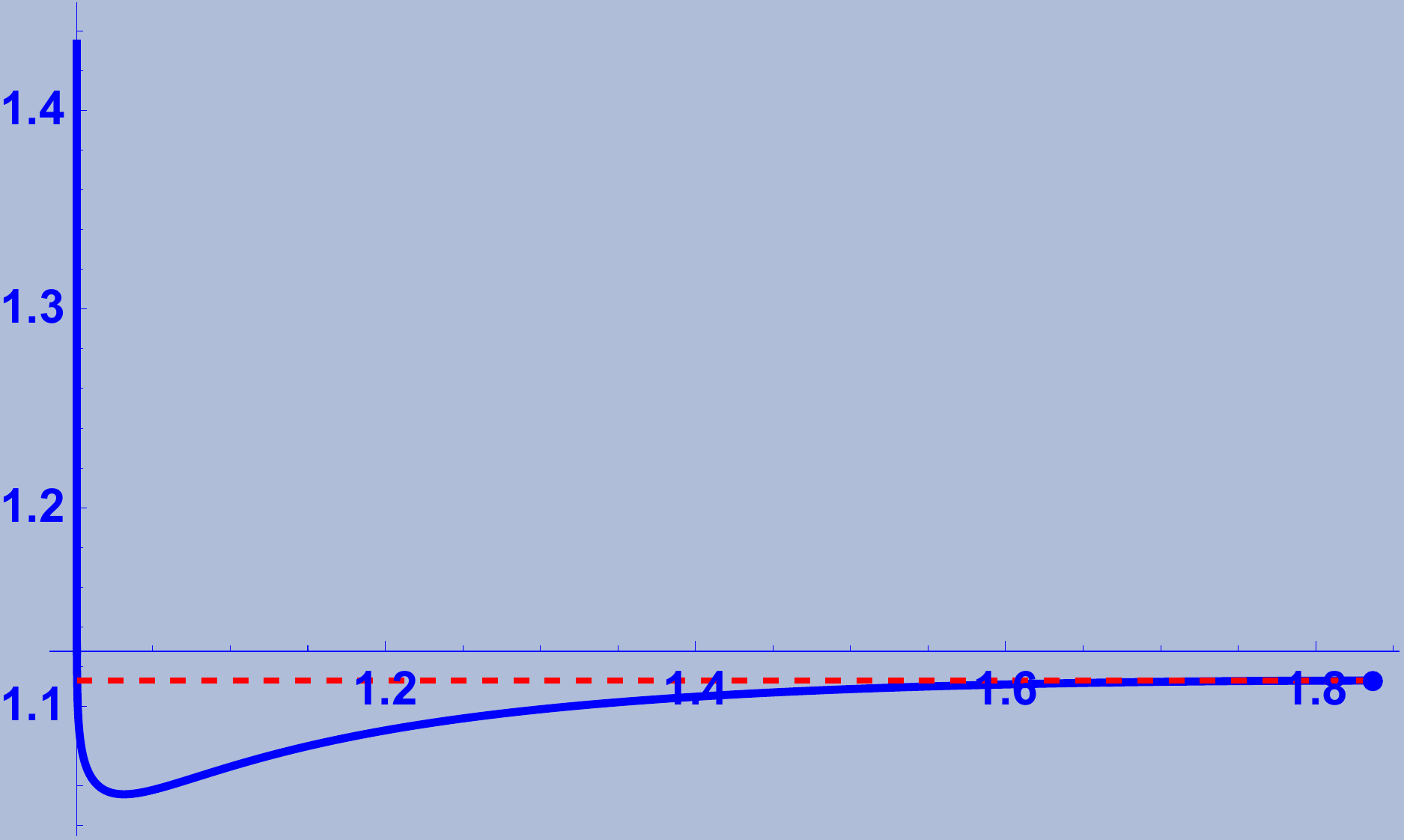}\quad
		\includegraphics[height=4cm,width=4cm]{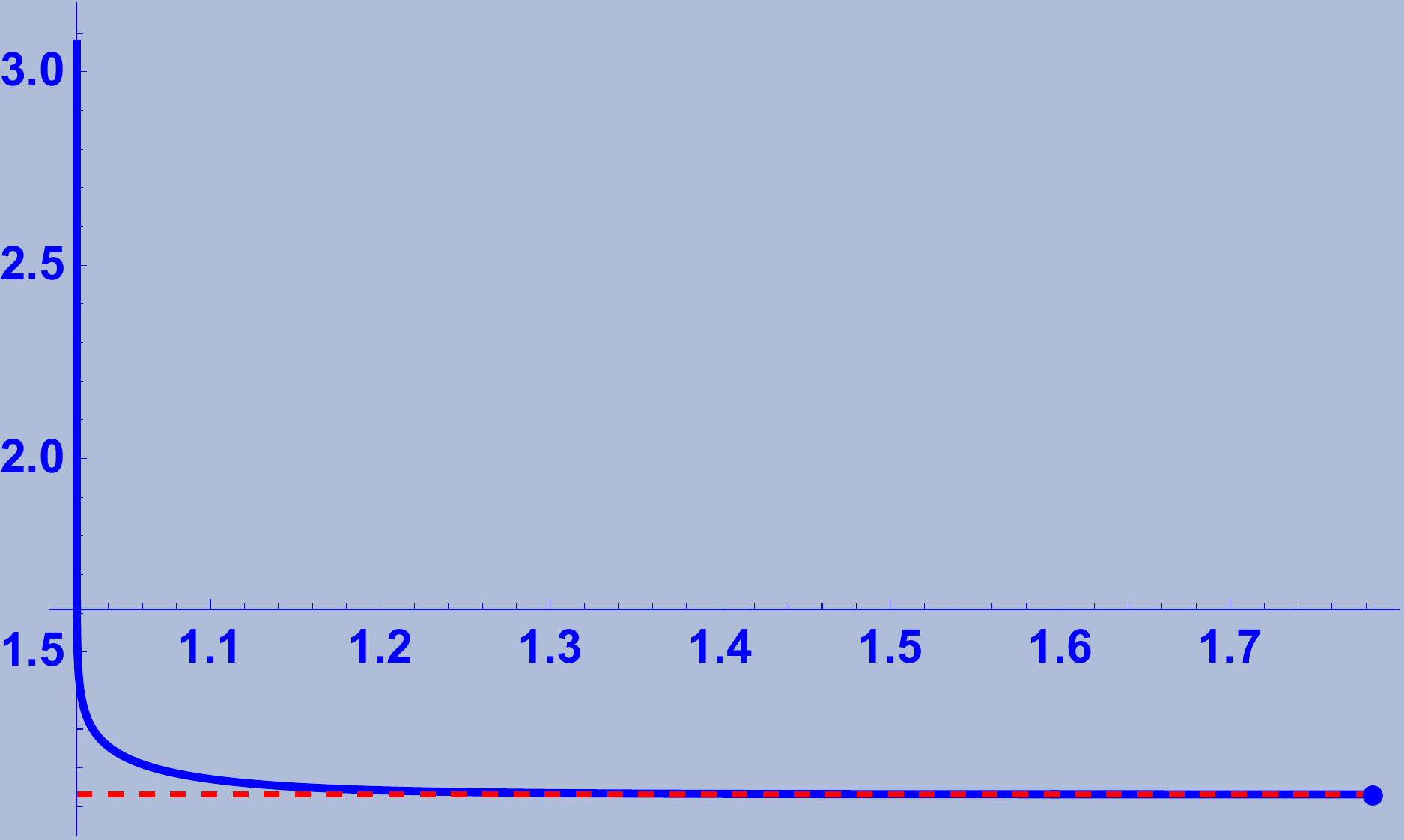}}
	\caption{\small{The graph of the functions $\mathcal{P}_\lambda$. From left to right: $\lambda=-1.3$, $-0.999$, and $-0.98$.}}\label{FIGFP3}
\end{figure}

\subsection{The Fibers of the Period Map} 

Consider the vector field $\vec{P}:\mathfrak{p}=(\lambda,e_2)\in\mathcal{T}\to (\partial_{e_2}\mathcal{P},-\partial_\lambda\mathcal{P})$, where $\mathcal{P}$ is the period map (for a plot see the picture on the left of Figure \ref{FIG21}). Based on experimental evidence, this vector field never vanishes and its second component is negative, ie. the period map $\mathcal{P}$ increases with $\lambda$. However, as highlighted in the previous remark, $\mathcal{P}$ may not be monotonous with respect to $e_2$. One can then deduce the following consequences:
\begin{itemize}
\item The fibers of the map ${\mathcal P}:{\mathcal T}\to \R$ are the integral curves of ${\vec P}$.
\item For every $q>1$ the fiber ${\mathcal T}_q={\mathcal P}^{-1}(q)$  is a simple open arc with boundary points ${\mathfrak p}_0=(-1,1)$ 
and ${\mathfrak p}_q=(\lambda^*_q,e^*_{q,2})$, where 
$$\lambda_q^{*}=-\frac{1+(e^{*}_{q,2})^4}{2(e^{*}_{q,2})^3}\,,\quad\quad   \frac{(e^{*}_{q,2})^4-1}{\sqrt{(e^{*}_{q,2})^8-4(e^{*}_{q,2})^4+3}}=q\,.$$
More precisely,  there exists a smooth function $\lambda _q:\tilde{{\rm J}}_q\to \R$,  defined on the open interval $\tilde{{\rm J}}_q=(1,e^{*}_{q,2})$ such that
${\frak p}_q:e_2\in \tilde{{\rm J}}_q\to (\lambda_q(e_2),e_2)$ is a parameterization of ${\mathcal T}_q$.  Moreover, $\lambda_q(1)=-1$ and 
$\lambda_q(e^{*}_{q,2})=-\lambda^*_q$ (see Figure \ref{FIG21}).
\item The fiber ${\mathcal T}_q$ intersects transversely the exceptional locus at a unique point $ \hat{\mathfrak p}_{q}=(\hat{\lambda}_2,\hat{e}_{q,2})$ (see Figure \ref{FIG21}, right).
\item When $q\to 1^+$, the fiber  ${\mathcal T}_q$ tends to ${\mathcal L}$.  If $q\to +\infty$, the fiber  ${\mathcal T}_q$ tends to the arc 
$\eta_-[(-1,-2/ \sqrt[4]{27})]$ (the green arc depicted on the right of Figure \ref{FIG21}).
\end{itemize}

\begin{figure}[h]
	\makebox[\textwidth][c]{
		\includegraphics[height=4cm,width=4cm]{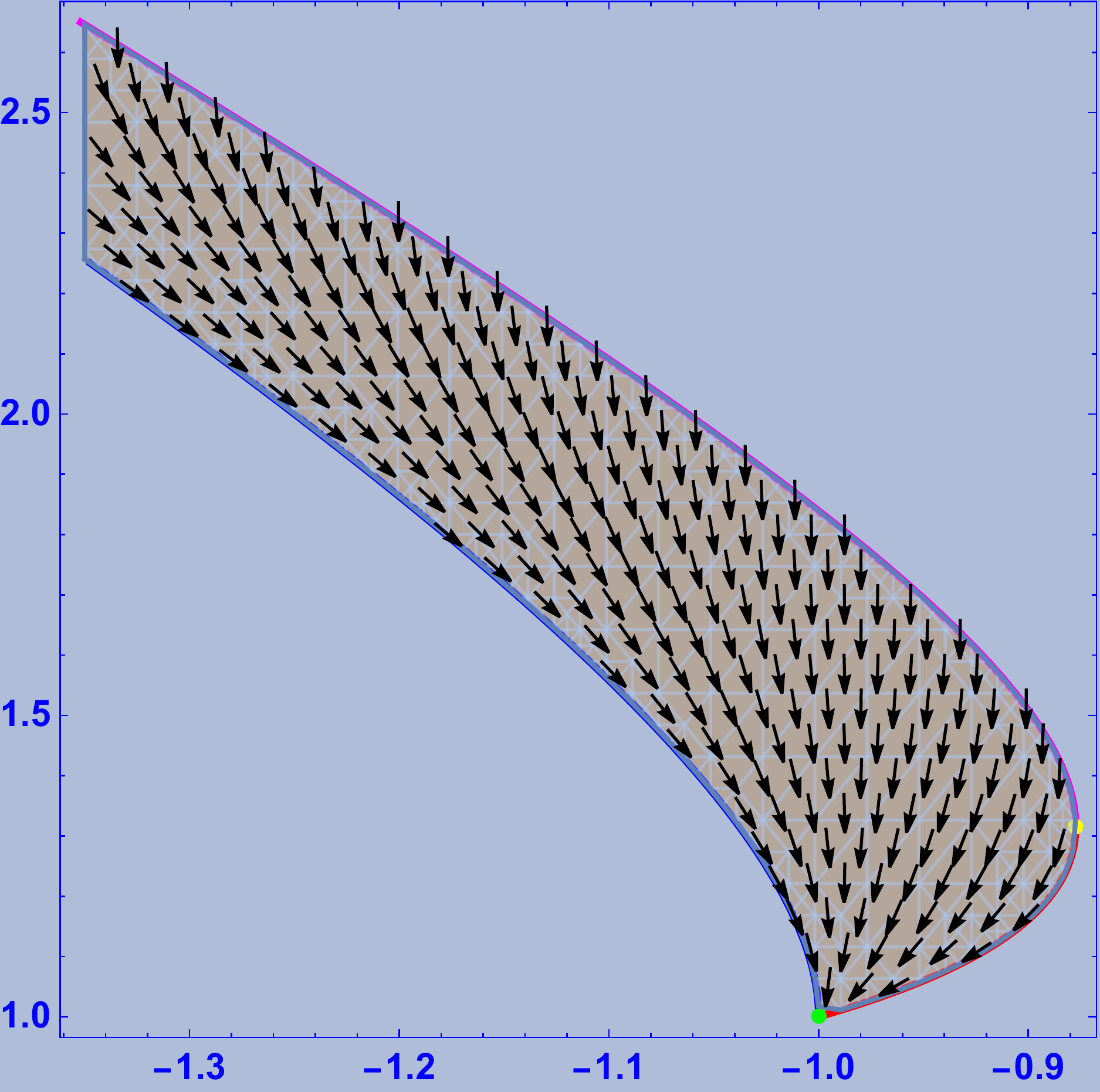}\quad
		\includegraphics[height=4cm,width=4cm]{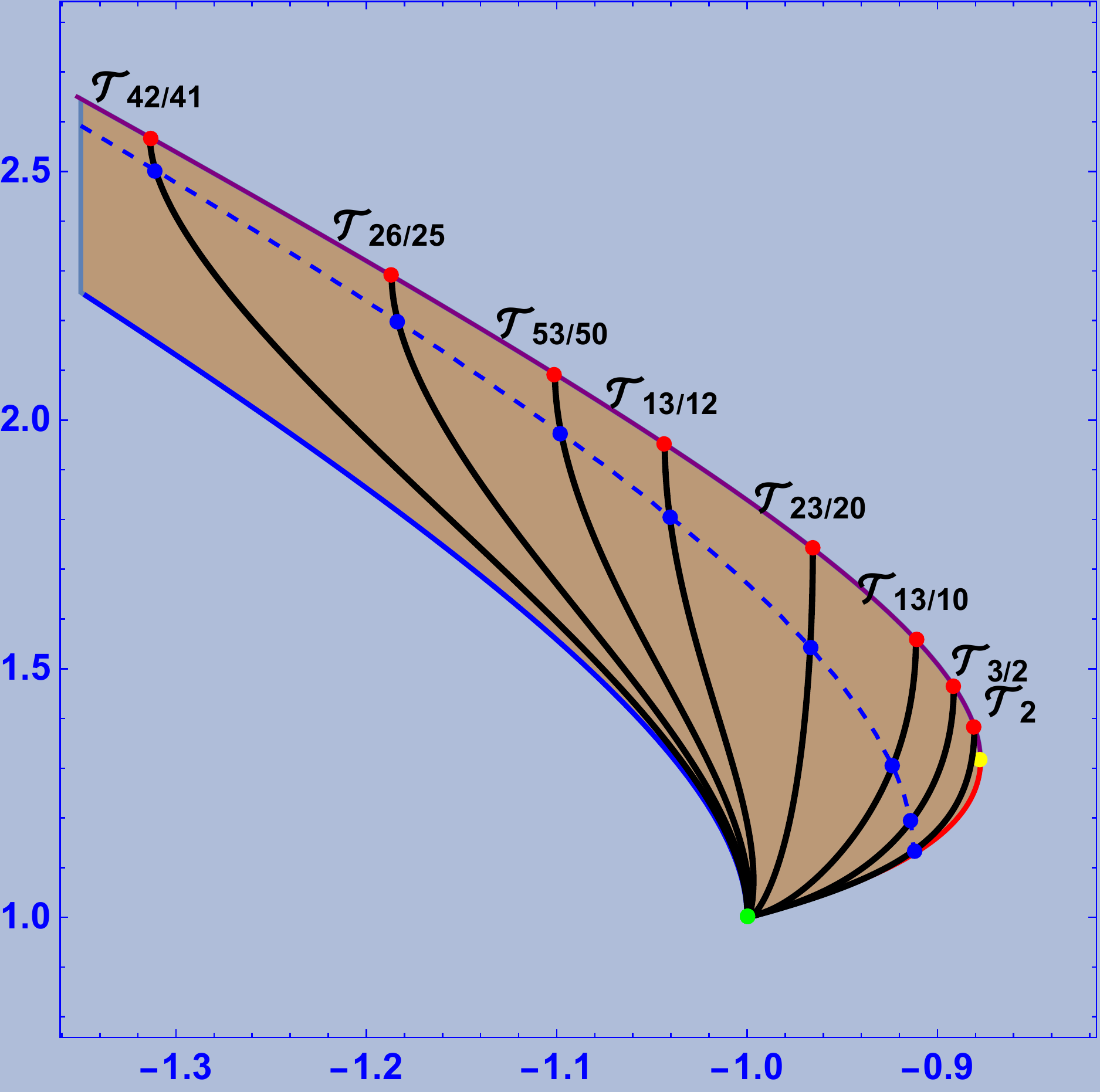}\quad
		\includegraphics[height=4cm,width=4cm]{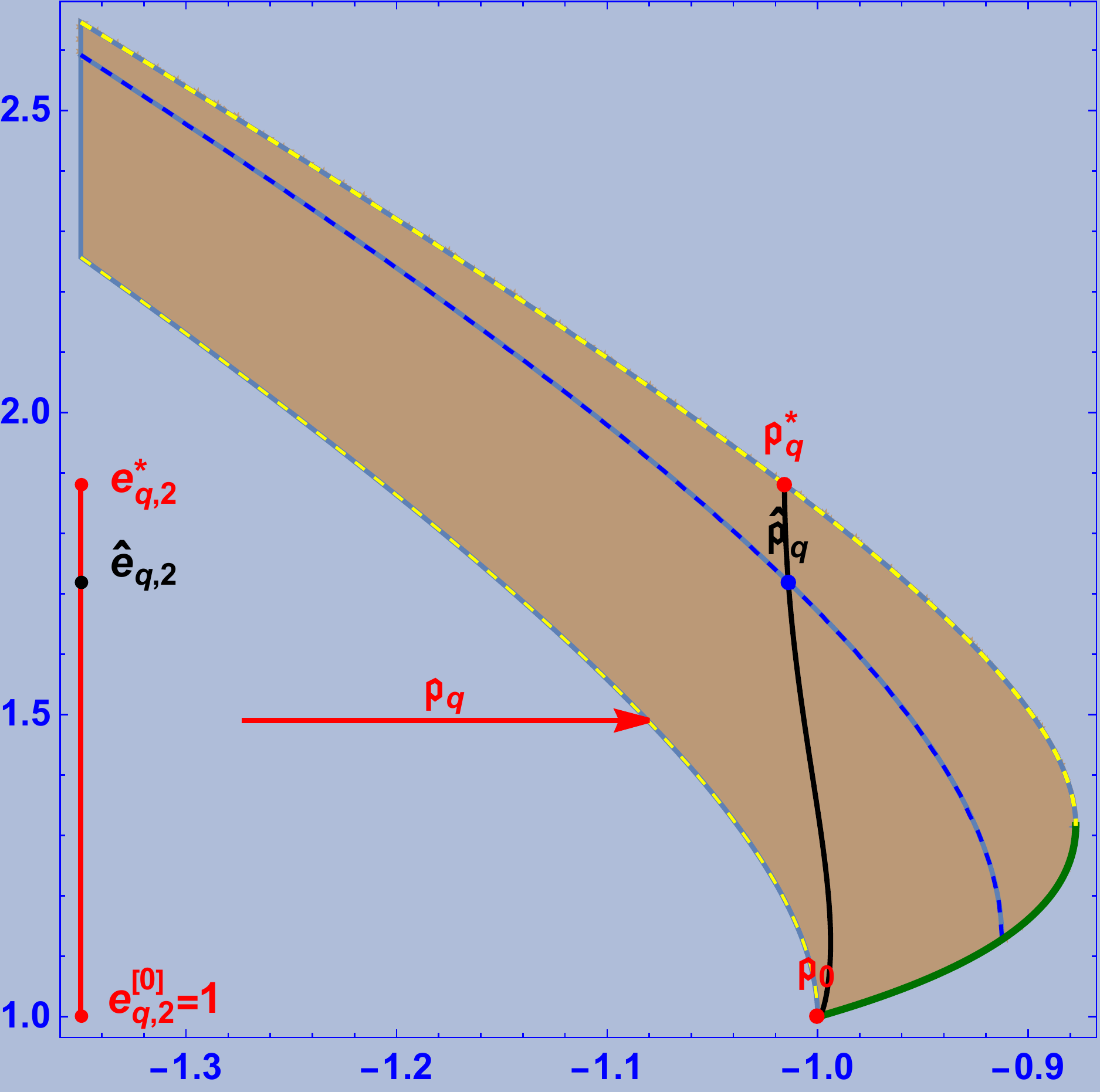}}
	\caption{\small{Left: The plot of the vector field $\vec{P}=(\partial_{e_2}{\mathcal P},-\partial_{\lambda}{\mathcal P})$. Center: The fibers ${\mathcal T}_q$, for several values of $q$. Right: The fiber ${\mathcal T}_{11/10}$.}}\label{FIG21}
\end{figure}

\subsection{Isomonodromic Families} \label{isomfam} Let $q=m/n$ be a rational number.  Then, ${\mathcal G}_q=\{\gamma_{\mathfrak p}\}_{{\mathfrak p}\in {\mathcal T}_q}$ is a $1$-parameter family of BT-strings,  the \emph{isomonodromic family} with characteristic number $q$. We list some salient geometric properties  of the  isomonodromic families:
\begin{enumerate}
\item If $e_2\to 1^{+}$,  $|[\gamma_{{\mathfrak p}_q(e_2)}]|$  tends to the ideal boundary of ${\rm D}^2$ (see the first image of Figure \ref{FIG22}).
\item If $e_2\to e^{*}_{q,2}\,^{-}$,  $|[\gamma_{{\mathfrak p}_q(e_2)}]|$ tends to the circle ${\mathcal C}_q$, centered at the origin, with radius 
$$r_q=\frac{1}{e^{*}_{q,2}+\sqrt{(e^{*}_{q,2})^4-1}}\,.$$
(See the last image of Figure \ref{FIG22}.)
\item The trajectory $|[\gamma_{{\mathfrak p}_q(e_2)}]|$ passes through the origin if and only if $e_2=\hat{e}_{q,2}$ (see the curve depicted on the fourth image of Figure \ref{FIG22}).
\item The BT-strings are never simple, however their multiple points are admissible in the sense that the tangent vectors at those points are not equal (for details see Figure \ref{FIG23}).
\item The BT-strings of the isomonodromic family ${\mathcal G}_q$ have the same stabilizer, namely, the group ${\rm R}_n$ of order $n$ generated by the rotation of an angle $2\pi m/n$ around the origin. Since $\kappa$ is even, the strings are also invariant by the reflection with respect to the horizontal axis. If $n=1$, the reflection is the unique residual symmetry. In particular, $n$ is the wave number of the string of the isomonodromic family.
\item The number $m$ is the hyperbolic turning number of the strings belonging to ${\mathcal G}_q$. In other words, $m$ is the homotopy class of the Frenet frame in $\O(1,2)$, identified as ${\rm D}^2\times\mathbb{S}^1$ \cite{Ch,Re}. If $e_2\in (1,\hat{e}_{q,2})$, $m-n$ is the homotopy class of  $\gamma_{{\mathfrak p}_q(e_2)}$ in the punctured disk $\dot{{\rm D}}^2={\rm D}^2\setminus \{(0,0)\}$  (with the obvious identification $\pi_1(\dot{{\rm D}}^2)\cong \Z$). If  $e_2\in (\hat{e}_{q,2},e^{*}_{q,2})$,  $m$ is the homotopy class of $\gamma_{{\mathfrak p}_q(e_2)}$ in $\dot{{\rm D}}^2$.
\item The function  ${\omega}_q : e_2\in \tilde{{\rm J}}_q  \to \omega_{\gamma_{{\mathfrak p}_q(e_2)}}\in \R$ is strictly decreasing and satisfies
$$\lim_{e_2\to 1^+ }\omega_q(e_2) =+\infty\,,\quad\quad  \lim_{e_2\to e^{*}_{q,2}\,^-}\omega_q(e_2)=\frac{4q\pi r_q}{1-r_q^2}\,.$$
\end{enumerate}

\begin{remark} \emph{From (5), (6) and (7) it follows that BT-strings are characterized by three global geometric invariants: the wave number (ie. the order of the symmetry group),  the hyperbolic turning number and the length.}
\end{remark}

\begin{figure}[h]
	\makebox[\textwidth][c]{
		\includegraphics[height=4cm,width=4cm]{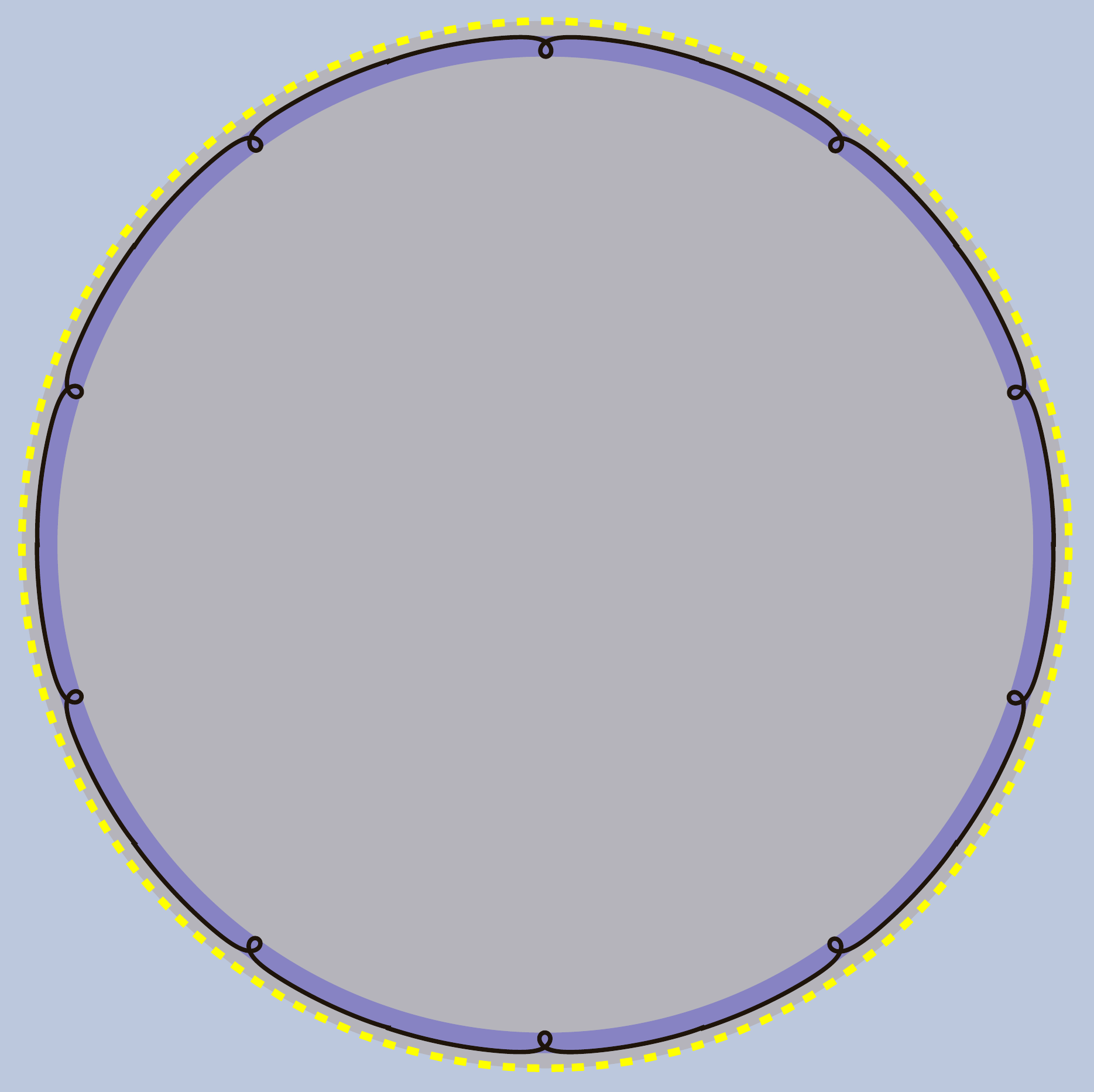}\quad
		\includegraphics[height=4cm,width=4cm]{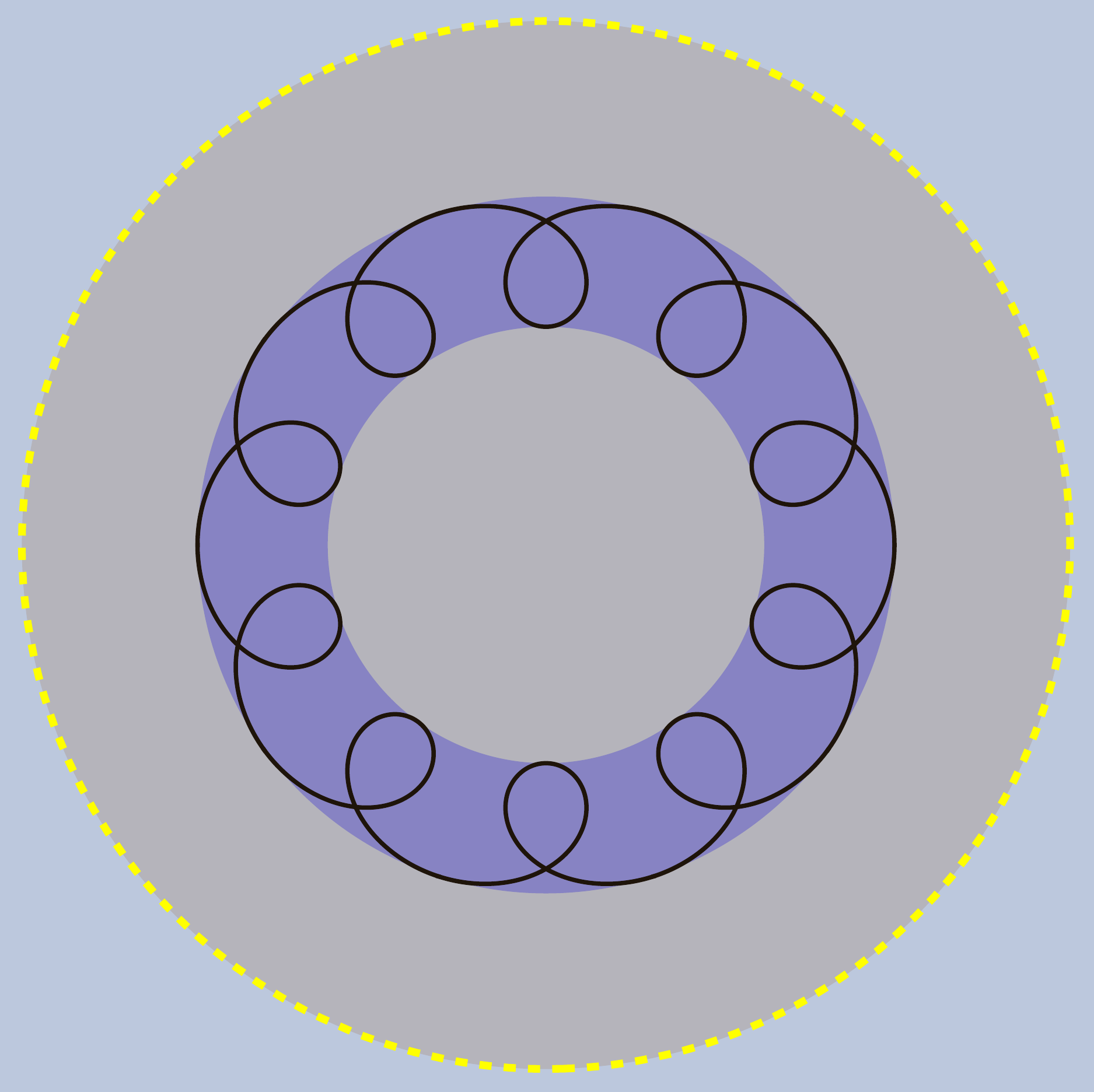}\quad
		\includegraphics[height=4cm,width=4cm]{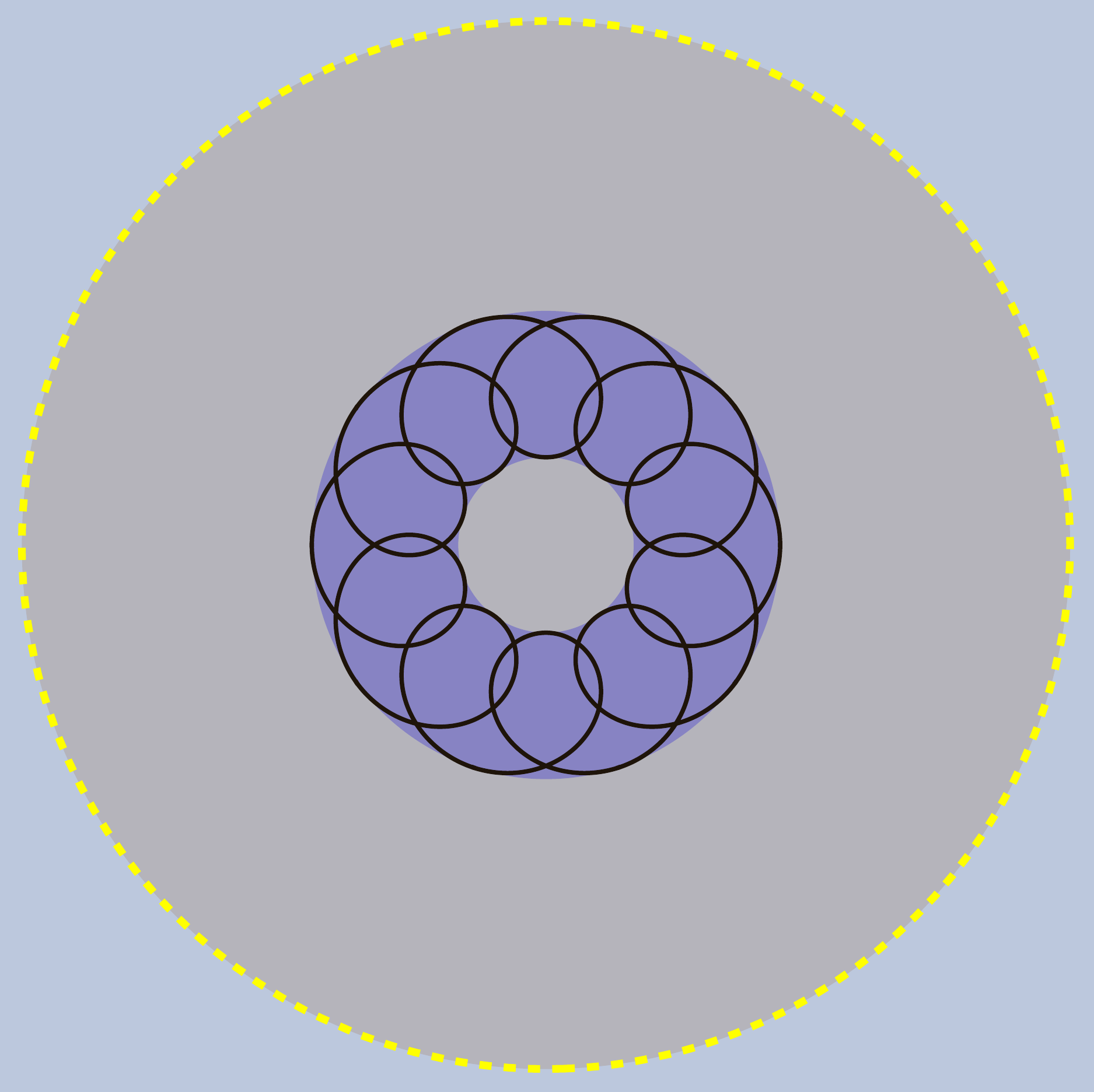}}\\\vspace{0.3cm}
	\makebox[\textwidth][c]{
		\hspace{0.01cm}	\includegraphics[height=4cm,width=4cm]{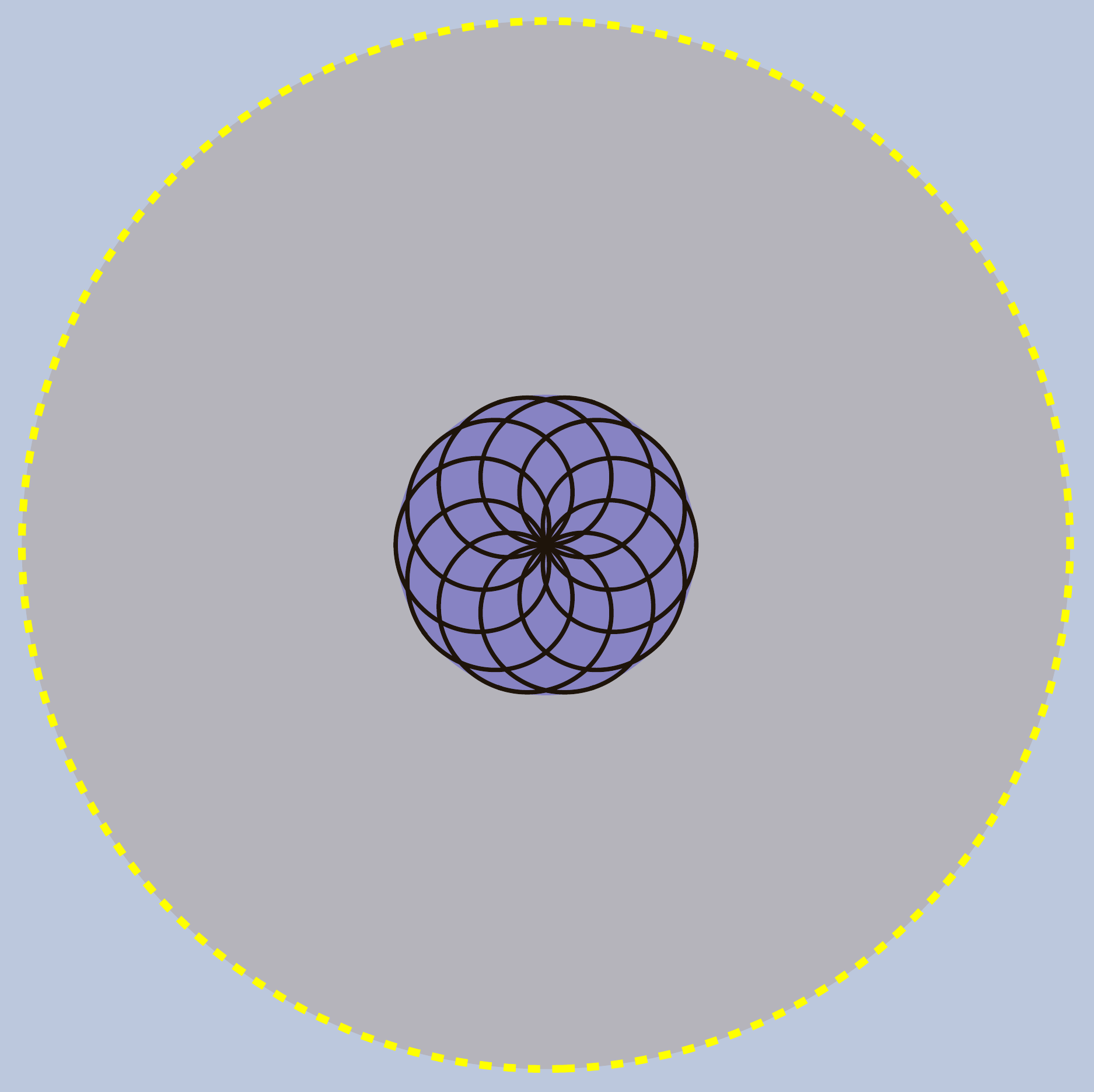}\quad
		\includegraphics[height=4cm,width=4cm]{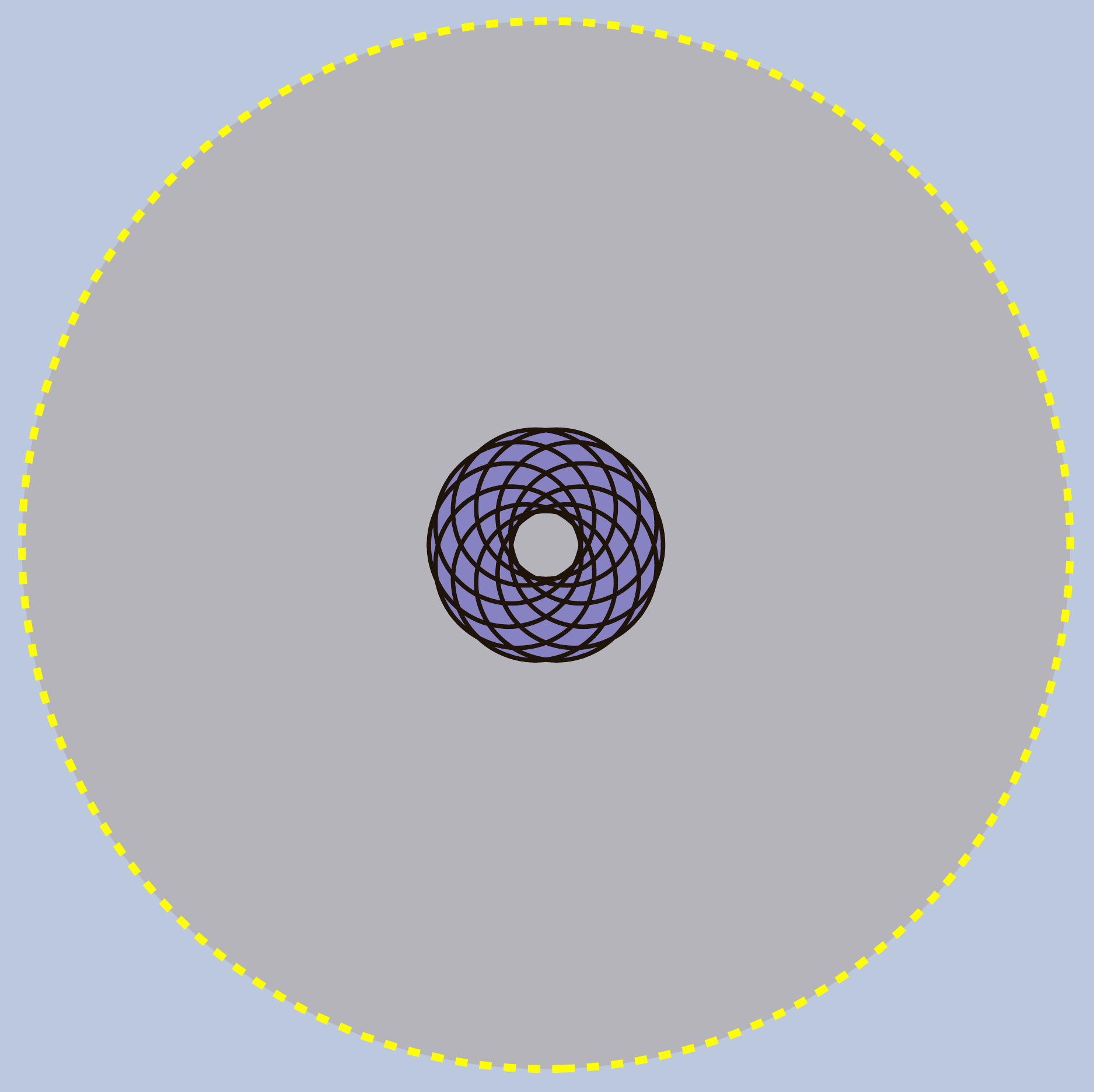}\quad
		\includegraphics[height=4cm,width=4cm]{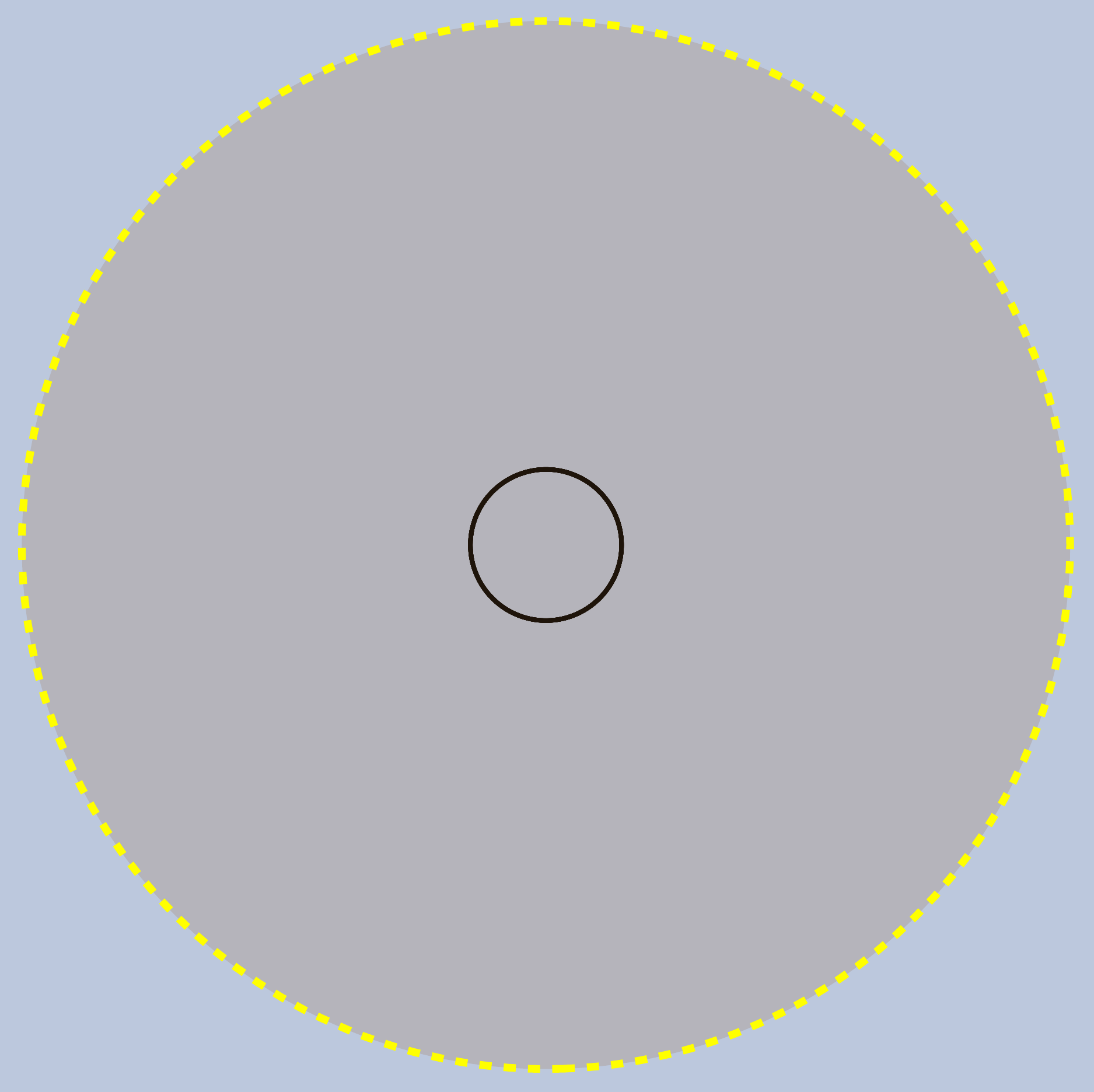}
	}
	\caption{\small{Trajectories of the BT-strings  $|[\gamma_{{\mathfrak p}_q(e_2)}]|$ with $q=11/10$ and $e_2$ increasing in $\tilde{{\rm J}}_{11/10}$.}}\label{FIG22}
\end{figure}

Figure \ref{FIG22} reproduces the trajectories of BT-strings with characteristic number $q=11/10$. The parameter $e_2$ varies in the interval $\tilde{{\rm J}}_{11/10}=(1,e^*_{11/10,2})$, where $e^*_{11/10,2}\approx 1.8812$.  The exceptional value $\hat{e}_{11/10,2}$ is approximately $1.71966$. The curves of Figure \ref{FIG22} correspond to $e_2\approx 1$, $e_2=1.26$, $e_2=1.53$, $e_2=\hat{e}_{11/10,2}$, $e_2=1.79$, respectively, while the last one is the circle $\mathcal{C}_{11/10}$. The BT-strings of the family have a tenth-fold rotational symmetry. If $e_2<\hat{e}_{11/10,2}$ the homotopy class in the punctured disk is $1$ while, if  $e_2>\hat{e}_{11/10,2}$ the homotopy class is $11$.

\begin{remark} \emph{As $e_2$ increases in the interval $\tilde{\rm J}_q$, the curves in the isomonodromic family ${\mathcal G}_q$ pass through different isotopy classes. More precisely, let $j[q]$, $q=m/n$, be defined by 
	\begin{equation*}
		j[q]=(n+{\rm mod}(n,2))/2-\delta_{1,n}\,,
	\end{equation*}
where $\delta_{1,n}$ is the Kronecker delta. There are $2 j[q]+1$ distinct isototopy classes of BT-strings in the isomonodromic family ${\mathcal G}_q$.}
\end{remark}

\begin{figure}[h]
	\makebox[\textwidth][c]{
		\includegraphics[height=4cm,width=4cm]{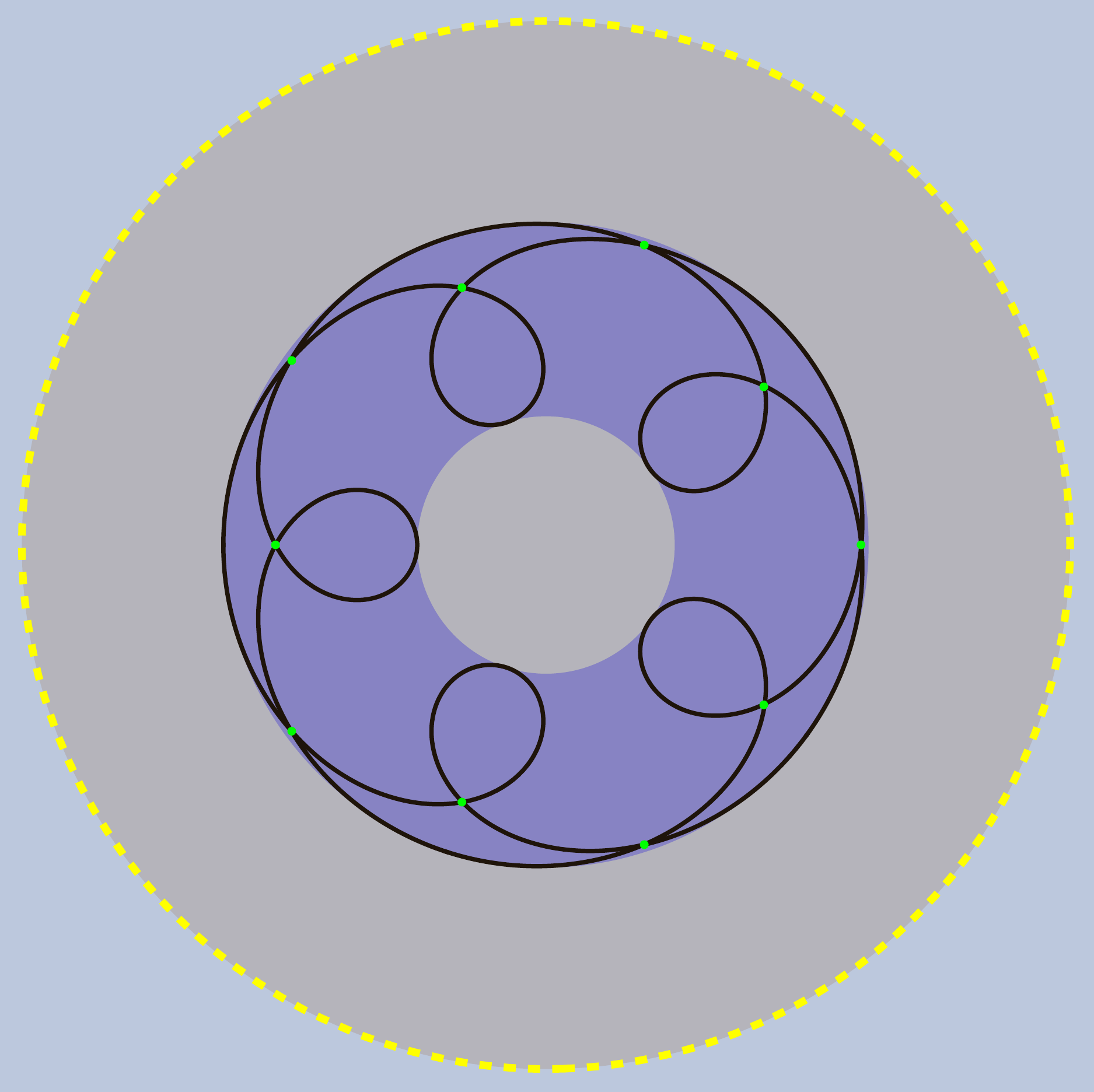}\quad
		\includegraphics[height=4cm,width=4cm]{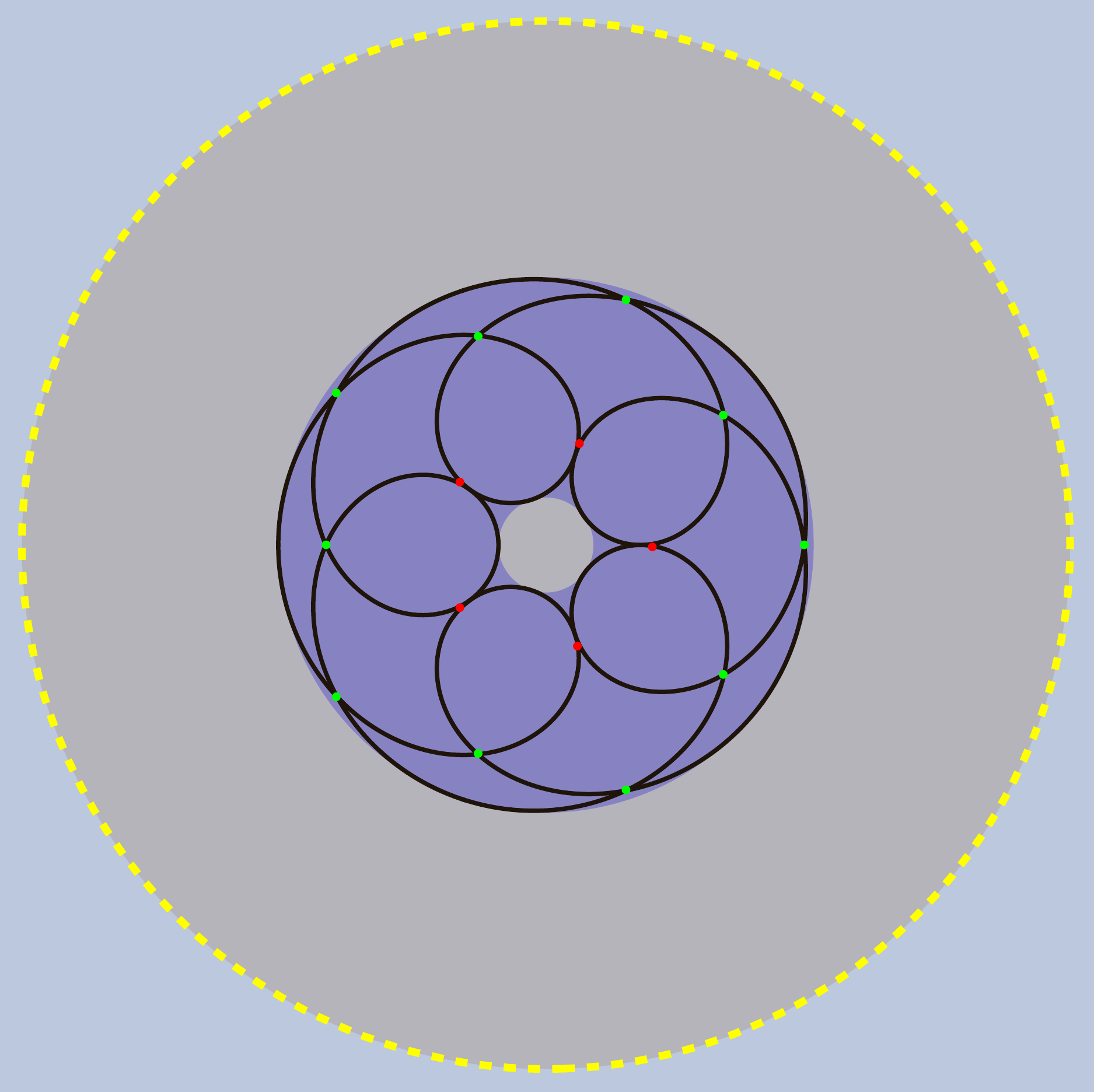}\quad
		\includegraphics[height=4cm,width=4cm]{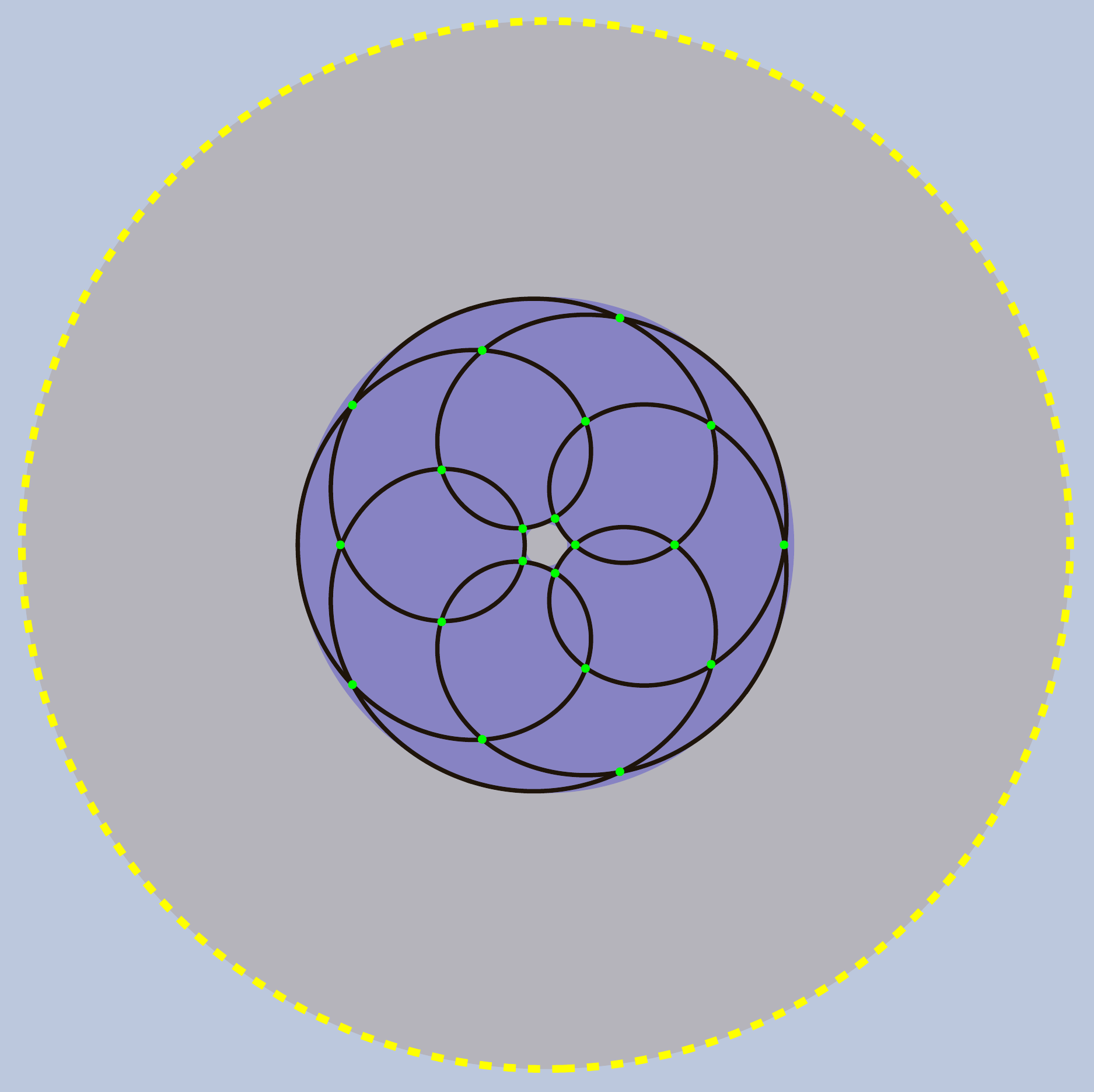}\quad
		\includegraphics[height=4cm,width=4cm]{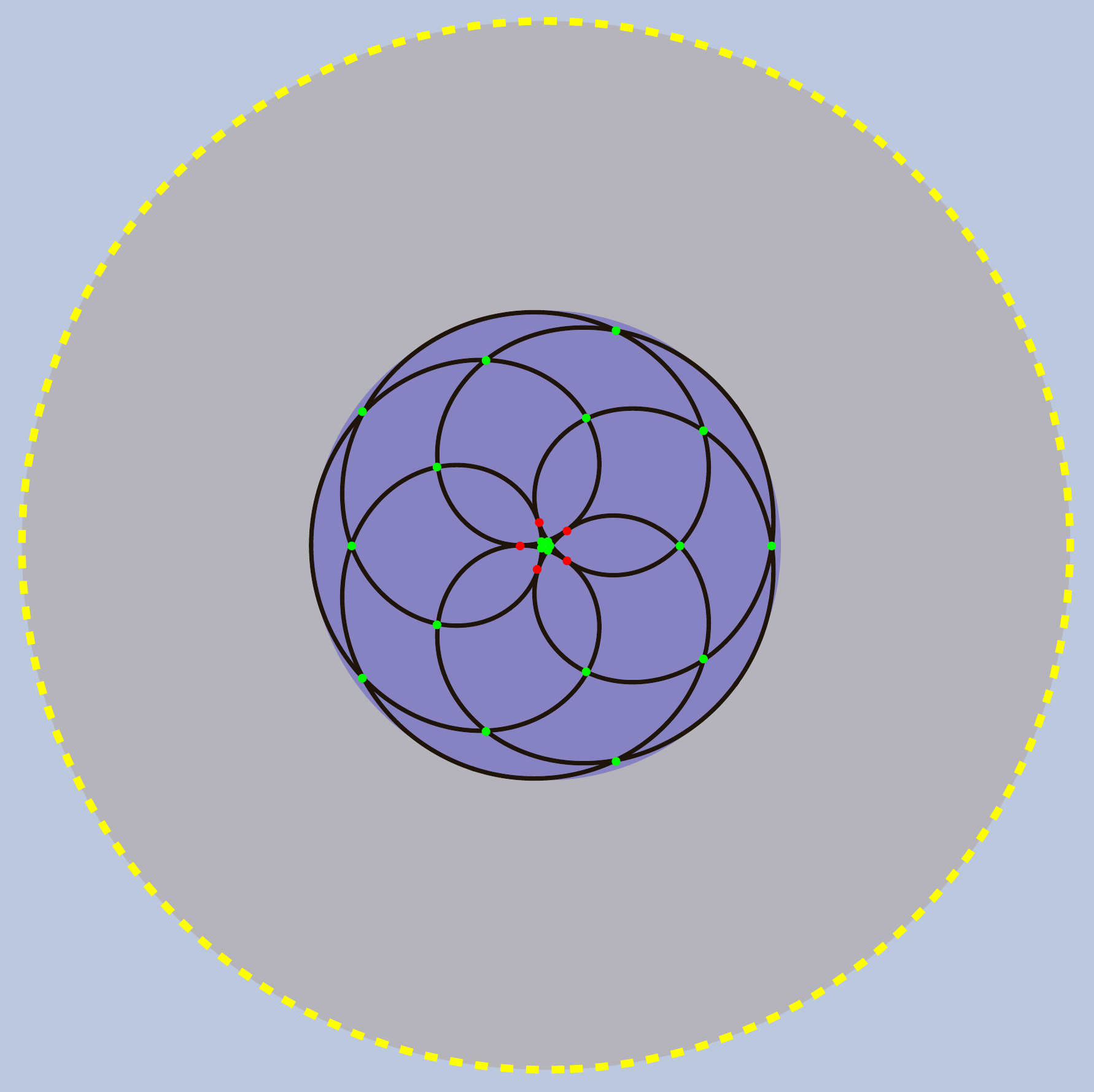}}\\\vspace{0.3cm}
	\makebox[\textwidth][c]{
		\includegraphics[height=4cm,width=4cm]{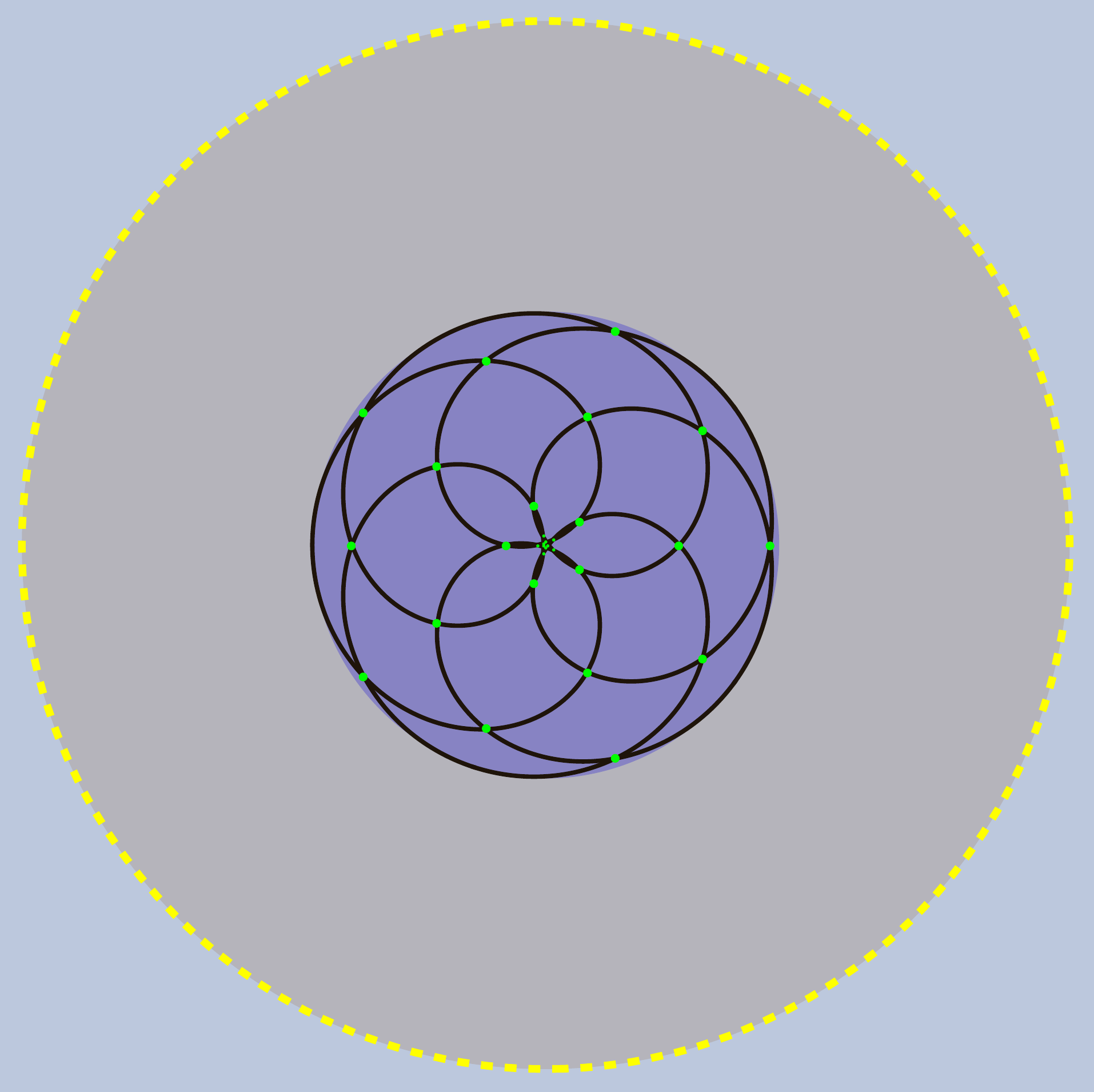}\quad
		\includegraphics[height=4cm,width=4cm]{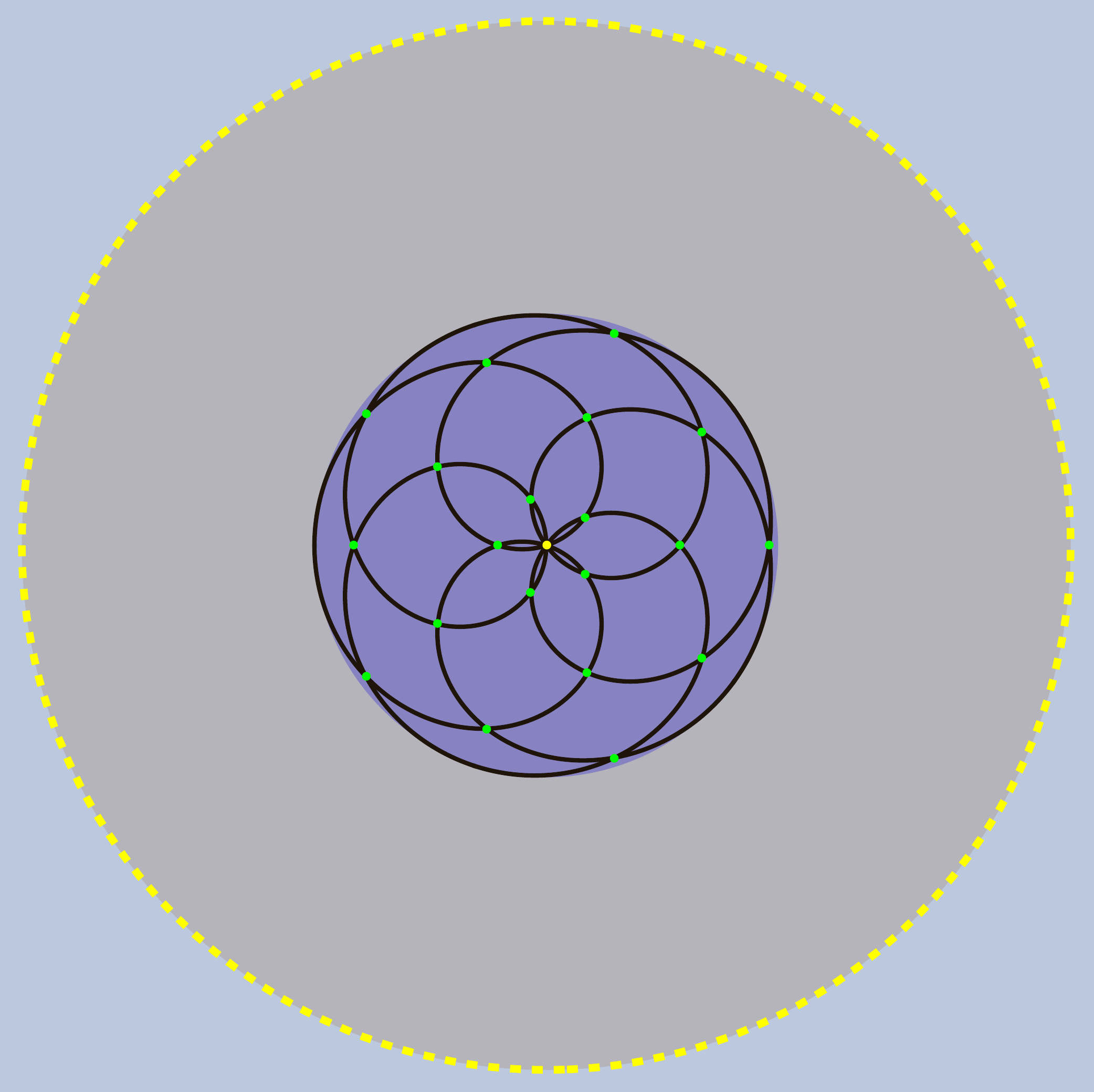}\quad
		\includegraphics[height=4cm,width=4cm]{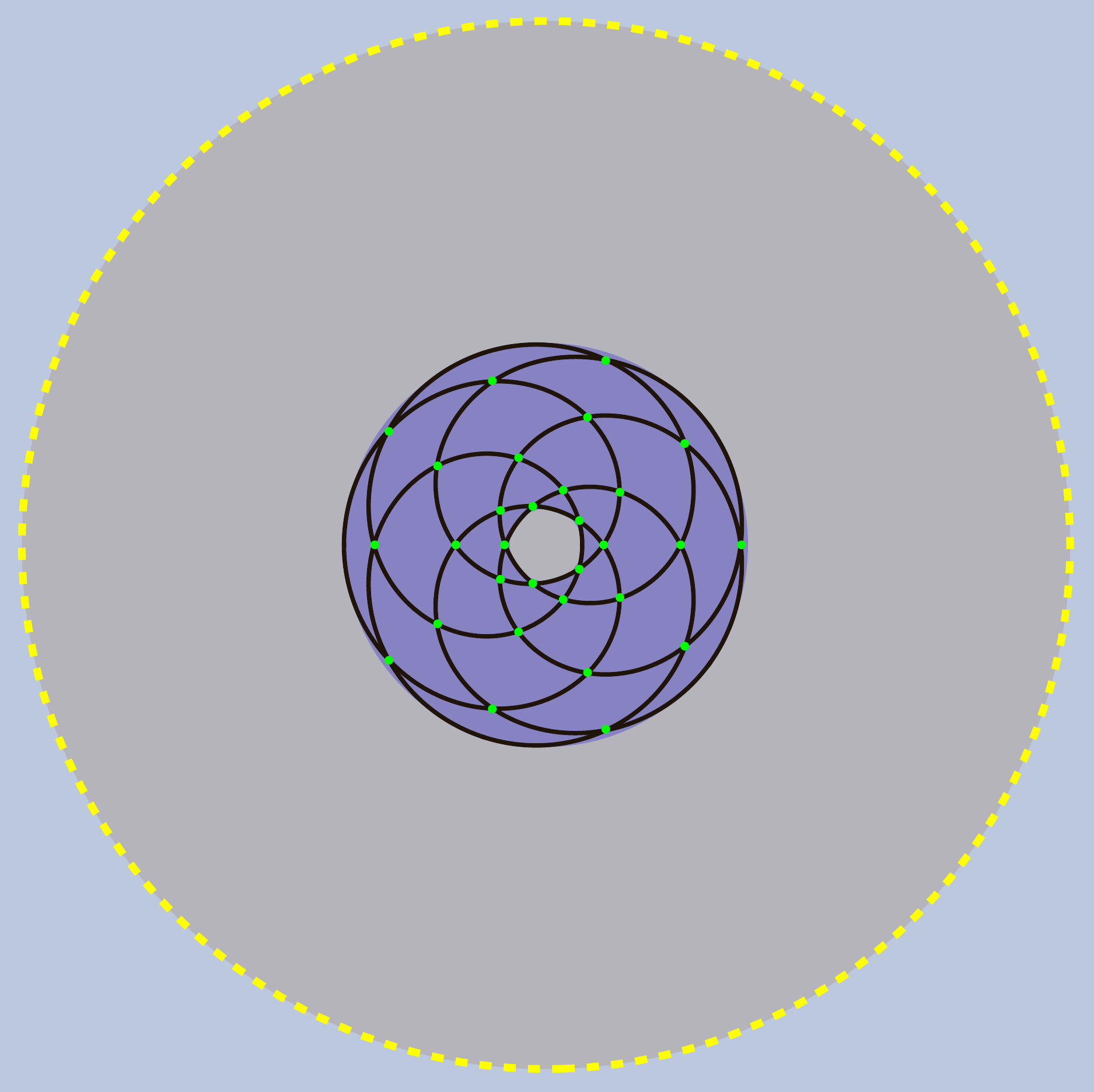}}
	\caption{\small{Trajectories of the BT-strings  $|[\gamma_{{\mathfrak p}_q(e_2)}]|$ with $q=7/5$ and $e_2$ increasing in $\tilde{{\rm J}}_{7/5}$. Ordinary double points are represented by the green dots, tangential admissible double points are the red dots and the multiple point of multiplicity bigger than two is shown in yellow.}}\label{FIG23}
\end{figure}  

The curves depicted in Figure  \ref{FIG23} are the representatives of the seven isotopy classes of strings in the isomonodromic family ${\mathcal G}_{7/5}$. In the same figure, the multiple points of the BT-strings are also illustrated.

\begin{figure}[h]
	\begin{center}
		\includegraphics[height=4cm,width=4cm]{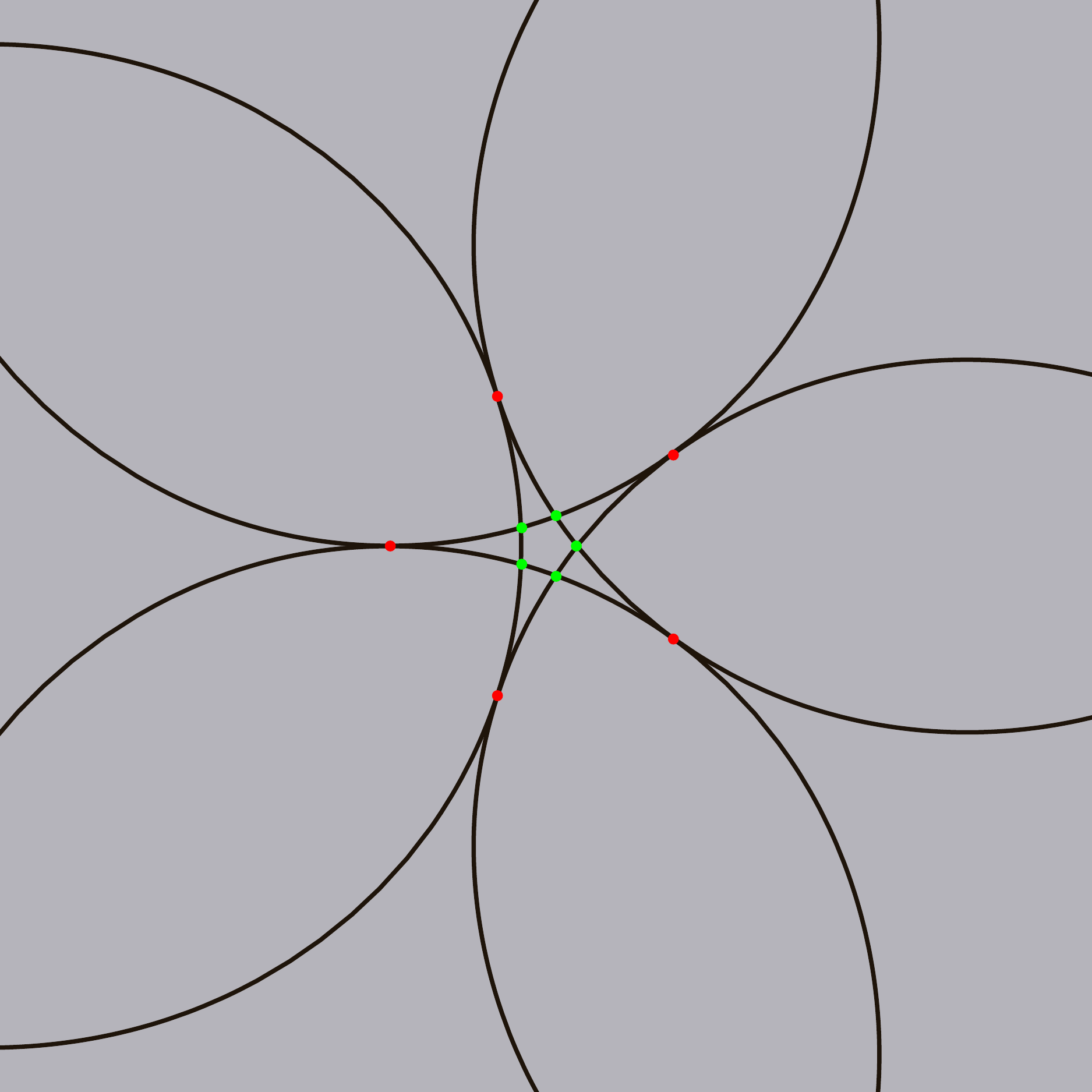}\quad\quad\quad\quad
		\includegraphics[height=4cm,width=4cm]{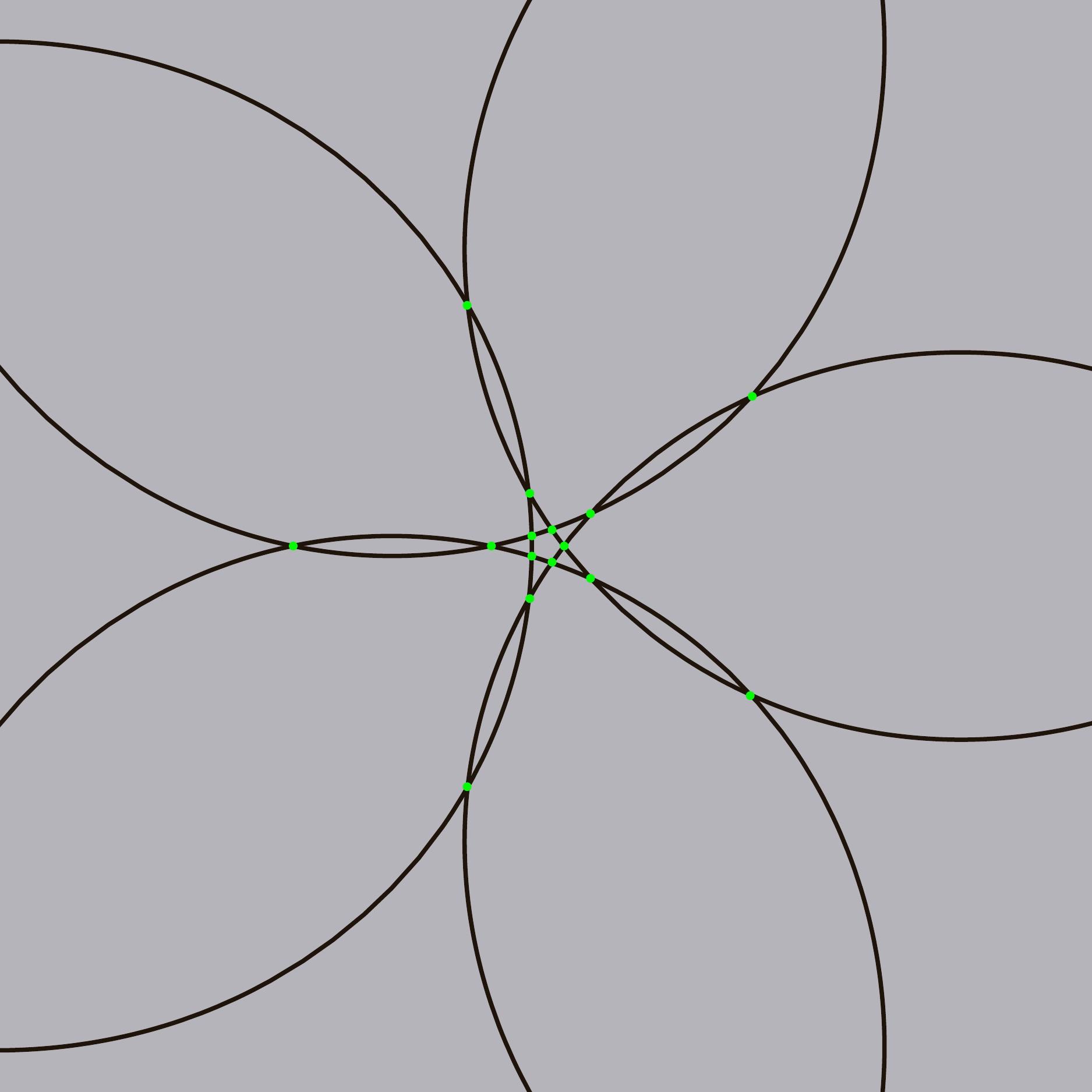}
		\caption{\small{Amplifications of the fourth and fifth images of Figure \ref{FIG23}, respectively.}}\label{FIG29}
	\end{center}
\end{figure}

\section{Appendix. Proof of Theorem \ref{existence}}

In this appendix, we will prove Theorem \ref{existence}. As mentioned in Section 4, the proof follows from the claim with respect to the limits of the period map $\mathcal{P}_\lambda$. In turn, the claim follows from five lemmas:

\begin{lemma} If $\lambda \in [-\sqrt[4]{\varphi^5}/2,-2/\sqrt[4]{27})$ where $\varphi=(1+\sqrt{5})/2$ is the golden ratio, the period map ${\mathcal P}_{\lambda}$ is real-analytic.
\end{lemma}
\begin{proof}
Let ${\mathtt g}$ and ${\mathtt m}$ be as in (\ref{ellipticmoduli}). Put
\begin{equation}\label{coeff}
	\begin{cases}
		(i)&{\mathtt n}_1=\frac{\left(1-2\sqrt{|c|}\,e_4\right)\left(e_2-e_1\right)}{\left(1-2\sqrt{|c|}\,e_1\right)\left(e_2-e_4\right)}\,,\\
		(ii)&{\mathtt n}_2=\frac{\left(1+2\sqrt{| c|}\,e_4\right)\left(e_2-e_1\right)}{\left(1+2\sqrt{|c|}\,e_1\right)\left(e_2-e_4\right)}\,,\\
		(iii)&{\mathtt A}=-\frac{{\mathtt g}e_4\left(e_4+2\lambda\right)}{1+4ce_4^2}\,\\
		(iv)&{\mathtt B}=-\frac{{\mathtt g}\left(1+4\sqrt{|c|}\,\lambda\right)\left(e_1-e_4\right)}{4\sqrt{|c|}\left(1-2\sqrt{|c|}\,e_4\right)\left(1-2\sqrt{|c|}\,e_1\right)}\,,\\
		(v)&{\mathtt C}=\frac{{\mathtt g}\left(1-4\sqrt{| c|}\,\lambda\right)\left(e_1-e_4\right)}{4\sqrt{| c|}\left(1+2\sqrt{|c|}\,e_4\right)\left(1+2\sqrt{|c|}\,e_1\right)}\,.
	\end{cases}
\end{equation}
If we fix $\lambda$, these are functions of $e_2\in {\rm I}_{\lambda}=(a(\lambda),\eta_+(\lambda))$ (see Remark \ref{4.2}),  denoted by ${\mathtt  g}_{\lambda}$, ${\mathtt m}_{\lambda}$, ${\mathtt n}_{1,\lambda}$, ${\mathtt n}_{2,\lambda}$, ${\mathtt A}_\lambda$, ${\mathtt B}_\lambda$ and ${\mathtt C}_\lambda$, respectively. In view of (\ref{eqexlc}), ${\mathtt B}_\lambda$ and ${\mathtt n}_{1,\lambda}$ are defined on ${\mathcal T}\setminus {\mathcal E}$.  The functions
${\mathtt g}_{\lambda},{\mathtt m}_{\lambda},{\mathtt n}_{2,\lambda},{\mathtt A}_{\lambda}$ and ${\mathtt C}_{\lambda}$ are real-analytic.  If $\lambda \in [-\sqrt[4]{\varphi^5}/2,-2/\sqrt[4]{27})$ also
${\mathtt n}_{1,\lambda}$ and ${\mathtt B}_{\lambda}$ are real-analytic.  Instead,  if $\lambda<-\sqrt[4]{\varphi^5}/2$,   ${\mathtt n}_{1,\lambda}$ and ${\mathtt B}_{\lambda}$ are real-analytic on ${\rm I}_{\lambda}\setminus \{c(\lambda)\}$ and
\[ \lim_{e_2\to c(\lambda)^{\pm}} {\mathtt n}_{1,\lambda}(e_2) = -\infty,\quad  \lim_{e_2\to c(\lambda)^{\pm}} {\mathtt B}_{\lambda}(e_2) =\pm \infty. \]
Recall that $c(\lambda)$ is defined in \eqref{c} and its graph represents the exceptional locus $\mathcal{E}$. 

If $\mathfrak{p}=(\lambda,e_2)\notin {\mathcal E}$, then
\[\begin{split}
	&- \frac{1}{2\pi} \Theta_{\mathfrak{p}}(\omega_{\mathfrak{p}})= 
	-\frac{\sqrt{|c|\,}}{\pi} \int_0^{\omega_{\mathfrak{p}}} \frac{\mu_\mathfrak{p}^2(\mu_\mathfrak{p}+2\lambda)}{1+4 c\mu_\mathfrak{p}^2}\,ds=\\
	& =-\frac{2\sqrt{|c|\,}}{\pi}\int_{e_2}^{e1}\frac{\mu\left(\mu+2\lambda\right)}{(1+4c\mu^2)\sqrt{-(\mu-e_1)(\mu-e_2)(\mu-e_3)(\mu-e_4)}}\,d\mu\,.
\end{split}
\]
The last integral, denoted by ${{\mathfrak  I}}_{\mathfrak{p}}$, can be written as
\begin{equation}\label{general}
	{\mathfrak I}_{\mathfrak{p}}=\frac{2\sqrt{|c|\,}}{\pi}\left({\mathtt A}{\rm K}({\mathtt m}) +{\mathtt B}\Pi({\mathtt n}_1,{\mathtt m})+{\mathtt C}\Pi({\mathtt n}_2,{\mathtt m}) \right).
\end{equation}
The map ${\mathfrak I}:\mathfrak{p}=(\lambda,e_2)\in {\mathcal T}\setminus {\mathcal E}\to {\mathfrak I}_{\mathfrak{p}}$ is real-analytic.  

Let $\mathfrak{p}=(\lambda,e_2)\in {\mathcal E}$ (ie. $e_2=c(\lambda)$),  then
\[\begin{split}
	& - \frac{1}{2 \pi} \Theta_{\mathfrak{p}}(\omega_{\mathfrak{p}})= \frac{4\sqrt{|c|\,}\lambda^2}{\pi}\int_0^{\omega_{\mathfrak{p}}}
	\frac{\mu_{\mathfrak{p}}^2}{\mu_{\mathfrak{p}}-2\lambda}\,ds=\\
	&= \frac{8\sqrt{|c|\,}\lambda^2}{\pi}\int_{e_2}^{e_1}\frac{\mu}{(\mu-2\lambda)\sqrt{-(\mu-e_1)(\mu-e_2)(\mu-e_3)(\mu-e_4)}}\,d\mu\,.
\end{split}\]
This last integral, denoted by $\widehat{\mathfrak I}_{\mathfrak{p}}$ can be written as
\begin{equation}\label{exceptional}
	\widehat{\mathfrak I}_{\mathfrak{p}}=\frac{2\sqrt{|c|\,}}{\pi}\left({\mathtt A}{\rm K}({\mathtt m}) +{\mathtt C}\Pi({\mathtt n}_2,{\mathtt m}) \right).
\end{equation}

If $\lambda \in [-\sqrt[4]{\varphi^5}/2,-2/\sqrt[4]{27})$, then ${\rm I}_{\lambda} \cap {\mathcal E}=\emptyset$ and ${\mathcal P}_{\lambda}={\mathfrak I}_{\mathfrak{p}}+1$. This proves that ${\mathcal P}_{\lambda}$ is real-analytic, for every $\lambda \in [-\sqrt[4]{\varphi^5}/2,-2/\sqrt[4]{27})$.
\end{proof}

\begin{lemma} If $\lambda < -\sqrt[4]{\varphi^5}/2$, the period map ${\mathcal P}_{\lambda}$ is continuous on ${\rm I}_{\lambda}=(a(\lambda),\eta_+(\lambda))$ and real-analytic on  ${\rm I}_{\lambda}\setminus \{c(\lambda)\}$.
\end{lemma}
\begin{proof}
Following with the notation of the previous lemma, for every $\lambda<-\sqrt[4]{\varphi^5}/2 $ we have ${\rm I}_{\lambda} \cap {\mathcal E}=\{(\lambda,c(\lambda))\}$ and
\[\begin{cases}{\mathcal P}_{\lambda}(e_2)={\mathfrak I}_{\mathfrak{p}}+1\,,\quad\quad\quad  e_2\in (a(\lambda),c(\lambda))\\
{\mathcal P}_{\lambda}(e_2)={\mathfrak I}_{\mathfrak{p}},\quad\quad\quad\quad\,\,\,\,\,   e_2\in (c(\lambda),\eta_+(\lambda))\end{cases}.\]

Then, ${\mathcal P}_{\lambda}$ is real-analytic on ${\rm I}_{\lambda}\setminus \{c(\lambda)\}$. The continuity part of the lemma  follows from the limits
\begin{equation}\label{lim1} 
	\lim_{e_2\to c(\lambda)^{-} }{\mathfrak I}_{\mathfrak{p}} = \widehat{\mathfrak I}_{\mathfrak{p}}-1/2\,,\quad\quad\quad 
	\lim_{e_2\to c(\lambda)^{+} }{\mathfrak I}_{\mathfrak{p}} = \widehat{\mathfrak I}_{\mathfrak{p}}+1/2\,.
\end{equation}
Or, equivalently, from
\begin{equation}\label{lim1bis} 
	\lim_{e_2\to c(\lambda)^{\pm} }
	\frac{2\sqrt{|c|\,}}{\pi} {\mathtt B}_{\lambda}\Pi({\mathtt n}_{1,\lambda},{\mathtt m}_{\lambda})=\pm\frac{1}{2}\,.
\end{equation}
The proof of (\ref{lim1bis}) is organized into two steps:\\
\noindent \emph{Step I}. In the first step we show that
\begin{equation}\label{lim2} 
	\lim_{e_2\to c(\lambda)^{\pm} } \frac{1+4\lambda\sqrt{|c|\,}}{\sqrt{1-2e_{\lambda,1}\sqrt{|c|\,}}}=\mp \frac{\sqrt{2}}{e_{\lambda,1}(c(\lambda))^2}\,,
\end{equation}
where $e_{\lambda,j}(e_2)=e_j(\lambda,e_2)$, $j=1,3,4$. Using $(v)$ and $(vi)$ of (\ref{rele1e2Lc}) we have
$$ \frac{1+4\lambda\sqrt{|c|\,}}{\sqrt{1-2e_{\lambda,1}\sqrt{|c|\,}}}=\frac{2e_{\lambda,1}^3e_2^2-{\mathtt q}_{\lambda}(e_2)\sqrt{1-{\mathtt p}_{\lambda}(e_2)^2}}
	{2e_{\lambda,1}^3e_2^2\sqrt{1-\sqrt{1-{\mathtt p}_{\lambda}(e_2)^2}}}\,,
	$$
	where
	$${\mathtt p}_{\lambda}(e_2)=\frac{e_{\lambda,1}^2e_2^3-e_{\lambda,1}^3e_2^2+e_{\lambda,1}+e_2}{2e_{\lambda,1}e_2^2}\,,\quad\quad\quad {\mathtt q}_{\lambda}(e_2)=(e_{\lambda,1}+e_2)(1+e_{\lambda,1}^2e_2^2)\,.
	$$
	By  $(iii)$ of Remark \ref{lemma5Modulispaces},  ${\mathtt p}_{\lambda}(e_2)$ tends to $0$ when $e_2\to c(\lambda)$, it is negative if $e_2<c(\lambda)$ and positive if $e_2>c(\lambda)$.  Then,
	\[\begin{split}& \lim_{e_2\to c(\lambda)^{\pm} }  \frac{1+4\lambda\sqrt{|c|\,}}{\sqrt{1-2e_{\lambda,1}\sqrt{|c|\,}}} =\lim_{e_2\to c(\lambda)^{\pm} } \frac{2e_{\lambda,1}^3e_2^2-{\mathtt q}_{\lambda}(e_2)\sqrt{1-{\mathtt p}_{\lambda}(e_2)^2}}
		{2e_{\lambda,1}^3e_2^2\sqrt{1-\sqrt{1-{\mathtt p}_{\lambda}(e_2)^2}}}  =\\
		&= \lim_{e_2\to c(\lambda)^{\pm} }-\frac{{\mathtt p}_{\lambda}(e_2)}{e_{\lambda,1}^2\sqrt{1-\sqrt{1-{\mathtt p}_{\lambda}(e_2)^2}}}=\mp \frac{\sqrt{2}}{e_{\lambda,1}(c(\lambda))^2}\,.
	\end{split}\]
	
\noindent \emph{Step II}. We will now prove (\ref{lim1bis}). For this purpose, recall the standard limit (see \cite{BF},  906.00,  p. 302)
	$$ \Pi(n,m)\sim \frac{\pi}{2\sqrt{1-n}}\,,\quad\,\,\, {\rm as}\,\,\,\  n\to -\infty\,, 	\quad\quad\,\,\, \forall m\in (0,1)\,.$$
	Since
	\[ \lim_{e_2\to c(\lambda)^{\pm} } {\mathtt n}_{1,\lambda}(e_2)=-\infty\,,\quad\quad 0<{\mathtt m}_{\lambda}(e_2)<1\,,\quad\quad \forall \lambda<-\sqrt[4]{\varphi^5}/2\,,
	\]
	we have
	\[\lim_{e_2\to c(\lambda)^{\pm} }
	\frac{2\sqrt{|c|\,}}{\pi} {\mathtt B}_{\lambda}\Pi({\mathtt n}_{1,\lambda},{\mathtt m}_{\lambda})=
	\lim_{e_2\to c(\lambda)^{\pm} }\frac{ \sqrt{|c|\,}{\mathtt B}_{\lambda}}{\sqrt{1-{\mathtt n}_{1,\lambda}}}\,.
	\]
	Using (\ref{coeff}) we obtain
	$$\frac{ \sqrt{|c|\,}{\mathtt B}_{\lambda}}{\sqrt{1-{\mathtt n}_{1,\lambda}}}=\frac{\widehat{\mathtt p}_{\lambda}}{\widehat{\mathtt q}_{\lambda}}\,,
	$$
	where $\widehat{\mathtt p}_{\lambda}$ is given by
	$$\widehat{\mathtt p}_\lambda=-\frac{1}{4}{\mathtt g}_{\lambda}(1+4\sqrt{|c|\,}\lambda)(e_{\lambda,1}-e_{\lambda,4})\sqrt{e_2-e_{\lambda,4}}(1-2\sqrt{|c|\,}e_{\lambda,4})^{-1}
	$$
	and $\widehat{\mathtt q}_{\lambda}$ is the square root of
	$$\widehat{\mathtt q}_\lambda^2=
	(1-2\sqrt{|c|\,}e_{\lambda,1})\left((1-2\sqrt{|c|\,}e_{\lambda,4})(e_{\lambda,1}-e_2)+(e_2-e_{\lambda,4})(1-2\sqrt{|c|\,}e_{\lambda,1})\right).
	$$
	Using (\ref{lim2}) we obtain
	$$\lim_{e_2\to c(\lambda)^{\pm} }\frac{ \sqrt{|c|\,}{\mathtt B}_{\lambda}}{\sqrt{1-{\mathtt n}_{\lambda,1}}} =
	\pm {\mathtt h}_{\lambda}(c(\lambda))\,,$$
	where
	$$
	{\mathtt h}_{\lambda}(e_2)=
	\frac{{\mathtt g}_{\lambda}(e_{\lambda,1}-e_{\lambda,4})\sqrt{e_2-e_{\lambda,4}}(1-2\sqrt{|c|\,}e_{\lambda,4})^{-1}}{2\sqrt{2}e_{\lambda,1}^2\sqrt{(e_{\lambda,1}-e_2)(1-2\sqrt{|c|\,}e_{\lambda,4})+(e_2-e_{\lambda,4})(1-2\sqrt{|c|\,}e_{\lambda,1})}}\,.
	$$
	From (\ref{ellipticmoduli}),  using $(iv)$, $(vi)$ of (\ref{rele1e2Lc}) and  $(iii)$ of (\ref{eqexlc}),  it follows that ${\mathtt h}_{\lambda}(c(\lambda))=1/2$, which proves the result.
\end{proof}

Before proceeding further, we list eight properties that will be used to prove the remaining lemmas:
\begin{enumerate}
	\item For every $\lambda<-2/\sqrt[4]{27}$, $0<{\mathtt m}_{\lambda}<1$, and
	\[\begin{split}&\lim_{e_2\to a(\lambda)^+ } {\mathtt m}_{\lambda}(e_2)=1\,,\quad\quad\quad \forall \lambda\in [-1,-2/\sqrt[4]{27})\,,\\
		&\lim_{e_2\to a(\lambda)^+ } {\mathtt m}_{\lambda}(e_2)<1\,,\quad\quad\quad \forall \lambda\in (-\infty,-1)\,.\end{split}
	\]
	Recall that ${\mathtt m}_\lambda$ is the notation of ${\mathtt m}$ \eqref{ellipticmoduli} for fixed $\lambda$.
	\item The functions ${\mathtt n}_{j,\lambda}(e_2)$ are negative when $e_2\to a(\lambda)^+$ and $\lim_{e_2\to a(\lambda)^+ } {\mathtt n}_{2,\lambda}(e_2)$ is finite, 
	for every $\lambda<-2/\sqrt[4]{27}$.  The limit $\lim_{e_2\to a(\lambda)^+ } {\mathtt n}_{1,\lambda}(e_2)$ is finite,  for every $\lambda<-2/\sqrt[4]{27}$,
	$\lambda\neq -\sqrt[4]{\varphi^5}/2 $ and is $-\infty$ if  $\lambda= -\sqrt[4]{\varphi^5}/2 $.  In addition, if $\lambda\le -1$
	$$\lim_{e_2\to a(\lambda)^+ } {\mathtt n}_{1,\lambda}(e_2)=\lim_{e_2\to a(\lambda)^+ } {\mathtt n}_{2,\lambda}(e_2)\,.$$
	\item The limit $\lim_{e_2\to a(\lambda)^+ } {\mathtt B}_{\lambda}(e_2)$ is finite, for every $\lambda\in (-1,-2/\sqrt[4]{27})$ and  $\lambda\neq -\sqrt[4]{\varphi^5}/2$, while it is $-\infty$ for every $\lambda\le -1$.
	\item The limit $\lim_{e_2\to a(\lambda)^+ } {\mathtt C}_{\lambda}(e_2)$ is finite, for every $\lambda\in (-1,-2/\sqrt[4]{27})$, and it is $+\infty$ for every $\lambda\le -1$.
	\item The limit $\lim_{e_2\to a(\lambda)^+ } c$ ($c=c(\lambda,e_2)$) is finite, for every $\lambda<-2/\sqrt[4]{27}$.  It is negative if $\lambda> -1$ and zero if $\lambda\le -1$.
	\item The limit $\lim_{e_2\to a(\lambda)^+ } c\,{\mathtt C}_{\lambda}(e_2)$ is finite for every $\lambda<-2/\sqrt[4]{27}$.  It is $0$ if $\lambda\le -1$ and it is negative if $\lambda>-1$.
	\item The limit $\lim_{e_2\to a(\lambda)^+ } {\mathtt A}_{\lambda}(e_2)$ is finite for every $\lambda<-2/\sqrt[4]{27}$.
	\item Let ${\mathtt Q}_{\lambda}$ be defined by
	\begin{equation}\label{Qfunction}
		{\mathtt Q}_{\lambda}={\mathtt A}_{\lambda}+\frac{{\mathtt B}_{\lambda}}{1-{\mathtt n}_{1,\lambda}}+\frac{{\mathtt C}_{\lambda}}{1-{\mathtt n}_{2,\lambda}}\,.
	\end{equation}
	Using $(iii)-(vi)$ of (\ref{rele1e2Lc}) and taking into account (\ref{ellipticmoduli}) and (\ref{coeff}) we have
	$${\mathtt Q}_{\lambda}=-\frac{{\mathtt  g}_{\lambda}e_2(e_2+2\lambda)}{1+4ce_2^2}\,.$$
	Hence, ${\mathtt Q}_{\lambda}$ is real-analytic, it is positive and tends to a finite limit when $e_2\to a(\lambda)^+$ and when $e_2\to \eta_+(\lambda)^-$ for every $\lambda <-2/\sqrt[4]{27}$.
\end{enumerate}

We now continue with the proofs of the remaining lemmas, which will show the limits of the period map $\mathcal{P}_\lambda$.

\begin{lemma}
Let $\lambda\in (-1,-2/\sqrt[4]{27})$. Then,
\[ \lim_{e_2\to a(\lambda)^+} {\mathcal P}_{\lambda}(e_2) = +\infty\,. \]
\end{lemma}
\begin{proof} We begin by recalling some asymptotic behavior of elliptic integrals. From 906.03 at p. 302 of \cite{BF} we have the asymptotic expansion
	\begin{equation}\label{asymp1}
		\Pi(n,m)\sim \frac{1}{1-n}\left(\log(\frac{4}{\sqrt{1-m}})+\sqrt{-n}\arctan(\sqrt{-n})\right),\quad\quad {\rm as}\,\, m\to 1^-\,,
	\end{equation}
while from 112.01 at p. 11 of \cite{BF}	we have the standard asymptotic expansion 
	\begin{equation}\label{asymp2}
		{\rm K}(m)\sim \log\left(\frac{4}{\sqrt{1-m}} \right),\quad\quad {\rm as}\,\,\, m\to 1^-\,.
	\end{equation}
Then, from Property (1), (\ref{general}), (\ref{asymp1}) and (\ref{asymp2}) we have
	\begin{equation}\label{asymp3}
		{\mathfrak  I}_{\mathfrak{p}}\sim \frac{\sqrt{|c|\,}}{\pi}\left({\mathtt R}_{\lambda} + \log\left( \frac{4}{\sqrt{1-{\mathtt m}_{\lambda}}}\right){\mathtt Q}_{\lambda}\right),\quad\quad {\rm as}\,\,\,\ e_2\to a(\lambda)^+\,,
	\end{equation}
	where
	\begin{equation}\label{R}
		{\mathtt R}_{\lambda}=\frac{\sqrt{-{\mathtt n}_{1,\lambda}}\arctan(\sqrt{-{\mathtt n}_{1,\lambda}})}{1-{\mathtt n}_{1,\lambda}}{\mathtt B}_{\lambda}+\frac{\sqrt{-{\mathtt n}_{2,\lambda}}\arctan(\sqrt{-{\mathtt n}_{2,\lambda}})}{1-{\mathtt n}_{2,\lambda}}{\mathtt C}_{\lambda}\,.
	\end{equation}
	Next, using the Properties (2)-(6), we conclude that the limits 
	\[\lim_{e_2\to a(\lambda)^+ } \sqrt{|c|\,}{\mathtt R}_{\lambda}(e_2)\,,\quad\quad\quad \lim_{e_2\to a(\lambda)^+ } \sqrt{|c|\,}{\mathtt Q}_{\lambda}\]
	are finite, for every $\lambda\in (-1,-2/\sqrt[4]{27})$ and  $\lambda\neq -\sqrt[4]{\varphi^5}/2 $. It then follows from Property (1) and (\ref{asymp3}) that
	\[\lim_{e_2\to a(\lambda)^{+} } {\mathcal P}_{\lambda}(e_2)= \lim_{e_2\to a(\lambda)^{+} } {\mathfrak  I}_{\mathfrak{p}}=+\infty\,,\]
	for every $\lambda\in (-1,-2/\sqrt[4]{27})$ and  $\lambda \neq -\sqrt[4]{\varphi^5}/2$.  
	
	Consider now $\lambda =  -\sqrt[4]{\varphi^5}/2$.  Then $a(\lambda)=\varphi^{1/4}$ and the limits
	$$\lim_{e_2\to a(\lambda)^+ } \sqrt{|c|\,}{\mathtt Q}_{\lambda}\,,\quad\quad\quad
	\lim_{e_2\to a(\lambda) } \frac{\sqrt{-{\mathtt n}_{2,\lambda}}\arctan(\sqrt{-{\mathtt n}_{2,\lambda}})}{1-{\mathtt n}_{2,\lambda}}{\mathtt C}_{\lambda}
	$$
	exist and are positive. The functions ${\mathtt B}_{\lambda}$ and $\sqrt{-{\mathtt n}_{1,\lambda}}\arctan(\sqrt{-{\mathtt n}_{1,\lambda}})(1-{\mathtt n}_{1,\lambda})^{-1}$
	are positive on $(a(\lambda),  a(\lambda)+\epsilon)$,  for some $\epsilon >0$.  
	Consequently,  $ \sqrt{|c|\,}{\mathtt R}_{\lambda}$ is positive on $(a(\lambda),  a(\lambda)+\epsilon)$. Thus,  also in this case
	\[\lim_{e_2\to a(\lambda)^{+} } {\mathcal P}_{\lambda}(e_2)= \lim_{e_2\to a(\lambda)^{+} } {\mathfrak  I}_{\mathfrak{p}}=+\infty\,,\]
	as claimed.
\end{proof}

\begin{lemma} Let $\lambda\in (-\infty,-1]$. Then, 
	\[ \lim_{e_2\to a(\lambda)^+} {\mathcal P}_{\lambda}(e_2) = 1\,. \]
\end{lemma}
\begin{proof} From $(v)$ and $(vi)$ of (\ref{coeff}) we have
	\begin{equation}\label{B+C}{\mathtt B}_{\lambda}+{\mathtt C}_{\lambda}=\frac{{\mathtt g}_{\lambda}(e_{\lambda,1}-e_{\lambda,4})(e_{\lambda,4}(8c\lambda e_{\lambda,1}-1)-e_{\lambda,1}-2\lambda)}{(1+4ce_{\lambda,1}^2)(1+4ce_{\lambda,4}^2)}\,.
	\end{equation}
	Using  (\ref{rele1e2Lc}) and the characterization of Remark \ref{lemma4Modulispaces} we deduce that
	$$ \lim_{e_2\to a(\lambda)^{+} } e_{\lambda,4}(e_2)+e_{\lambda,1}(e_2)+2\lambda=0\,,\quad\quad\quad \forall \lambda\le -1\,.$$
	This implies
	\begin{equation}\label{ort1}\lim_{e_2\to a(\lambda)^{+} } ({\mathtt B}_{\lambda}(e_2)+{\mathtt C}_{\lambda}(e_2))=0\,,\quad\quad\quad \forall \lambda\le -1\,.
	\end{equation}
	Assume first $\lambda = -1$.   From (\ref{ort1}) we have
	$${\mathtt R}_{\lambda}\sim \left(\frac{\sqrt{-{\mathtt n}_{2,\lambda}}\arctan(\sqrt{-{\mathtt n}_{2,\lambda}})}{1-{\mathtt n}_{2,\lambda}} -
	\frac{\sqrt{-{\mathtt n}_{1,\lambda}}\arctan(\sqrt{-{\mathtt n}_{1,\lambda}})}{1-{\mathtt n}_{1,\lambda}}
	\right){\mathtt C}_{\lambda},
	$$ 
	as $e_2\to a(-1)^+$, $a(-1)=1$.  
	From $(v)$ of  (\ref{coeff}) it follows that  $\lim_{e_2\to 1^+ } \sqrt{|c|\,}{\mathtt C}_{-1}$ is finite.  Property (2) implies 
	$$ \lim_{e_2\to a(\lambda)^{+} } \left(\frac{\sqrt{-{\mathtt n}_{2,\lambda}}\arctan(\sqrt{-{\mathtt n}_{2,\lambda}})}{1-{\mathtt n}_{2,\lambda}} -
	\frac{\sqrt{-{\mathtt n}_{1,\lambda}}\arctan(\sqrt{-{\mathtt n}_{1,\lambda}})}{1-{\mathtt n}_{1,\lambda}}
	\right)=0\,,
	$$
	for every $\lambda\le -1$. Then, 
	\begin{equation}\label{first}\lim_{e_2\to 1^+ } \sqrt{|c|\,}{\mathtt R}_{-1}=0\,.
	\end{equation}
	On the other hand\footnote{This limit has been computed with the software Mathematica 13.1.}
	\begin{equation}\label{second}
		\lim_{e_2\to 1^+ } \frac{c}{1-{\mathtt m}_{-1}(e_2)}=0\,.
	\end{equation}
	Combining  (\ref{asymp3}), (\ref{first}) and (\ref{second}) we obtain
	$$ \lim_{e_2\to 1^+ } {\mathcal P}_{-1}(e_2)=\lim_{e_2\to 1^+ }{{\mathfrak  I}}_{\mathfrak{p}}+1=1\,.	$$
	Consider now the case $\lambda<-1$. Using Property (2) we have
	\begin{equation}\label{ort2}
		\lim_{e_2\to a(\lambda)^+ }{\mathtt n}_{1,\lambda}(e_2)=\lim_{e_2\to a(\lambda)^+ }{\mathtt n}_{2,\lambda}(e_2)<0\,,\quad\quad\quad\lim_{e_2\to a(\lambda)^+ }{\mathtt m}_{\lambda}(e_2)\in (0,1)\,.
	\end{equation}
	Then
	$$\lim_{e_2\to a(\lambda)^+ }\Pi\left({\mathtt n}_{1,\lambda}(e_2),{\mathtt m}_{\lambda}(e_2)\right)=
	\lim_{e_2\to a(\lambda)^+ }\Pi\left({\mathtt n}_{2,\lambda}(e_2),{\mathtt m}_{\lambda}(e_2)\right).$$
	Moreover, (\ref{ort2}) implies that the limits
	$$\lim_{e_2\to a(\lambda)^+ }\Pi({\mathtt n}_{1,\lambda}(e_2),{\mathtt m}_{\lambda}(e_2))\,,\quad\quad\quad \lim_{e_2\to a(\lambda)^+ }{\rm K}({\mathtt m}_{\lambda}(e_2))
	$$
	are finite for every $\lambda<-1$.   From $(v)$ of  (\ref{coeff}) it follows that  $\lim_{e_2\to a(\lambda)^+ } \sqrt{|c|\,}{\mathtt C}_{\lambda}$ is finite for every $\lambda<-1$. Properties (5) and (7) implies that  $\lim_{e_2\to a(\lambda)^+ } {\mathtt A}_{\lambda}$ is also finite, and that $\lim_{e_2\to a(\lambda)^+ } \sqrt{|c|\,}=0$ for every $\lambda<-1$.  Then,  using (\ref{ort1}), we have
	\[\begin{split}{\mathfrak I}_{\mathfrak{p}}&=\frac{2\sqrt{|c|\,}}{\pi}\left({\mathtt A}_{\lambda}{\rm K}({\mathtt m}_{\lambda}) +{\mathtt B}_{\lambda}\Pi({\mathtt n}_{1,\lambda},{\mathtt m}_{\lambda})+{\mathtt C}_{\lambda}\Pi({\mathtt n}_{2,\lambda},{\mathtt m}_{\lambda}) \right)\\
		&\sim_{e_2\to a(\lambda)^+}  \frac{2\sqrt{|c|\,}}{\pi}\left(\Pi({\mathtt n}_{2,\lambda},{\mathtt m}_{\lambda}) - \Pi({\mathtt n}_{1,\lambda},{\mathtt m}_{\lambda}) \right){\mathtt C}_{\lambda}\,,
	\end{split}\]
	for every $\lambda<-1$. This implies
	$$ \lim_{e_2\to a(\lambda)^+} {\mathcal P}_{\lambda}(e_2)=\lim_{e_2\to 1^+ }{{\mathfrak  I}}_{\mathfrak{p}}+1=1\,,\quad\quad\quad \forall \lambda<-1\,,$$
	concluding the proof.
\end{proof}

\begin{lemma} Let $\lambda<-2/\sqrt[4]{27}$. Then, 
	\[ \lim_{e_2\to \eta_+(\lambda)^-} {\mathcal P}_{\lambda}(e_2) = \chi(\lambda)\,, \]
	where $chi(\lambda)$ is the function defined in Definition \ref{4.3}.
\end{lemma}
\begin{proof} From
\[\lim_{e_2\to \eta_+(\lambda)^{-} } {\mathtt m}_{\lambda}(e_2)= \lim_{e_2\to \eta_+(\lambda)^{-} } {\mathtt n}_{1,\lambda}(e_2)= 
\lim_{e_2\to \eta_+(\lambda)^{-} } {\mathtt n}_{2,\lambda}(e_2)=0\,,\]
it follows that
\[\begin{split}&\lim_{e_2\to \eta_+(\lambda)^{-} } {\rm K}({\mathtt m}_{\lambda}(e_2))=\lim_{e_2\to \eta_+(\lambda)^{-} } \Pi({\mathtt n_{1,\lambda}}(e_2),{\mathtt m}_{\lambda}(e_2))=\\&= \lim_{e_2\to \eta_+(\lambda)^{-} } \Pi({\mathtt n_{2,\lambda}}(e_2),{\mathtt m}_{\lambda}(e_2))=\frac{\pi}{2}\,.\end{split}\]
Then,
\begin{equation}\label{lim5}\lim_{e_2\to \eta_+(\lambda)^{-} } {\mathcal P}_{\lambda}=\lim_{e_2\to \eta_+(\lambda)^{-} }{{\mathfrak  I}}_{\mathfrak{p}}=
	\lim_{e_2\to \eta_+(\lambda)^{-} }
	\sqrt{|c|\,} ({\mathtt A}_{\lambda} +{\mathtt B}_{\lambda}+{\mathtt C}_{\lambda})\,.
\end{equation}
From (\ref{rele1e2Lc}) and (\ref{coeff}) we obtain the following limits:
\[\begin{split}
	&(i) \lim_{e_2\to \eta_+(\lambda)^{-} } c=\frac{\eta_+(\lambda)^4-1}{\eta_+(\lambda)^6}\,,\\
	&(ii) \lim_{e_2\to \eta_+(\lambda)^{-} } {\mathtt A}_{\lambda}=\frac{2\eta_+(\lambda)^3(1+\sqrt{1+\eta_+(\lambda)^4}(1-\eta_+(\lambda)^4))}{\sqrt{\eta_+(\lambda)^4-3}(\eta_+(\lambda)^8-\eta_+(\lambda)^4-1)},\\
	&(iii) \lim_{e_2\to \eta_+(\lambda)^{-} } {\mathtt B}_{\lambda}=\frac{\eta_+(\lambda)^3(\sqrt{\eta_+(\lambda)^4-1}(1+\eta_+(\lambda)^4)-\eta_+(\lambda)^6)(2+\sqrt{1+\eta_+(\lambda)^4})}
	{\sqrt{3+\eta_+(\lambda)^8-4\eta_+(\lambda)^4}(1+\eta_+(\lambda)^4\sqrt{1+\eta_+(\lambda)^4}-\eta_+(\lambda)^2\sqrt{\eta_+(\lambda)^8-1}},\\
	& (iv)\lim_{e_2\to \eta_+(\lambda)^{-} } {\mathtt C}_{\lambda}=\frac{\eta_+(\lambda)^3(\sqrt{\eta_+(\lambda)^4-1}(1+\eta_+(\lambda)^4)+\eta_+(\lambda)^6)(2+\sqrt{1+\eta_+(\lambda)^4})}
	{\sqrt{3+\eta_+(\lambda)^8-4\eta_+(\lambda)^4}(1+\eta_+(\lambda)^4\sqrt{1+\eta_+(\lambda)^4}-\eta_+(\lambda)^2\sqrt{\eta_+(\lambda)^8-1}}
\end{split}
\]
Substituting $(i)$-$(iv)$ in (\ref{lim5}) we conclude
\[ \lim_{e_2\to \eta_+(\lambda)^-} {\mathcal P}_{\lambda}(e_2) = \chi(\lambda)\,,\]
as stated.
\end{proof}

Finally, the proof of Theorem \ref{existence} follows from the limits shown in previous lemmas in combination with the continuity of the period map $\mathcal{P}_\lambda$.

\bibliographystyle{amsalpha}

\end{document}